\documentclass[a4paper,bibliography=totoc,index=totoc]{scrbook}

\usepackage[ngerman,UKenglish]{babel}
\usepackage[T1]{fontenc}
\usepackage[utf8]{inputenc}
\usepackage[autostyle=true]{csquotes}
\usepackage[german]{isodate}

\PassOptionsToPackage{hidelinks}{hyperref}

\usepackage{lmodern}

\usepackage{amsmath}
\usepackage{amssymb}
\usepackage{amsthm}

\numberwithin{equation}{chapter}
\numberwithin{figure}{chapter}
\numberwithin{table}{chapter}
\newtheorem{definition}{Definition}[chapter]

\newtheorem{proposition}[definition]{Proposition}
\newtheorem{theorem}[definition]{Theorem}

\newtheorem{example}[definition]{Example}
\newtheorem{remark}[definition]{Remark}

\DeclareMathOperator{\sgn}{sgn}
\DeclareMathOperator{\blur}{bl}

\DeclareMathOperator{\nan}{NaN}
\DeclareMathOperator{\rd}{rd}
\DeclareMathOperator{\rde}{rd_{\mathcal{E}}}
\DeclareMathOperator{\asc}{asc}
\DeclareMathOperator{\LN}{LN}

\newcommand{\Z}{\mathbb{Z}}

\newcommand{\R}{\mathbb{R}}
\newcommand{\M}{\mathbb{M}}
\newcommand{\N}{\mathbb{N}}
\newcommand{\I}{\mathbb{I}}
\newcommand{\F}{\mathbb{F}}
\newcommand{\U}{\mathbb{U}}
\newcommand{\power}{\mathcal{P}}
\newcommand{\iffy}{\breve{\infty}}
\newcommand{\ul}[1]{\underline{#1}}
\newcommand{\ol}[1]{\overline{#1}}

\makeatletter
\newcommand{\dotminus}{\mathbin{\text{\@dotminus}}}
\newcommand{\@dotminus}{%
  \ooalign{\hidewidth\raise1ex\hbox{.}\hidewidth\cr$\m@th-$\cr}%
}
\makeatother

\usepackage{accsupp}
\newcommand{\noncopynumber}[1]{
    \BeginAccSupp{method=escape,ActualText={}}
    #1
    \EndAccSupp{}
}

\usepackage{cases}

\usepackage{tikz}
\usepackage{graphicx}
\usepackage{pgfplots}

\usepackage{enumerate}
\usepackage{url}
\usepackage{here}
\usepackage{longtable}
\usepackage{listingsutf8}

\begin{document}
\selectlanguage{ngerman}
\frontmatter
\begin{titlepage}
	\raggedright
	\null
	\includegraphics[width=0.25\textwidth]{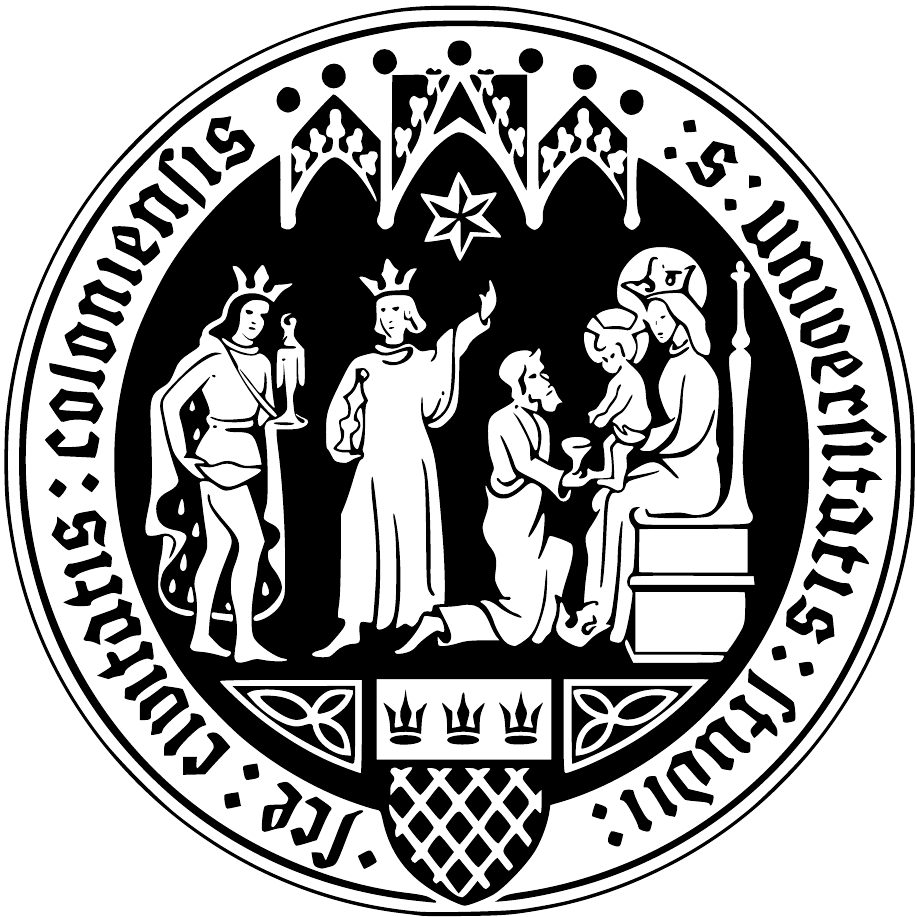}\par\vspace{1cm}
	{\scshape\LARGE Universität zu Köln\par}\vspace{0.3cm}
	{\Large Mathematisches Institut\\}\par\vspace{1.5cm}
	{\scshape\Large Bachelorarbeit\par}
	\vspace{1.5cm}
	{\huge\bfseries The Unum Number Format: Mathematical Foundations,
	Implementation and Comparison to IEEE~754 Floating-Point Numbers\par}
	\vspace{1.5cm}
	{\Large\itshape Laslo Hunhold\par}
	\vfill
	{\em Erstgutachterin:}\par
	Prof.~Dr.~Angela \textsc{Kunoth}\\[3mm]
	{\em Zweitgutachter:}\par
	Samuel \textsc{Leweke}
	\vfill
	{\large\printdate{2016-11-08}}
\end{titlepage}
\selectlanguage{UKenglish}
\selectlanguage{UKenglish}
\tableofcontents
\mainmatter
\chapter{Introduction}
This thesis examines a modern concept for machine numbers based on interval
arithmetic called \enquote{Unums} and compares it to IEEE~754
floating-point arithmetic, evaluating possible uses of this format where
floating-point numbers are inadequate.
In the course of this examination, this thesis builds theoretical
foundations for IEEE~754 floating-point numbers, interval arithmetic
based on the projectively extended real numbers and Unums.
\paragraph{Machine Number Concepts}
Following the invention of machine floating-point numbers by
Leonardo \textsc{Torres y Quevedo} in 1913 (see \cite[Section~3]{ra82})
the format has evolved to be the standard used in numerical computations
today. Over time though, different concepts and new approaches for
representing numbers in a machine have emerged. One of these new concepts
is the \emph{infinity computer} developed by
Yaroslav D. \textsc{Sergeyev}, introducing \emph{grossone} arithmetic for
superfinite calculations. See \cite{se15} for further reading.
\par
Another concept are the \emph{universal numbers} (\enquote{Unums}) proposed by
John L. \textsc{Gustafson}. They were first introduced as a
variable-length floating-point format with an uncertainty-bit as a
superset of IEEE~754 floating-point numbers called \enquote{Unums 1.0}
(see \cite{gu15}). Reasoning about the complexity of machine
implementations for and decimal calculation trade-offs with IEEE~754
floating-point numbers (see \cite[Section~2]{gu16-b}), \textsc{Gustafson}
presented a new format aiming to be easy to implement in the machine and
provide a simple way to do decimal calculations with guaranteed error
bounds (see \cite{gu16-a} and \cite{gu16-b}). He called the new format
\enquote{Unums 2.0}. In the course of this thesis, we are referring to
\enquote{Unums 2.0} when talking about Unums.
\paragraph{Projectively Extended Real Numbers}
Besides the well-known and established concept of extending the real
numbers with signed infinities $+\infty$ and $-\infty$, called the
\emph{affinely extended real numbers}, a different
approach is to only use one unsigned symbol for infinity, denoted as
$\iffy$ in this thesis. This extension is called the
\emph{projectively extended real numbers} and
we will prove that it is well-defined in terms of finite and infinite
limits. It is in our interest to examine how much we lose and what we
gain with this reduction, especially in regard to interval arithmetic.
\paragraph{Interval Arithmetic}
The concept behind interval arithmetic is to model quantities bounded by
two values, thus in general being subsets rather than elements of the real
numbers. Despite the fact that interval arithmetic in the machine can give
definite bounds for a result, it is easy to find examples where it gives
overly pessimistic results, for instance the dependency problem.
\par
This thesis will present a theory of interval arithmetic based on the
projectively extended real numbers, picking up the idea of modelling
degenerate intervals across the infinity point as well, allowing division
by zero and showing many other useful properties.
\paragraph{Goal of this Thesis}
The goal of this thesis is to evaluate the Unum number format in a
theoretical and practical context, make out advantages and see how
reasonable it is to use Unums rather than the ubiquitous IEEE~754
floating-point format for certain tasks.
\par
At the time of writing, all available implementations of the Unum
arithmetic are using floating-point arithmetic at runtime instead of
solely relying on lookup tables as \textsc{Gustafson} proposes. The
provided toolbox developed in the course of this thesis limits
the use of floating-point arithmetic at runtime to the initialisation of
input data. Thus it is a special point of interest to evaluate the format
the way it was proposed and not in an artificial floating-point
environment created by the currently available implementations.
\paragraph{Structure of this Thesis}
Following Chapter~\ref{ch:ieee}, which provides a formalisation of IEEE
754 floating-point numbers from the ground up solely based on the standard,
deriving characteristics of the set of floating-point numbers and using
numerical examples that show weaknesses of the format, Section~\ref{sec:pern}
introduces the projectively extended real numbers and proves
well-definedness of this extension after introducing a new concept of
finite and infinite limits on it.
Based on this foundation, Sections \ref{sec:oi} and \ref{sec:flakes}
construct a theory of interval arithmetic on top of the projectively
extended real numbers, formalizing intuitive concepts of interval
arithmetic.
\par
Using the results obtained in Chapter~\ref{ch:ia},
Chapter~\ref{ch:ua} embeds the Unums 2.0 number format proposed by
John L. \textsc{Gustafson} (see \cite{gu16-a} and \cite{gu16-b}) within
this interval arithmetic, evaluating it both from a theoretical and
practical perspective, providing a Unum 2.0 toolbox that was developed in
the course of this thesis and giving numerical examples implemented in
this toolbox.
\chapter{IEEE 754 Floating-Point Arithmetic}\label{ch:ieee}
Floating-point numbers have gone a long way since Konrad \textsc{Zuse}'s
Z1 and Z3, which were among the first machines to implement floating-point
numbers, back then obviously using a non-standardised format
(see \cite[pp.~31, 40--48]{ro98}). With more and more computers seeing the
light of day in the decades following the pioneering days, the demand for
a binary floating-point standard rose in the face of many different
proprietary floating-point formats.
\par
The Institute of Electrical and Electronics
Engineers (IEEE) took on the task and formulated the \enquote{ANSI\slash
IEEE~754-1985, Standard for Binary Floating-Point Arithmetic}
(see \cite{ieee85}), published and adopted internationally in 1985 and
revised in 2008 (see \cite{ieee08}) with a few extensions, including
decimal floating-point numbers (see \cite[Section~3.5]{ieee08}), which
are not going to be presented here. This standardisation effort led to a
homogenisation of floating-point formats across computer manufacturers,
and this chapter will only deal with this standardised format and follow
the concepts presented in the IEEE 754-2008 standard. All results in
this chapter are solely derived from this standard.
\section{Number Model}
The idea behind floating-point numbers rests on the observation that given
a base $b\in\N$ with $b\ge 2$ any $x\in\R$ can be represented by
\begin{equation*}
	\exists
	(s,e,d) \in \{ 0,1 \} \times \Z \times
	{\{ 0,\ldots, b-1\}}^{\N_0}: x = {(-1)}^s \cdot b^e \cdot
	\sum_{i=0}^{\infty} d_i \cdot b^{-i}.
\end{equation*}
There exist multiple parametres $(s,e,d)$ for a single
$x$. For instance, $x=6$ in the base $b=10$ yields $(0,0,\{6,0,\ldots\})$
and $(0,1,\{0,6,0,\ldots\})$ as two of many possible parametrisations.
\par
Given the finite nature of the computer, the number of possible exponents
$e$ and digits $d_i$ is limited. Within these bounds we can model a
machine number $\tilde{x}$ with exponent bounds $\ul{e},\ol{e} \in \Z$,
$\ul{e} \le e \le \ol{e}$ and a fixed number of digits $n_m \in \N$ and
base $b=2$ as
\begin{equation*}
	\tilde{x} =
	{(-1)}^s \cdot 2^e \cdot \sum_{i=0}^{n_m} d_i \cdot 2^{-i}.
\end{equation*}
Given binary is the native base the computer works with, we will
assume $b=2$ in this chapter. Despite being able to model finite
floating-point numbers in the machine now, we still have problems with
the lack of uniqueness. The IEEE~754 standard solves this by reminding
that the only difference between those multiple parametrisations for a
given machine number $\tilde{x}\neq 0$ is that
\begin{equation*}
	\min\left\{ i \in \{ 0,\ldots,n_m \} \, | \,
	d_i = 0 \right\}
\end{equation*}
is variable (see \cite[Section~3.4]{ieee08}). This means that we have a
varying amount of $0$'s in the sequence $\{d_i\}_{i\in\{0,\ldots,n_m\}}$
until we reach the first $1$.
One way to work around this redundancy is to use \emph{normal} floating
point numbers, which force $d_0=1$ (see \cite[Section~3.4]{ieee08}). The
$d_0$ is not stored as it has always the same value. This results
in the
\begin{definition}[set of normal floating-point numbers]\label{def:sonfpn}
	Let $n_m \in \N$ and $\ul{e},\ol{e} \in \Z$. The set of
	normal floating-point numbers is defined as
	\begin{equation*}
		\M_1(n_m,\ul{e},\ol{e}) := \left\{
		{(-1)}^s \cdot 2^e \cdot \left( 1 + \sum_{i=1}^{n_m} d_i
		\cdot 2^{-i} \right) \, \Bigg| \, s\in\{0,1\} \wedge
		\ul{e} \le e \le \ol{e} \wedge
		d \in {\{ 0,1 \}}^{n_m} \right\}.
	\end{equation*}
\end{definition}
In addition to normal floating-point numbers, we can also define
\emph{subnormal} floating-point numbers, also known as \emph{denormal}
floating-point numbers, which force $d_0=0$ and $e=\ul{e}$
and are smaller in magnitude than the smallest (positive) normal
floating-point number (see \cite[Section~3.4d]{ieee08}).
\begin{definition}[set of subnormal floating-point numbers]\label{def:sosfpn}
	Let $n_m \in \N$ and $\ul{e}\in \Z$. The set of subnormal
	floating-point numbers is defined as
	\begin{equation*}
		\M_0(n_m,\ul{e}) := \left\{
		{(-1)}^s \cdot 2^{\ul{e}} \cdot \left( 0 +
		\sum_{i=1}^{n_m} d_i \cdot 2^{-i}
		\right) \, \Bigg| \, s\in\{0,1\} \wedge
		d \in {\{ 0,1 \}}^{n_m} \right\}.
	\end{equation*}
\end{definition}
The subnormal floating-point numbers allow us to express $0$ with
$d=0$ and fill the so called \enquote{underflow gap}
between the smallest normal floating-point number and $0$.
With $d$ and $s$ variable, we use boundary values of the
exponent to fit subnormal, normal and exception cases under one roof
(see \cite[Section~3.4a-e]{ieee08}).
\begin{definition}[set of floating-point numbers]\label{def:sofpn}
	Let $n_m \in \N$, $\ul{e},\ol{e} \in \Z$ and
	$d \in {\{ 0,1 \}}^{n_m}$. The set of floating
	point numbers is defined as
	\begin{equation*}
		\M(n_m,\ul{e}-1,\ol{e}+1) \ni: \tilde{x}(s,e,d)
		\begin{cases}
			\in \M_0(n_m,\ul{e}) & e = \ul{e}-1 \\
			\in \M_1(n_m,\ul{e},\ol{e}) & \ul{e} \le e \le \ol{e} \\
			= (-1)^s \cdot \infty & e = \ol{e}+1 \wedge
				d = 0 \\
			= \nan & e = \ol{e} + 1 \wedge
				d \neq 0.
		\end{cases}
	\end{equation*}
\end{definition}
In the interest of comparing different parametrisations for
$\M$, we want to find expressions for the smallest positive non-zero
subnormal, smallest positive normal and largest normal floating-point
numbers.
\begin{proposition}[smallest positive non-zero subnormal 
	floating-point number]
	Let $n_m \in \N$ and $\ul{e}\in \Z$. The smallest positive non-zero
	floating-point number is
	\begin{equation*}
		\min\left(\M_0(n_m,\ul{e})\cap \R^+_{\neq 0}\right) =
		2^{\ul{e}-n_m}.
	\end{equation*}
\end{proposition}
\begin{proof}
	Let $0 \neq d \in {\{ 0,1 \}}^{n_m}$. It follows that
	\begin{equation*}
		\min\left(\M_0(n_m,\ul{e})\cap \R^+_{\neq 0}\right) =
		\min\left((-1)^0 \cdot 2^{\ul{e}} \cdot
		\left[0+\sum_{i=1}^{n_m} d_i \cdot 2^{-i}\right] \right) =
		2^{\ul{e}} \cdot 2^{-n_m} = 2^{\ul{e}-n_m}.\qedhere
	\end{equation*}
\end{proof}
\begin{proposition}[smallest positive normal floating-point number]
	Let $n_m \in \N$ and $\ul{e},\ol{e}\in \Z$.
	The smallest positive normal floating-point number is
	\begin{equation*}
		\min\left(\M_1(n_m,\ul{e},\ol{e})\cap \R^+_{\neq 0}
		\right) = 2^{\ul{e}}.
	\end{equation*}
\end{proposition}
\begin{proof}
	Let $0 \neq d \in {\{ 0,1 \}}^{n_m}$ and $\ul{e}\le e\le\ol{e}$. It follows
	that
	\begin{equation*}
		\min\left(\M_1(n_m,\ul{e},\ol{e})\cap \R^+_{\neq 0}\right) =
		\min\left((-1)^0 \cdot 2^{e} \cdot
		\left[1 + \sum_{i=1}^{n_m} d_i \cdot 2^{-i}\right]
		\right) = 2^{\ul{e}}.\qedhere
	\end{equation*}
\end{proof}
\begin{proposition}[largest normal floating-point number]
	Let $n_m \in \N$ and $\ul{e},\ol{e}\in \Z$. The largest normal floating-point
	number is
	\begin{equation*}
		\max\left(\M_1(n_m,\ul{e},\ol{e})\right) =
		2^{\ol{e}}\cdot\left( 2-2^{-n_m}\right).
	\end{equation*}
\end{proposition}
\begin{proof}
	Let $d \in {\{ 0,1 \}}^{n_m}$ and $\ul{e}\le e\le\ol{e}$. It follows
	with the finite geometric series that
	\begin{align*}
		\max\left(\M_1(n_m,\ul{e},\ol{e})\right) &= \max\left(
			{(-1)}^s \cdot 2^e \cdot \left[ 1 +
			\sum_{i=1}^{n_m} d_i \cdot 2^{-i}
			\right]\right) \\
		&= (-1)^0 \cdot 2^{\ol{e}} \cdot
			\left( 1+\sum_{i=1}^{n_m} 2^{-i} \right) \\
		&= 2^{\ol{e}} \cdot \sum_{i=0}^{n_m} 2^{-i} \\
		&= 2^{\ol{e}} \cdot \sum_{i=0}^{n_m} 
			\left(\frac{1}{2}\right)^{i} \\
		&= 2^{\ol{e}} \cdot \frac{1-
			\left(\frac{1}{2}\right)^{n_m+1}}{1-\frac{1}{2}} \\
		&= 2^{\ol{e}}\cdot\left( 2-2^{-n_m}\right) \qedhere
	\end{align*}
\end{proof}
\begin{proposition}[number of $\nan$ representations]\label{prop:nonr}
	Let $n_m \in \N$ and $\ul{e},\ol{e} \in \Z$. The number
	of $\nan$ representations is
	\begin{equation*}
		|\nan|(n_m) := \left|\bigg\{
		\tilde{x}(s,e,d)\in\M(n_m,\ul{e}-1,\ol{e}+1) \,
		\bigg| \, \tilde{x} = \nan
		\bigg\}\right| = 2^{n_m+1}-2.
	\end{equation*}
\end{proposition}
\begin{proof}
	Let $0 \neq d \in {\{ 0,1 \}}^{n_m}$. It follows from Defintion~\ref{def:sofpn}
	that
	\begin{equation*}
		\tilde{x}(s,e,d) = \nan
		\quad \Leftrightarrow \quad  e = \ol{e} + 1 \wedge
		d \neq 0.
	\end{equation*}
	This means that there are $2^{n_m}-1$ possible choices for $d$,
	yielding with the arbitrary $s\in\{0,1\}$ that
	\begin{equation*}
		|\nan|(n_m) = 2\cdot (2^{n_m}-1) = 2^{n_m+1}-2.\qedhere
	\end{equation*}
\end{proof}
\section{Memory Structure}
It is in our interest to map $\M(n_m,\ul{e}-1,\ol{e}+1)$ into a memory
region, more specifically a bit array. The format defined by the IEEE
754-2008 standard is shown in Figure~\ref{fig:ieeemem},
where $n_e$ stands for the number of bits in the exponent, $n_m$ for the bits
in the mantissa and the leading single bit is reserved for the sign bit.
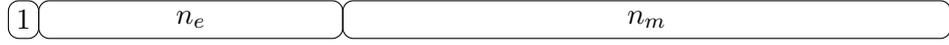
\begin{figure}[htpb]
	\centering
	\begin{tikzpicture}
		\draw[rounded corners] (0,  0  ) rectangle (0.4,0.5) node[pos=.5] {$1$};
		\draw[rounded corners] (0.4,0  ) rectangle (4.4,0.5) node[pos=.5] {$n_e$};
		\draw[rounded corners] (4.4,0  ) rectangle (12.4,0.5) node[pos=.5] {$n_m$};
	\end{tikzpicture}
	\caption{IEEE~754 Floating-point memory layout; see \cite[Figure 3.1]{ieee08}.}
	\label{fig:ieeemem}
\end{figure}
\par
Handling the exponent just as an unsigned integer would not allow
the use of negative exponents. To solve this, the so called
\emph{exponent bias} was introduced in the IEEE~754 standard, which is the
value $2^{n_e-1}-1$ subtracted from the unsigned value of the exponent
(see \cite[Section~3.4b]{ieee08}) and should not be confused with the
\emph{two's complement}, the usual way to express signed integers in a
machine. Looking at the exponent values, the exponent bias results in
\begin{equation*}
	\ul{e}-1 = -2^{n_e-1}+1 \le e \le 2^{n_e}-2^{n_e-1} = \ol{e} + 1
\end{equation*}
and thus
\begin{equation*}
	(\ul{e},\ol{e}) = (-2^{n_e-1}+2, 2^{n_e}-2^{n_e-1}-1) =
	(-2^{n_e-1}+2, 2^{n_e-1}-1).
\end{equation*}
This can be formally expressed as the
\begin{definition}[exponent bias]\label{def:eb}
	Let $n_e \in \N$. The exponent bias is defined as
	\begin{align*}
		\ul{e}(n_e) &:= -2^{n_e-1}+2 \\
		\ol{e}(n_e) &:= 2^{n_e-1}-1.
	\end{align*}
\end{definition}
With the exponent bias representation, we know how many exponent
values can be assumed. Because of that it is now possible to determine the
\begin{proposition}[number of normal floating-point numbers]
	Let $n_m, n_e \in \N$. The number of normal floating-point numbers is
	\begin{equation*}
		|\M_1(n_m,\ul{e}(n_e),\ol{e}(n_e))| = 2^{1+n_e+n_m}-2^{n_m+2}.
	\end{equation*}
\end{proposition}
\begin{proof}
	According to Definition~\ref{def:sonfpn} there are
	\begin{equation*}
		\ol{e}(n_e)-\ul{e}(n_e)+1 = 2^{n_e-1}-1+2^{n_e-1}-2+1 =
		2^{n_e}-2
	\end{equation*}
	different exponents for $\M_1(n_m,\ul{e}(n_e),\ol{e}(n_e))$. Given
	$d \in {\{ 0,1 \}}^{n_m}$ and $s\in\{0,1\}$ are arbitrary it follows that
	\begin{equation*}
		|\M_1(n_m,\ul{e}(n_e),\ol{e}(n_e))| = 2\cdot 2^{n_m}\cdot(2^{n_e}-2)
		= 2^{1+n_e+n_m}-2^{n_m+2}. \qedhere
	\end{equation*}
\end{proof}
\begin{proposition}[number of subnormal floating-point numbers]
	Let $n_m, n_e \in \N$. The number of subnormal floating-point numbers is
	\begin{equation*}
		|\M_0(n_m,\ul{e}(n_e))| = 2^{n_m+1}.
	\end{equation*}
\end{proposition}
\begin{proof}
	According to Definition~\ref{def:sosfpn} it follows with arbitrary
	$d \in {\{ 0,1 \}}^{n_m}$ and $s\in\{0,1\}$ that
	\begin{equation*}
		|\M_0(n_m,\ul{e}(n_e))|=2\cdot 2^{n_m}=2^{n_m+1}.\qedhere
	\end{equation*}
\end{proof}
\begin{proposition}[number of floating-point numbers]
	Let $n_m,n_e \in \N$. The number of floating point numbers is
	\begin{equation*}
		|\M(n_m,\ul{e}(n_e)+1,\ol{e}(n_e)-1)| = 2^{1+n_e+n_m}.
	\end{equation*}
\end{proposition}
\begin{proof}
	We define
	\begin{equation*}
		|\infty| := \left|\bigg\{
		\tilde{x}(s,e,d)\in\M(n_m,\ul{e}-1,\ol{e}+1) \,
		\bigg| \, \tilde{x} = \pm\infty
		\bigg\}\right| = 2
	\end{equation*}
	and conclude from Definition~\ref{def:sofpn} that
	\begin{align*}
		|\M(n_m,\ul{e}(n_e)+1,\ol{e}(n_e)-1)| &=
			|\M_0(n_m,\ul{e}(n_e))| +
			|\M_1(n_m,\ul{e}(n_e),\ol{e}(n_e))| +
			|\infty| + |\nan| \\
		&= 2^{n_m+1} + 2^{1+n_e+n_m}-2^{n_m+2} + 2 +2^{n_m+1}-2 \\
		&= 2^{1+n_e+n_m} + 2^{n_m+1} + 2^{n_m+1} - 2\cdot 2^{n_m+1} \\
		&= 2^{1+n_e+n_m}.
	\end{align*} \qedhere
\end{proof}
Excluding the extended precisions above 64 bit, the
IEEE~754 standard defines three storage sizes for floating-point
numbers (see \cite[Section~3.6]{ieee08}), parametrised by $n_m$ and $n_e$,
as can be seen in Table~\ref{table:floatsizes}.
Half precision floating-point numbers (binary16) were introduced in
IEEE~754-2008 and are just meant to be a storage format and not
used for arithmetic operations given the low dynamic range.
\begin{table}[htpb]
	\centering
	\begin{tabular}{ | l || l | l | l | }
		\hline
		precision & half (binary16) & single (binary32) &
			double (binary64) \\
		\hline\hline
		storage size (bit) & 16 & 32 & 64 \\
		\hline
		$n_e$ (bit) & 5  & 8  & 11 \\
		\hline
		$n_m$ (bit) & 10 & 23 & 52 \\
		\hline
		\hline
		exponent bias & 15 & 127 & 1023 \\
		\hline
		$\ul{e}$ & -14 & -126 & -1022 \\
		\hline
		$\ol{e}$ &  15 &  127 &  1023 \\
		\hline
		$\min(\M_0\cap \R^+_{\neq 0})$ &
			$\approx 5.96 \cdot 10^{-8}$ &
			$\approx 1.40 \cdot 10^{-45}$ &
			$\approx 4.94 \cdot 10^{-324}$ \\
		\hline
		$\min(\M_1\cap \R^+_{\neq 0})$ &
			$\approx 6.10 \cdot 10^{-5}$ &
			$\approx 1.18 \cdot 10^{-38}$ &
			$\approx 2.23 \cdot 10^{-308}$ \\
		\hline
		$\max(\M_1)$ &
			$\approx 6.55 \cdot 10^{+4}$ &
			$\approx 3.40 \cdot 10^{+38}$ &
			$\approx 1.80 \cdot 10^{+308}$ \\
		\hline
		\hline
		$|\M_0|$ &
			$\approx 2.04 \cdot 10^{+3}$ &
			$\approx 1.68 \cdot 10^{+7}$ &
			$\approx 9.01 \cdot 10^{+15}$ \\
		\hline
		$|\M_1|$ &
			$\approx 6.14 \cdot 10^{+4}$ &
			$\approx 4.26 \cdot 10^{+9}$ &
			$\approx 1.84 \cdot 10^{+19}$ \\
		\hline
		$|\nan|$ &
			$\approx 2.05 \cdot 10^{+3}$ &
			$\approx 1.68 \cdot 10^{+7}$ &
			$\approx 9.01 \cdot 10^{+15}$ \\
		\hline
		$|\M|$ &
			$\approx 6.55 \cdot 10^{+4}$ &
			$\approx 4.29 \cdot 10^{+9}$ &
			$\approx 1.84 \cdot 10^{+19}$ \\
		\hline
		$|\nan|/|\M|$ ($\%$) &
			$\approx 3.12$ &
			$\approx 0.39$ &
			$\approx 0.05$ \\
		\hline
	\end{tabular}
	\caption{IEEE~754-2008 binary floating-point numbers
		up to $64$ bit with their characterizing properties.}
	\label{table:floatsizes}
\end{table}
\section{Rounding}
Given $\M(n_m,\ul{e}-1,\ol{e}+1)$ is a finite set, we need a way
to map arbitrary real values into it if we want floating-point
numbers to be a useful model of the real numbers. The IEEE~754
standard solves this with \emph{rounding}, an operation
mapping real values to preferrably close floating-point numbers
based on a set of rules (see \cite[Section~4.3]{ieee08}).
Given the different requirements depending on the task at hand,
the IEEE~754 standard defines five rounding rules. Two
based on rounding to the nearest value (see \cite[Section~4.3.1]{ieee08})
and three based on a directed approach (see \cite[Section~4.3.2]{ieee08}).
\subsection{Nearest}\label{subsec:rounding:nearest}
The most intuïtive approach is to just round to the nearest
floating-point number. In case of a tie though, there has to
be a rule in place to make a decision possible.
Two rules proposed by the IEEE~754 standard are \emph{tiing to
even} (also known as \emph{Banker's rounding}) and
\emph{tiing away from zero}. Only the first mode is
presented here, which is also the default rounding mode
(see \cite[Section~4.3.3]{ieee08}).
\par
This part of the standard is often misunderstood, resulting
in many publications not presenting nearest and tie to even rounding
as the standard rounding operation but nearest and tie away
from zero rounding, which is not correct but easy to overlook.
\begin{definition}[nearest and tie to even rounding]\label{def:natter}
	Let $n_m \in \N$, $\ul{e},\ol{e} \in \Z$ and $x\in\R$ with
	$(s,e,d) \in \{ 0,1 \} \times \Z \times
	{\{ 0,1\}}^{\N_0}$ satisfying
	\begin{equation*}
		x = {(-1)}^s \cdot 2^e \cdot \sum_{i=0}^{\infty}
		\left(d_i \cdot 2^{-i}\right).
	\end{equation*}
	The nearest and tie to even rounding reduction
	\begin{equation*}
		\rde \colon \R \to \M(n_m,\ul{e}-1,\ol{e}+1)
	\end{equation*}
	is defined for
	\begin{align*}
		\ul{x} &:= {(-1)}^s \cdot 2^e \cdot \sum\limits_{i=0}^{n_m}
			\left(d_i \cdot 2^{-i}\right) \\
		\ol{x} &:= {(-1)}^s \cdot 2^e \cdot \left[\sum\limits_{i=0}^{n_m}
			\left(d_i \cdot 2^{-i}\right) + 1\cdot 2^{-n_m}\right]
	\end{align*}	
	as
	\begin{equation*}
		x \mapsto
		\begin{cases}
			{(-1)}^{s} \cdot \infty & |x| \ge
				\max(\M_1)-
				\left(2^{\ol{e}}\cdot 2^{-n_m}\right) =
				2^{\ol{e}}\cdot\left(2-2^{(-n_m-1)}\right) \\
			\ol{x} & |x - \ol{x}| < |x - \ul{x}| \lor \left[
				|x - \ol{x}| = |x - \ul{x}| \land d_{n_m} = 1
				\right] \\
			\ul{x} & |x - \ol{x}| > |x - \ul{x}| \lor \left[
				|x - \ol{x}| = |x - \ul{x}| \land d_{n_m} = 0
				\right].
		\end{cases}
	\end{equation*}
\end{definition}
What this means is that if two nearest machine numbers $\ul{x}$
and $\ol{x}$ are equally close to $x$, the last mantissa bit
$d_{n_m}$ of $\ul{x}$ decides whether $x$ is rounded
to $\ul{x}$ or $\ol{x}$. For $d_{n_m}=0$ we know that $\ul{x}$ is
even and for $d_{n_m}=1$ it follows from the definition that
$\ol{x}$ is even.
\par
Tiing to even may seem like an arbitrary and complicated approach to
rounding, but its stochastic properties make it very useful to avoid
biased rounding-effects in only one direction. Given for a
set of rounding-operations the number of even and odd ties, if they
appear, will be roughly the same with the number of rounding-operations
approaching infinity, it results in a balanced behaviour of up- and
downrounding in tie-cases.
\subsection{Directed}\label{subsec:rounding:directed}
Another way to round numbers is a directed rounding approach to
a given orientation. The three modes have three distinct orientations:
Rounding \emph{toward zero}, \emph{upward} and \emph{downward}.
The first mode is not presented here.
\begin{definition}[upward rounding]\label{def:ur}
	Let $n_m \in \N$, $\ul{e},\ol{e} \in \Z$ and $x\in\R$ with
	$(s,e,d) \in \{ 0,1 \} \times \Z \times
	{\{ 0,1\}}^{\N_0}$ satisfying
	\begin{equation*}
		x = {(-1)}^s \cdot 2^e \cdot \sum_{i=0}^{\infty}
		\left(d_i \cdot 2^{-i}\right).
	\end{equation*}
	The upward rounding reduction
	\begin{equation*}
		\rd_{\uparrow} \colon \R \to \M(n_m,\ul{e}-1,\ol{e}+1)
	\end{equation*}
	is defined for $\ul{x},\ol{x}$ as in
	Definition~\ref{def:natter} as
	\begin{equation*}
		x \mapsto
		\begin{cases}
			\ul{x} & x < 0 \\
			\begin{cases} 
				\ul{x} & \forall i > n_m: d_i = 0 \\
				\ol{x} & \exists i > n_m: d_i = 1
			\end{cases} & x \ge 0.
		\end{cases}
	\end{equation*}
\end{definition}
\begin{definition}[downward rounding]\label{def:dr}
	Let $n_m \in \N$, $\ul{e},\ol{e} \in \Z$ and $x\in\R$ with
	$(s,e,d) \in \{ 0,1 \} \times \Z \times
	{\{ 0,1\}}^{\N_0}$ satisfying
	\begin{equation*}
		x = {(-1)}^s \cdot 2^e \cdot \sum_{i=0}^{\infty}
		\left(d_i \cdot 2^{-i}\right).
	\end{equation*}
	The downward rounding reduction
	\begin{equation*}
		\rd_{\downarrow} \colon \R \to \M(n_m,\ul{e}-1,\ol{e}+1)
	\end{equation*}
	is defined for $\ul{x},\ol{x}$ as in
	Definition~\ref{def:natter} as
	\begin{equation*}
		x \mapsto
		\begin{cases}
			\begin{cases} 
				\ul{x} & \forall i > n_m: d_i = 0 \\
				\ol{x} & \exists i > n_m: d_i = 1
			\end{cases} & x < 0 \\
			\ul{x} & x \ge 0.
		\end{cases}
	\end{equation*}
\end{definition}
The directed rounding modes are important for interval-arithmetic where
it is important not to round down the upper bound or
round up the lower bound of an interval. This way it is always
guaranteed that for $a,b\in\R$ and $a\le b$
\begin{equation}\label{eq:round-interval}
	[a,b] \subseteq [\rd_{\downarrow}(a), \rd_{\uparrow}(b)]
\end{equation}
is satisfied. The bounds may grow faster than by using
a to-nearest rounding mode, but it is guaranteed that the
solution lies inbetween them.
\section{Problems}\label{sec:ieeeproblems}
As with any numerical system, we can find problems exhibiting
its weaknesses. In this context we examine three different
kinds of problems.
Using the results obtained here it will allow us to evaluate
if and how good the Unum arithmetic solves these problems
respectively.
\subsection{The Silent Spike}\label{subsec:spike}
This example has been taken from \cite[§7]{ka06} and simplified.
Consider the function $f\colon\R\to\R$ defined as
\begin{equation}\label{eq:spike}
	f(x) := \ln(|3\cdot(1-x)+1|).
\end{equation}
It is easy to see that we hit a spike where
\begin{align*}
	& |3\cdot(1-x)+1| = 0 \\
	\Leftrightarrow \quad &
		3\cdot (1-x)+1 = 0 \\
	\Leftrightarrow \quad &
		3-3\cdot x+1 = 0 \\
	\Leftrightarrow \quad &
		x = \frac{4}{3}.
\end{align*}
More specifically,
\begin{equation*}
	\lim_{x\downarrow \frac{4}{3}}{\left(f(x)\right)} =
	\lim_{x\uparrow \frac{4}{3}}{\left(f(x)\right)} =
	-\infty.
\end{equation*}
Implementing this problem using IEEE 754 floating-point numbers
(see listing~\ref{lst:spike}),
we might expect to receive a very small number or even negative
infinity in an environment of $\frac{4}{3}$. However, this is not the
case.
\par
Instead, as you can see in Figure~\ref{fig:spike}, the program
claims that $f(\frac{4}{3})\approx -36.044$ is the minimum in
direct vicinity of $\frac{4}{3}$, completely hiding the fact that
$f$ is singular in $\frac{4}{3}$.
\begin{figure}[htpb]
	\centering
	\begin{tikzpicture}
		\begin{axis}[xlabel=$x-\frac{4}{3}$, ylabel=$f(x)$,
			         extra y ticks={-36.04365338911715355152},
			         extra y tick style={yticklabel pos=right,
			         	ytick pos=right},
		             yticklabel style={/pgf/number format/fixed,
		             	/pgf/number format/precision=3},
		             extra x ticks={-2.22044604925031308085e-15,
		                             2.22044604925031308085e-15},
		             extra x tick style={xticklabel pos=right,
			         	xtick pos=right},
			         extra tick style={grid=major}]
			\addplot[mark=x, smooth] coordinates {
				(-2.22044604925031308085e-15, -32.60966618463201172062)
				(-1.99840144432528177276e-15, -32.71144887894195107947)
				(-1.77635683940025046468e-15, -32.82477756424895432019)
				(-1.55431223447521915659e-15, -32.95261093575884103757)
				(-1.33226762955018784851e-15, -33.09921440995071861835)
				(-1.11022302462515654042e-15, -33.27106466687737196253)
				(-8.88178419700125232339e-16, -33.47870403165561725700)
				(-6.66133814775093924254e-16, -33.74106829612311031497)
				(-4.44089209850062616169e-16, -34.09774324006184542668)
				(-2.22044604925031308085e-16, -34.65735902799726630974)
				( 0.00000000000000000000e+00, -36.04365338911715355152)
				( 2.22044604925031308085e-16, -35.35050620855721348335)
				( 4.44089209850062616169e-16, -34.43421547668305748857)
				( 6.66133814775093924254e-16, -33.96421184743731913613)
				( 8.88178419700125232339e-16, -33.64575811631878821117)
				( 1.11022302462515654042e-15, -33.40459605950189825307)
				( 1.33226762955018784851e-15, -33.21044004506094182716)
				( 1.55431223447521915659e-15, -33.04792111556316314136)
				( 1.77635683940025046468e-15, -32.90815917318800387648)
				( 1.99840144432528177276e-15, -32.78555685109567718882)
				( 2.22044604925031308085e-15, -32.67635755913067896472)
			};
		\end{axis}
	\end{tikzpicture}
	\caption{Interpolated evaluations (demarked by crosses) of $f$
		(see (\ref{eq:spike})) in the neighbourhood of $\frac{4}{3}$ for
		all possible double floating-point numbers in
		$\left[\frac{4}{3}-2.22\cdot 10^{-15},\frac{4}{3}+2.22\cdot 10^{-15}
		\right]$ (see listing~\ref{lst:spike}).}
	\label{fig:spike}
\end{figure}
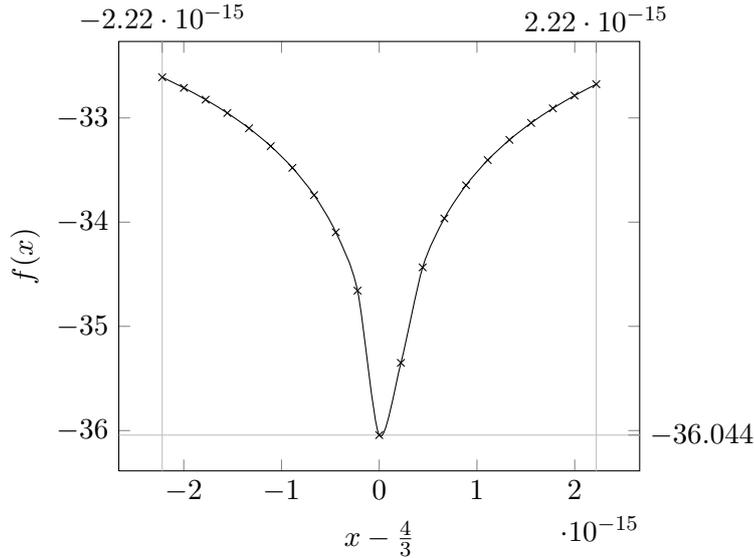
The reason why the floating-point implementation hides the
singularity is not that the logarithm implementation is faulty,
but because the value passed to the logarithm is off in the first
place. It is easy to see the singular point $\frac{4}{3}$ cannot
be exactly represented in the machine. This effect is increased
with rounding errors occuring during the evaluation
(see Listing~\ref{lst:spike}) of
\begin{equation*}
	\left|\rde\left\{
	\rde\left[\rde(3)\cdot\rde\left(\rde(1)-
	\rde\left(\frac{4}{3}\right)\right)\right]+\rde(1)
	\right\}\right| \approx
	2.2204 \cdot 10^{-16}.
\end{equation*}
In magnitude, this is relatively close to zero, but given
\begin{equation*}
	\ln(2.2204 \cdot 10^{-16}) \approx -36.0437
\end{equation*}
we not only see the significance of the rounding error, but also the
reason why the floating-point implementation claims that
$-36.044$ is the minimum of $f$ in direct vicinity of
$\frac{4}{3}$.
\par
This result indicates that there are simple examples where
floating-point numbers fail for piecewise continuous functions
with singularities.
Not being able to spot singularities for a given function might
have drastic consequences, for example \enquote{hiding} destructive
frequencies in resonance curves for the oscillation of bridge
stay cables, which are, for instance, derived in \cite{pi96}.
\subsection{Devil's Sequence}\label{subsec:devil}
This example has been taken from \cite[Chapter~1.3.2]{mu10}. Consider the
recurrent series ${\left\{ u_n\right\}}_{n\in\N_0}$ defined as
\begin{equation}\label{eq:devil}
	u_n :=
	\begin{cases}
		2 & n = 0\\
		-4 & n = 1 \\
		111 - \frac{1130}{u_{n-1}}+\frac{3000}{u_{n-1}\cdot u_{n-2}} &
			n \ge 2
	\end{cases}
\end{equation}
and determine the possible limits of this series, if they exist.
For this purpose, we assume convergence with $u:=u_n=u_{n-1}=u_{n-2}$
and obtain the characteristic polynomial relation
\begin{align*}
	& u = 111 - \frac{1130}{u}+\frac{3000}{u^2} \\
	\Leftrightarrow \quad &
		u^3 = 111 \cdot u^2 - 1130 \cdot u + 3000 \\
	\Leftrightarrow \quad &
		0 = u^3-111 \cdot u^2 + 1130 \cdot u - 3000
\end{align*}
with solutions $5$, $6$ and $100$. As further described in \cite[§5]{ka06}
for a similar recurrence, we obtain the general solution with $\alpha,\beta,
\gamma\in\R$ under the condition $|\alpha|+|\beta|+|\gamma|\neq 0$
\begin{equation}\label{eq:devil-closed-form}
	u_n = \frac{\alpha \cdot 100^{n+1}+\beta \cdot 6^{n+1}+\gamma \cdot
		5^{n+1}}{\alpha \cdot 100^n + \beta \cdot 6^n + \gamma \cdot 5^n}.
\end{equation}
For $u_0=2$ and $u_1=-4$ we obtain $\alpha=0$ and
$\gamma=-\frac{3}{4}\cdot\beta\neq 0$, resulting in
\begin{align*}
	u_n & = \frac{6^{n+1}-\frac{3}{4}\cdot 5^{n+1}}
		{6^n-\frac{3}{4}\cdot 5^n} \\
	& = \frac{6^{n+1}-\frac{3}{4}\cdot
		{\left(\frac{5}{6}\cdot 6\right)}^{n+1}}
		{6^n-\frac{3}{4}\cdot
		{\left(\frac{5}{6}\cdot 6\right)}^{n}} \\
	& = \frac{6^{n+1}}{6^n} \cdot
		\frac{1-\frac{3}{4}\cdot
		{\left( \frac{5}{6}\right)}^{n+1}}
		{1 -\frac{3}{4}\cdot
		{\left( \frac{5}{6} \right)}^{n}} \\
	& = 6 \cdot
		\frac{1-\frac{3}{4}\cdot
		{\left( \frac{5}{6}\right)}^{n+1}}
		{1 -\frac{3}{4}\cdot
		{\left( \frac{5}{6} \right)}^{n}}.
\end{align*}
It follows that
\begin{equation*}
	\lim_{n\to\infty}{(u_n)} = 6.
\end{equation*}
If we take a look at the floating-point implementation (see
listing~\ref{lst:devil}) of this problem, we can observe
a rather peculiar behaviour:
Figure~\ref{fig:devil} shows that the IEEE~754-based
solver behaves completely opposite from what one might expect.
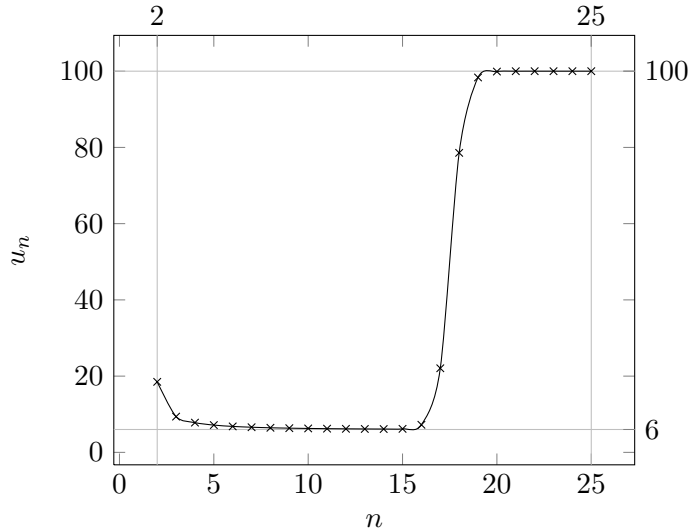
\begin{figure}[htpb]
	\centering
	\begin{tikzpicture}
		\begin{axis}[xlabel=$n$, ylabel=$u_n$,
		             extra x ticks={2,25},
		             extra x tick style={xticklabel pos=right,
			         	xtick pos=right},
			         extra y ticks={6,100},
			         extra y tick style={yticklabel pos=right,
			         	ytick pos=right},
			         extra tick style={grid=major}]
			\addplot[mark=x, smooth] coordinates {
				(   2,  18.500000) (   3,   9.378378)
				(   4,   7.801153) (   5,   7.154414)
				(   6,   6.806785) (   7,   6.592633)
				(   8,   6.449466) (   9,   6.348452)
				(  10,   6.274439) (  11,   6.218697)
				(  12,   6.175854) (  13,   6.142627)
				(  14,   6.120249) (  15,   6.166087)
				(  16,   7.235021) (  17,  22.062078)
				(  18,  78.575575) (  19,  98.349503)
				(  20,  99.898569) (  21,  99.993871)
				(  22,  99.999630) (  23,  99.999978)
				(  24,  99.999999) (  25, 100.000000)
			};
		\end{axis}
	\end{tikzpicture}
	\caption{Interpolated double floating-point evaluations
		(demarked by crosses) of the devil's sequence $u_n$
		(see (\ref{eq:devil})) for $n\in\{2,\ldots,25\}$
		(see listing~\ref{lst:devil}).}
	\label{fig:devil}
\end{figure}
Using the closed form (\ref{eq:devil-closed-form}) we have
shown that the recurrence (\ref{eq:devil}) converges to
$6$. However, even though the floating-point solver comes quite
close to $6$ up until $n=15$, it unexpectedly converges to $100$ in
subsequent iterations. The reason for that is found within
consecutive rounding errors of $u_n$, which skew the results so
far that the parametre $\alpha$ of the closed form
(\ref{eq:devil-closed-form}) becomes non-zero.
\par
The carefully chosen starting values $u_0=2$ and $u_1=-4$
deliberately make $\alpha$ disappear in
(\ref{eq:devil-closed-form}), which shows how even little rounding
errors can give completely wrong results for such a pathologic
example.
\subsection{The Chaotic Bank Society}\label{subsec:bank}
This example has been taken from \cite[Chapter~1.3.2]{mu10}.
Consider the recurrent series ${\left\{ a_n\right\}}_{n\in\N_0}$
defined for $a_0\in\R$ as
\begin{equation}\label{eq:bank}
	a_n :=
	\begin{cases}
		a_0 & n = 0\\
		a_{n-1} \cdot n - 1 & n \ge 1
	\end{cases}
\end{equation}
with the task being to determine $u_{25}$ for $a_0=e-1$.
\par
The name of this example can be derived by thinking of the series
as an imaginary offer by a bank to start with a
deposit of $e-1$ currency units and in each year for 25 years,
multiply it by the current running year number and subtract one
currency unit as banking charges.
\par
For a theoretical answer, we first want to find a closed form of
$u_n$. We observe the pattern
\begin{align*}
	a_0 &= a_0 &=~& 0! \cdot \left( a_0 \right) \\
	a_1 &= a_0 \cdot 1 - 1 &=~& 1! \cdot \left(
		a_0 - \frac{1}{1!} \right)\\
	a_2 &= \left( a_0 \cdot 1 - 1 \right) \cdot 2 - 1
		&=~& 2! \cdot \left( a_0 - \frac{1}{1!} -
		\frac{1}{2!} \right)\\
	a_3 &= \left[ \left( a_0 \cdot 1 - 1 \right)
		\cdot 2 - 1 \right] \cdot 3 - 1 &=~&
		3! \cdot \left( a_0 - \frac{1}{1!} -
		\frac{1}{2!} - \frac{1}{3!} \right).
\end{align*}
This leads us to the
\begin{proposition}[closed form of $a_n$]
	The closed form of the recurrent series (\ref{eq:bank}) is
	\begin{equation*}
		a_n = n! \cdot \left( a_0 - \sum_{k=1}^n{\frac{1}{k!}} \right)
	\end{equation*}
\end{proposition}
\begin{proof}
	We prove the statement by induction over $n\in\N_0$.
	\begin{enumerate}[a)]
		\item{$a_0 = a_0 = 0! \cdot a_0$.}
		\item{Assume $a_n = n! \cdot \left( a_0 -
			\sum_{k=1}^n{\frac{1}{k!}} \right)$ holds true for
			an arbitrary but fixed $n\in\N$.}
		\item{Show $n\mapsto n+1$.
			\begin{align*}
				a_{n+1} & = a_n \cdot (n+1) - 1 \\
				& \overset{b)}{=}
					n! \cdot \left( a_0 -
					\sum_{k=1}^n{\frac{1}{k!}}
					\right) \cdot (n+1) - 1 \\
				& = (n+1)! \cdot \left(
					a_0 - \sum_{k=1}^n{\frac{1}{k!}} -
					\frac{1}{(n+1)!} \right) \\
				& = (n+1)! \cdot \left(
					a_0 - \sum_{k=1}^{n+1}{\frac{1}{k!}} \right)
					\qedhere
			\end{align*}
		}
	\end{enumerate}
\end{proof}
Using the closed form of $a_n$ and the definition of \textsc{Euler}'s
number, we get for a disturbed $a_0=(e-1)+\delta$ with $\delta\in\R$
\begin{align*}
	a_n &= n! \cdot \left( (e-1)+\delta- \sum_{k=1}^n{\frac{1}{k!}} \right)\\
	&= n! \cdot \left( \delta + e - 1 - \sum_{k=1}^n{\frac{1}{k!}} \right)\\
	&= n! \cdot \left( \delta + \sum_{k=0}^{+\infty}{\frac{1}{k!}} -
		\sum_{k=0}^n{\frac{1}{k!}} \right) \\
	&= n! \cdot \left( \delta + \sum_{k=n+1}^{+\infty}{\frac{1}{k!}}\right)\\
	&= n! \cdot \delta + \sum_{k=n+1}^{+\infty}{\frac{n!}{k!}}.
\end{align*}
It follows that
\begin{equation*}
	\lim_{n\to+\infty}{(a_n)} = \begin{cases}
		-\infty & \delta < 0 \\
		0 & \delta = 0 \\
		+\infty & \delta > 0
	\end{cases}
\end{equation*}
and, thus, we can assume $a_{25}\in(0,e-1)$ for an undisturbed $a_0=e-1$.
In regard to the
banking context this means that this offer would not be favourable for any
investor.
\par
A sloppy but quicker approach to get an answer to the problem is to write a
program based on IEEE~754 floating-point numbers to calculate the account balance
$a_{25}$ (see listing~\ref{lst:bank}). However, the answer it gives is
$a_{25} = 1201807247.410449$, suggesting a profitable offer by the bank,
which it clearly is not. The reason for this erratic behaviour is that
\begin{equation*}
	\rde(1.718281828459045235) > e - 1,
\end{equation*}
resulting in $\delta>0$ and $a_n$ going towards positive infinity.
\par
This example shows how rounding errors in floating-point arithmetic can lead
to false predictions and ultimately decisions, indicating the need
for guaranteed solution bounds.
As elaborated in Subsection~\ref{subsec:rounding:nearest}, the nearest and tie to even
rounding reduction has some advantages, but in cases like this can skew the
result undesiredly and unexpectedly due to the inhomogenous behaviour of rounding.
Because of that, using another constant expression for a value close to $e-1$ might
result in the answer going towards negative infinity.
\chapter{Interval Arithmetic}\label{ch:ia}
The foundation for modern interval arithmetic was set by Ramon E.
\textsc{Moore} in 1967 (see \cite{mo67}) as a means for automatic error
analysis in algorithms. Since then, the usage of interval arithmetic
beyond stability analysis was limited to some applications (see
\cite{mu06}, \cite{mo79} and \cite{mo09}),
which is also indicated by the
fact that the first IEEE standard for interval arithmetic, IEEE 1788-2015,
was published in 2015 (see \cite{ieee15}).
The standard is based on the ubiquitous \emph{affinely extended real numbers}
\begin{equation*}
	\ol{\R}:=\R \cup \{ +\infty\} \cup \{ -\infty \},
\end{equation*}
which this chapter will not make use of. Instead, the basis will be
the \emph{projectively extended real numbers}
\begin{equation*}
	\R^*:=\R\cup\{ \iffy \}.
\end{equation*}
The motivation for
this chapter is to find out how much we lose when only having one
symbol for infinity, and more importantly, what we gain in this process,
ultimately proving well-definedness of $\R^*$.
Based on the findings, it is in our interest to construct an interval
arithmetic on top of $\R^*$, which we can later use to formalise
the Unum arithmetic.
\section{Projectively Extended Real Numbers}\label{sec:pern}
With respect to simple reciprocation and negation of
numbers, the projectively extended real numbers come to
mind. Topologically speaking, this is the Alexandroff
compactification of $\R$ with the point $\iffy\not\in\R$
(see \cite[Section~25.4]{ko14} for further reading).
\par
As one can see in Figure~\ref{fig:Rschema}, the geometric
projection of $\R$ and infinity $\iffy$ onto a circle,
and thinking of reciprocation and negation as horizontal
and vertical reflections on this circle respectively, is
the ideal model in this context, presenting
an intuïtive approach to arithmetic operations on sets of
real numbers.
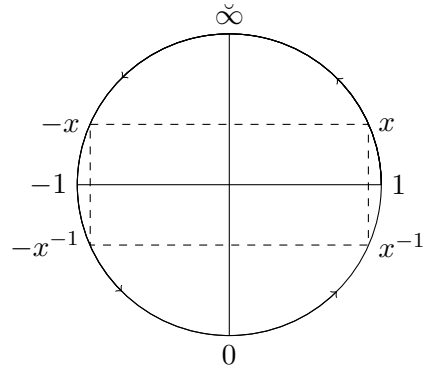
\begin{figure}[htpb]
	\centering
	\begin{tikzpicture}
		\draw[black,arrows={->}] (2,0) arc (0:315:2);
		\draw[black,arrows={->}] (2,0) arc (0:225:2);
		\draw[black,arrows={->}] (2,0) arc (0:135:2);
		\draw[black,arrows={->}] (2,0) arc (0:45:2);
		\draw (0,0) circle (2);
		\draw[dashed] (-1.83, 0.8) node[anchor=east] {$-x$}
		           -- ( 1.83, 0.8) node[anchor=west] {$x$}
		           -- ( 1.83,-0.8) node[anchor=west] {$x^{-1}$}
		           -- (-1.83,-0.8) node[anchor=east] {$-x^{-1}$}
		           -- cycle;
		\draw[      ] ( 0, -2) node[anchor=north] {$0$}
		           -- ( 0,  2) node[anchor=south] {$\iffy$};
		\draw[      ] (-2,  0) node[anchor=east]  {$-1$}
		           -- ( 2,  0) node[anchor=west]  {$1$};
	\end{tikzpicture}
	\caption{Schema of $\R^*$ with the counter-clockwise orientation
		indicated by arrows (see Definition~\ref{def:pern}).}
	\label{fig:Rschema}
\end{figure}
\par
Just like we can not definitely give the number $0$
a sign and just by convention denote it as a positive number,
there is no reason for its reciprocal $\iffy$ to have a sign.
As intuïtive as this approach is, rigorous results and a formal
definition are necessary to build a solid foundation for interval
arithmetic on the projectively extended real numbers.
In the course of the following chapter we are going to define
finite and infinite limits on the projectively extended real
numbers and show well-definedness of this extension in terms
of infinite limits. The formal definition of $\R^*$ is according
to \cite{re82}.
\begin{definition}[projectively extended real numbers]\label{def:pern}
	The projectively extended real numbers are defined as
	\begin{equation*}
		\R^* := \R \cup \{ \iffy \}.
	\end{equation*}
	The arithmetic operations $+$ and $\cdot$ are partially extended
	for $a,b \in \R$ with $b \neq 0$ to
	\begin{subequations}
		\begin{align}
			-(\iffy) &:= \iffy \label{eq:def:pern1}\\
			a + \iffy = \iffy + a  &:= \iffy \label{eq:def:pern2}\\
			b \cdot \iffy = \iffy \cdot b &:= \iffy \label{eq:def:pern3} \\
			a / \iffy &:= 0 \label{eq:def:pern4} \\
			b / 0 &:= \iffy. \label{eq:def:pern5}
		\end{align}
	\end{subequations}
	Left undefined are $\iffy + \iffy$, $\iffy \cdot \iffy$,
	$0 \cdot \iffy$, $0/0$, $\iffy/\iffy$ and $\iffy/0$.
\end{definition}
For more information on indeterminate forms on extensions
of the real numbers see \cite{tf95}.
\par
To be able to show well-definedness of the extension of the
arithmetic operations in $\R^*$ in terms of infinite limits,
we first have to introduce the concept of $\iffy$-infinite
limits on $\R^*$.
\subsection{Finite and Infinite Limits}
Since we can not use two signed symbols for infinity, namely $\pm\infty$,
directed limits can be specified with the direction of
approach to $\iffy$, from above or below, indicated by vertical arrows.
In this regard, ascension is interpreted in regard to the natural
order of $\R$, from smallest to largest number. Approaching
$\iffy$ from below corresponds to a limit toward
$+\infty$ on $\R$, approaching $\iffy$ from above corresponds to a limit
toward $-\infty$ on $\R$.
\par
There is no sacrifice in only having one symbol for infinity up
to this point, given $+\infty$ and $-\infty$ can only be approached
from one direction in standard analysis. Having one symbol that can be
approached from two directions fills the gap seamlessly for
finite limits.
\begin{definition}[$\iffy$-finite limit]\label{def:finlimits}
	Let $f\colon\R \to \R$. The $\iffy$-finite limit of $f$
	for $x$ approaching $\iffy$ is defined for $\ell \in \R$ as
	\begin{align*}
		\lim_{x\downarrow \iffy}{(f(x))} = \ell &
			\quad :\Leftrightarrow \quad
			\forall \varepsilon > 0 : \exists c \in \R :
			\forall x \in \R: x < c : |f(x) - \ell| < \varepsilon \\
		\lim_{x\uparrow \iffy}{(f(x))} = \ell &
			\quad :\Leftrightarrow \quad
			\forall \varepsilon > 0 : \exists c \in \R :
			\forall x \in \R: x > c : |f(x) - \ell| < \varepsilon \\
		\lim_{x\to \iffy}{(f(x))} = \ell &
			\quad :\Leftrightarrow \quad
			\lim_{x\downarrow \iffy}{(f(x))} = \ell \wedge
			\lim_{x\uparrow \iffy}{(f(x))} = \ell.
	\end{align*}
\end{definition}
\begin{remark}[standard-finite limit relationship]
	Let $f\colon\R \to \R$ and $\ell \in \R$. One can convert
	between standard-finite limits and $\iffy$-finite
	limits using the relations
	\begin{align*}
		\lim_{x\downarrow \iffy}{(f(x))} = \ell &
			\quad \Leftrightarrow \quad
			\lim_{x\to-\infty}{(f(x))} = \ell \\
		\lim_{x\uparrow \iffy}{(f(x))} = \ell &
			\quad \Leftrightarrow \quad
			\lim_{x\to+\infty}{(f(x))} = \ell.
	\end{align*}
\end{remark}
Besides finite limits, we also need a way to express when a function
diverges. In this regard, having only one infinity-symbol induces
some losses, as only the absolute values of the functions can be
evaluated. However, it still holds that if a function
diverges in standard-infinite limits it also diverges in
$\iffy$-infinite limits.
\begin{definition}[$\iffy$-infinite limit]\label{def:inflimits}
	Let $f\colon\R \to \R$. The $\iffy$-infinite limit of $f$
	for $x\in\R$ approaching $a \in \R$ is defined as
	\begin{align*}
		\lim_{x\downarrow a}{(f(x))} = \iffy &
			\quad :\Leftrightarrow \quad
			\forall \varepsilon > 0 : \exists \delta > 0 :
			0 < x - a < \delta \Rightarrow |f(x)|>\varepsilon \\
		\lim_{x\uparrow a}{(f(x))} = \iffy &
			\quad :\Leftrightarrow\quad
			\forall \varepsilon > 0 : \exists \delta > 0 :
			0 < a - x < \delta \Rightarrow |f(x)|>\varepsilon \\
		\lim_{x\to a}{(f(x))} = \iffy &
			\quad :\Leftrightarrow \quad
			\lim_{x\downarrow a}{(f(x))} = \iffy \wedge
			\lim_{x\uparrow a}{(f(x))} = \iffy,
	\end{align*}
	and for $x\in\R$ approaching $\iffy$ as
	\begin{align*}
		\lim_{x\downarrow \iffy}{(f(x))} = \iffy &
			\quad :\Leftrightarrow\quad
			\forall \varepsilon > 0 : \exists c \in \R :
			\forall x \in \R : x < c : |f(x)|>\varepsilon \\
		\lim_{x\uparrow \iffy}{(f(x))} = \iffy &
			\quad :\Leftrightarrow \quad
			\forall \varepsilon > 0 : \exists c \in \R :
			\forall x \in \R :x > c : |f(x)|>\varepsilon \\
		\lim_{x\to \iffy}{(f(x))} = \iffy &
			\quad :\Leftrightarrow \quad
			\lim_{x\downarrow \iffy}{(f(x))} = \iffy \wedge
			\lim_{x\uparrow \iffy}{(f(x))} = \iffy.
	\end{align*}
\end{definition}
\begin{remark}[standard-infinite limit relationship]
	Let $f\colon\R \to \R$ and $a\in\R$. One can convert
	between standard-infinite limits and $\iffy$-infinite
	limits using the relations
	\begin{align*}
		\lim_{x\downarrow a}{(f(x))} = \iffy &
			\quad \Leftarrow \quad
			\lim_{x\downarrow a}{(f(x))} = \pm\infty \\
		\lim_{x\uparrow a}{(f(x))} = \iffy &
			\quad \Leftarrow \quad
			\lim_{x\uparrow a}{(f(x))} = \pm\infty \\
		\lim_{x\downarrow \iffy}{(f(x))} = \iffy &
			\quad \Leftarrow \quad
			\lim_{x\to -\infty}{(f(x))} = \pm\infty \\
		\lim_{x\uparrow \iffy}{(f(x))} = \iffy &
			\quad \Leftarrow \quad
			\lim_{x\to +\infty}{(f(x))} = \pm\infty.
	\end{align*}
\end{remark}
\subsection{Well-Definedness}
We can now use our definitions of $\iffy$-finite and
$\iffy$-infinite limits to show that $\R^*$ with the
extensions given in Definition~\ref{def:pern} is well-defined
in terms of infinite limits.
\begin{theorem}[well-definedness of $\R^*$]
	$\R^*$ is well-defined in terms of infinite limits.
\end{theorem}
\begin{proof}
	Let $f_{\iffy},f_a,f_b,f_0\colon\R \to \R$,
	$a,b \in \R$ and $b \neq 0$. Without loss of generality we
	assume that $\iffy$ is approached from below and specify
	\begin{subequations}
		\begin{align}
			\lim\limits_{x\uparrow\iffy}{(f_{\iffy}(x))} &= \iffy
				\label{eq:pern-pre-iffy}\\
			\lim\limits_{x\uparrow\iffy}{(f_a(x))} &= a
				\label{eq:pern-pre-a}\\
			\lim\limits_{x\uparrow\iffy}{(f_b(x))} &= b
				\label{eq:pern-pre-b}\\
			\lim\limits_{x\uparrow\iffy}{(f_0(x))} &= 0.
				\label{eq:pern-pre-0}
		\end{align}
	\end{subequations}
	To show well-definedness, we go through each axiom
	given in Definition~\ref{def:pern}.
	\par
	Let $\tilde{\varepsilon} > 0$.
	\begin{enumerate}
		\item[(\ref{eq:def:pern1})]{
			By Definition~\ref{def:inflimits} we know that
			\begin{equation*}
				\lim\limits_{x\uparrow\iffy}{\left(f_{\iffy}(x)\right)} = \iffy
				\quad \Leftrightarrow \quad
				\lim\limits_{x\uparrow\iffy}{\left(-f_{\iffy}(x)\right)} = \iffy
			\end{equation*}
			and, thus, $-(\iffy)=\iffy$ is well-defined.
		}
		\item[(\ref{eq:def:pern2})]{
			To show that $a + \iffy = \iffy + a = \iffy$
			is well-defined we have to show that
			\begin{equation}\label{eq:pern2-goal}
				\lim\limits_{x\uparrow\iffy}{\left(f_a(x) + f_{\iffy}(x)\right)} =
				\lim\limits_{x\uparrow\iffy}{\left(f_{\iffy}(x) + f_a(x)\right)} =
				\iffy.
			\end{equation}
			Following from precondition~(\ref{eq:pern-pre-a}),
			Definition~\ref{def:inflimits} and $\tilde{\varepsilon} > 0$ we know that
			\begin{equation*}
				\exists c_{2,a} \in \R: \forall x > c_{2,a}: |f_a(x) - a| <
				\tilde{\varepsilon}.
			\end{equation*}
			It follows for $x > c_{2,a}$ using the reverse triangle inequality
			that
			\begin{align}
				& \tilde{\varepsilon} > |f_a(x)-a| \ge
					||f_a(x)|-|a|| \ge |f_a(x)|-|a| \notag \\
				\Rightarrow \quad & |f_a(x)| < \tilde{\varepsilon} + |a|
					\label{eq:pern2-1}.
			\end{align}
			Following from precondition~(\ref{eq:pern-pre-iffy}),
			Definition~\ref{def:inflimits} and $2\cdot\tilde{\varepsilon}+|a| > 0$
			we also know that
			\begin{equation}\label{eq:pern2-2}
				\exists c_{2,\iffy} \in \R :\forall x > c_{2,\iffy}:
				|f_{\iffy}(x)| > 2\cdot \tilde{\varepsilon} + |a|.
			\end{equation}
			Let $x > \tilde{c}_2 := \max\{c_{2,a},c_{2,\iffy}\}$ to satisfy
			both (\ref{eq:pern2-1}) and (\ref{eq:pern2-2}).
			It follows using the reverse triangle inequality that
			\begin{align*}			
				& |f_{\iffy}(x)| > 2 \cdot \tilde{\varepsilon} +
					|a| = \tilde{\varepsilon} +
					(\tilde{\varepsilon} + |a|) > \tilde{\varepsilon} + |f_a(x)| \\
				\Rightarrow \quad &
					\tilde{\varepsilon} < |f_{\iffy}(x)| - |f_a(x)| =
					|f_{\iffy}(x)| - |-f_a(x)| \le
					|f_{\iffy}(x) - (-f_a(x)) | \\
				\Rightarrow \quad &
					|f_a(x) + f_{\iffy}(x)| = |f_{\iffy}(x) + f_a(x)| >
					\tilde{\varepsilon},
			\end{align*}
			which by Definition~\ref{def:inflimits} is equivalent to
			(\ref{eq:pern2-goal}) and was to be shown.
		}
		\item[(\ref{eq:def:pern3})]{
			To show that $b \cdot \iffy = \iffy \cdot b = \iffy$
			is well-defined we have to show that
			\begin{equation}\label{eq:pern3-goal}
				\lim\limits_{x\uparrow\iffy}{\left(f_b(x) \cdot f_{\iffy}(x)\right)} =
				\lim\limits_{x\uparrow\iffy}{\left(f_{\iffy}(x) \cdot f_b(x)\right)} =
				\iffy.
			\end{equation}
			Following from precondition~(\ref{eq:pern-pre-b}),
			Definition~\ref{def:inflimits} and $\frac{|b|}{2} > 0$ we know
			\begin{equation*}
				\exists c_{3,b} \in \R: \forall x > c_{3,b}:
				|f_b(x) - b| < \frac{|b|}{2}.
			\end{equation*}
			It follows for $x > c_{3,b}$ using the triangle and reverse triangle
			inequalities that
			\begin{align}
				& |f_b(x)-b| < \frac{|b|}{2} = \frac{|0-b|}{2} \le
					\frac{|0-f_b(x)|+|f_b(x)-b|}{2} \notag\\
				\Rightarrow \quad &
					\frac{|f_b(x)-b|}{2} <
					\frac{|f_b(x)|}{2} \notag\\
				\Leftrightarrow \quad &
					|f_b(x)| > |f_b(x)-b|=|b-f_b(x)| \ge
					||b|-|f_b(x)|| \ge |b|-|f_b(x)| \notag\\
				\Rightarrow \quad &
				|f_b(x)| > \frac{|b|}{2}. \label{eq:pern3-1}
			\end{align}
			Following from precondition~(\ref{eq:pern-pre-iffy}),
			Definition~\ref{def:inflimits} and $\frac{2\cdot\tilde{\varepsilon}}{|b|} > 0$
			we also know that
			\begin{equation}\label{eq:pern3-2}
				\exists c_{3,\iffy} \in \R :\forall x > c_{3,\iffy}:
				|f_{\iffy}(x)| > \frac{2\cdot\tilde{\varepsilon}}{|b|}.
			\end{equation}
			Let $x > \tilde{c}_3 := \max\{c_{3,b},c_{3,\iffy}\}$ to satisfy
			both (\ref{eq:pern3-1}) and (\ref{eq:pern3-2}).
			It follows that
			\begin{align*}
				& |f_{\iffy}(x)| > \frac{2 \cdot \tilde{\varepsilon}}{|b|}
					> \frac{\tilde{\varepsilon}} {|f_b(x)|} \\
				\Rightarrow \quad & |f_b(x)| \cdot |f_{\iffy}(x)| >
					\tilde{\varepsilon} \\
				\Leftrightarrow \quad & |f_b(x) \cdot f_{\iffy}(x)| =
					|f_{\iffy}(x) \cdot f_b(x)| > \tilde{\varepsilon},
			\end{align*}
			which by Definition~\ref{def:inflimits} is equivalent to
			(\ref{eq:pern3-goal}) and was to be shown.
		}
		\item[(\ref{eq:def:pern4})]{
			To show that $a / \iffy = 0$
			is well-defined we have to show that
			\begin{equation}\label{eq:pern4-goal}
				\lim\limits_{x\uparrow\iffy}{\left(
				\frac{f_a(x)}{f_{\iffy}(x)}\right)} = 0.
			\end{equation}
			Following from precondition~(\ref{eq:pern-pre-iffy}),
			Definition~\ref{def:inflimits} and
			$\frac{\tilde{\varepsilon}+|a|}{\tilde{\varepsilon}} > 0$
			we know
			\begin{equation}\label{eq:pern4-2}
				\exists c_{4,\iffy} \in \R :\forall x > c_{4,\iffy}:
				|f_{\iffy}(x)| > \frac{\tilde{\varepsilon}+|a|}
				{\tilde{\varepsilon}}.
			\end{equation}
			Let $x > \tilde{c}_4 := \max\{c_{2,a},c_{4,\iffy}\}$ to satisfy
			both (\ref{eq:pern2-1}) and (\ref{eq:pern4-2}).
			It follows that
			\begin{align*}
				& |f_a(x)| < \tilde{\varepsilon} + |a| =
					\tilde{\varepsilon} \cdot
					\frac{\tilde{\varepsilon}+|a|}
					{\tilde{\varepsilon}} <
					\tilde{\varepsilon} \cdot |f_{\iffy}(x)|\\
				\Rightarrow \quad &
					\frac{|f_a(x)|}{|f_{\iffy}(x)|} <
					\tilde{\varepsilon} \\
				\Leftrightarrow \quad &
					\left|\frac{f_a(x)}{f_{\iffy}(x)} - 0\right|
					< \tilde{\varepsilon},
			\end{align*}
			which by Definition~\ref{def:finlimits} is equivalent to
			(\ref{eq:pern4-goal}) and was to be shown.
		}
		\item[(\ref{eq:def:pern5})]{
			To show that $b / 0 = \iffy$
			is well-defined we have to show that
			\begin{equation}\label{eq:pern5-goal}
				\lim\limits_{x\uparrow\iffy}{\left(
				\frac{f_b(x)}{f_0(x)}\right)} = \iffy.
			\end{equation}
			Following from precondition~(\ref{eq:pern-pre-iffy}),
			Definition~\ref{def:inflimits} and
			$\frac{|b|}{2\cdot\tilde{\varepsilon}} > 0$
			we know
			\begin{equation}\label{eq:pern5-2}
				\exists c_{5,0} \in \R :\forall x > c_{5,0}:
					|f_0(x)| < \frac{|b|}
					{2\cdot\tilde{\varepsilon}}
			\end{equation}
			Let $x > \tilde{c}_5 := \max\{c_{3,b},c_{5,0}\}$ to satisfy
			both (\ref{eq:pern3-1}) and (\ref{eq:pern5-2}).
			It follows that
			\begin{align*}
				& |f_b(x)| > \frac{|b|}{2} = \tilde{\varepsilon}\cdot
					\frac{|b|}{2\cdot \tilde{\varepsilon}} >
					\tilde{\varepsilon} \cdot |f_0(x)| \\
				\Rightarrow \quad & |f_b(x)| >\tilde{\varepsilon} \cdot
					|f_0(x)| \\
				\Leftrightarrow \quad &
					\frac{|f_b(x)|}{|f_0(x)|} > \tilde{\varepsilon} \\
				\Leftrightarrow \quad &
					\left|\frac{f_b(x)}{f_0(x)}\right| >
					\tilde{\varepsilon},
			\end{align*}
			which by Definition~\ref{def:finlimits} is equivalent to
			(\ref{eq:pern5-goal}) and was to be shown.\qedhere
		}
	\end{enumerate}
\end{proof}
\section{Open Intervals}\label{sec:oi}
With well-definedness of $\R^*$ shown we have built a solid foundation for
$\R^*$-interval arithmetic.
Given $\R^*$ is not an ordered set, we have to introduce a new
definition for intervals that seamlessly extend to $\iffy$.
Our goal is to define operations on open intervals and singletons
and to use them to model arbitrary subsets of $\R^*$.
\par
To allow degenerate intervals across $\iffy$, the convention proposed
in \cite[pp.~88-89]{re82} is to give $\R^*$ a counter-clockwise
orientation (see Figure~\ref{fig:Rschema})
and define for $\ul{a},\ol{a}\in\R$ and $\ul{a}<\ol{a}$ the degenerate
interval $(\ol{a},\ul{a})$ by tracing all elements from $\ol{a}$ to $\ul{a}$.
It is in our interest to formalise this intuïtive but informal approach.
To denote degenerate intervals, we first need to define the
\begin{definition}[disjoint union]\label{def:du}
	Let $A$ be a set and $\{ A_i \}_{i\in I}$ a family of sets over an index set
	$I$ with $A_i \subseteq A$. $A$ is the disjoint union of $\{ A_i \}_{i\in I}$,
	denoted by
	\begin{equation*}
		A = \bigsqcup_{i\in I}{A_i},
	\end{equation*}
	if and only if
	\begin{equation}
		\forall i,j\in I : i \neq j: A_i \cap A_j = \emptyset
	\end{equation}
	and
	\begin{equation}
		A = \bigcup_{i\in I}{A_i}.
	\end{equation}
\end{definition}
\begin{definition}[open $\R^*$-interval]\label{def:ori}
	Let $\ul{a},\ol{a} \in \R^*$.
	An open $\R^*$-interval between
	$\ul{a}$ and $\ol{a}$ is defined as
	\begin{equation*}
		\R^* \supset (\ul{a},\ol{a}) :=
		\begin{cases}
			\R & \ul{a} = \ol{a} = \iffy\\
			\left\{ x \in \R \, \big| \, x < \ol{a} \right\} &
				\ul{a} = \iffy \\
			\left\{ x \in \R \, \big| \, x > \ul{a} \right\} &
				\ol{a} = \iffy \\
			\left\{ x \in \R \, \big| \, \ul{a} < x < \ol{a}
				\right\} & \ul{a} \le \ol{a} \\
			(\ol{a},\iffy) \sqcup \{ \iffy \} \sqcup
				(\iffy,\ul{a}) & \ul{a} > \ol{a}
		\end{cases}
	\end{equation*}
\end{definition}
In the interest of defining operations on open $\R^*$-intervals,
we introduce the
\begin{definition}[set of open $\R^*$-intervals]\label{def:soori}
	The set of open $\R^*$-intervals is defined as
	\begin{equation}
		\I := \{ (\ul{a},\ol{a}) \, \big| \, \ul{a},\ol{a}\in\R^* \}.
	\end{equation}
	with the operations $\oplus\colon \I \times \I \to \I$ defined
	as
	\begin{subnumcases}{\left( (\ul{a},\ol{a}), (\ul{b},\ol{b})
	                    \right) \mapsto}		
		\begin{cases}
			\emptyset & \ul{a} \in \R \\
			\emptyset & \ul{b},\ol{b}\in\R \wedge \ul{b} \ge \ol{b} \\
			\R & \text{else}
		\end{cases} & $\ul{a} = \ol{a}$ \label{eq:def:soori:1}
		\\
		(\ul{b},\ol{b})\oplus(\ul{a},\ol{a}) &
			$\ul{b} = \ol{b}$ \label{eq:def:soori:2}
		\\
		(\iffy, \ol{a}+\ol{b}) & $\ul{a}=\ul{b}=\iffy$ \label{eq:def:soori:3}
		\\
		(\ul{a}+\ul{b},\iffy) & $\ol{a}=\ol{b}=\iffy$ \label{eq:def:soori:4}
		\\
		\R & $\ul{a}=\ol{b}=\iffy$ \label{eq:def:soori:5}
		\\
		(\ul{b},\ol{b})\oplus(\ul{a},\ol{a}) &
			$\ol{a}=\ul{b}=\iffy$ \label{eq:def:soori:6}
		\\
		\begin{cases}
			\emptyset & \ul{b}>\ol{b} \\
			(\iffy,\ol{a}+\ol{b}) & \text{else}
		\end{cases} & $\ul{a} = \iffy$ \label{eq:def:soori:7}
		\\
		\begin{cases}
			\emptyset & \ul{b}>\ol{b} \\
			(\ul{a}+\ul{b},\iffy) & \text{else}
		\end{cases} & $\ol{a} = \iffy$ \label{eq:def:soori:8}
		\\
		(\ul{b},\ol{b})\oplus(\ul{a},\ol{a}) & $\ul{b} = \iffy$ \label{eq:def:soori:9}
		\\
		(\ul{b},\ol{b})\oplus(\ul{a},\ol{a}) & $\ol{b} = \iffy$ \label{eq:def:soori:10}
		\\
		\begin{cases}
			\emptyset & \ul{a} > \ol{a} \wedge \ul{b} > \ol{b} \\
			(\ul{a}+\ul{b}, \ol{a}+\ol{b}) & else
		\end{cases} & else \label{eq:def:soori:11}
	\end{subnumcases}
	and, using $\ul{A} := \{ \ul{a}\cdot\ul{b},\ul{a}\cdot\ol{b}\}$,
	$\ol{A} := \{ \ol{a}\cdot \ul{b}, \ol{a}\cdot\ol{b} \}$ and
	$A := \ul{A} \cup \ol{A}$ for $\ul{a},\ol{a},
	\ul{b},\ol{b}\in\R$, $\otimes\colon \I \times \I \to \I$
	defined as
	\begin{subnumcases}{\left( (\ul{a},\ol{a}), (\ul{b},\ol{b})
	                 \right) \mapsto}
		\begin{cases}
			\emptyset & \ul{a} \in \R \\
			\emptyset & \ul{b},\ol{b}\in\R \wedge \ul{b} \ge \ol{b} \\
			\R & \text{else}
		\end{cases} & $\ul{a} = \ol{a}$ \label{eq:def:soori:12}
		\\
		(\ul{b},\ol{b}) \otimes (\ul{a},\ol{a}) &
			$\ul{b}=\ol{b}$ \label{eq:def:soori:13}
		\\
		\begin{cases}
			(\ol{a} \cdot \ol{b}, \iffy) &
				\ol{a} \leq 0 \wedge \ol{b} \leq 0 \\
			\R & \text{else}
		\end{cases} & $\ul{a}=\ul{b}=\iffy$ \label{eq:def:soori:14}
		\\
		\begin{cases}
			(\ul{a} \cdot \ul{b}, \iffy) &
				\ul{a} \geq 0 \wedge \ul{b} \geq 0 \\
			\R & \text{else}
		\end{cases} & $\ol{a} = \ol{b} = \iffy$ \label{eq:def:soori:15}
		\\
		\begin{cases}
			(\iffy, \ol{a}\cdot \ul{b}) &
				\ol{a} \leq 0 \wedge \ul{b} \geq 0 \\
			\R & \text{else}
		\end{cases} & $\ul{a} = \ol{b} = \iffy$ \label{eq:def:soori:16}
		\\
		(\ul{b},\ol{b}) \otimes (\ul{a},\ol{a}) &
			$\ol{a} = \ul{b} = \iffy$ \label{eq:def:soori:17}
		\\
		\begin{cases}
			\R & \ul{b}>\ol{b} \\
			(\iffy, \max(\ol{A}))
				& \ul{b} \ge 0 \\
			(\min(\ol{A}),\iffy)
				& \ol{b} \le 0 \\
			\R & \text{else}
		\end{cases} & $\ul{a} = \iffy$ \label{eq:def:soori:18}
		\\
		\begin{cases}
			\R & \ul{b}>\ol{b} \\
			(\min(\ul{A}),\iffy)
				& \ul{b} \ge 0 \\
			(\iffy, \max(\ul{A}))
				& \ol{b} \le 0 \\
			\R & \text{else}
		\end{cases} & $\ol{a} = \iffy$ \label{eq:def:soori:19}
		\\
		(\ul{b},\ol{b}) \otimes (\ul{a},\ol{a}) &
			$\ul{b} = \iffy$ \label{eq:def:soori:20}
		\\
		(\ul{b},\ol{b}) \otimes (\ul{a},\ol{a}) &
			$\ol{b} = \iffy$ \label{eq:def:soori:21}
		\\
		\emptyset & $\ul{a}>\ol{a} \wedge \ul{b}>\ol{b}\quad$ \label{eq:def:soori:22}
		\\
		\begin{cases}
			(\max(\ul{A}),\min(\ol{A})) & \sgn(\ul{b})=\sgn(\ol{b}) \\
			\emptyset & \text{else}
		\end{cases} & $\ul{a}>\ol{a}$ \label{eq:def:soori:23}
		\\
		(\ul{b},\ol{b}) \otimes (\ul{a},\ol{a}) & $\ul{b}>\ol{b}$ \label{eq:def:soori:24}
		\\
		\emptyset & $\ul{a}=\ol{a} \lor \ul{b}=\ol{b}\qquad\quad$ \label{eq:def:soori:25}
		\\
		(\min(A),\max(A)) & else \label{eq:def:soori:26}
	\end{subnumcases}	
\end{definition}
\begin{remark}[role of empty set in definition]
	The use of the empty set in
	Definition~\ref{def:soori} denotes cases where
	undefined behaviour occurs.
\end{remark}
\begin{theorem}[well-definedness of $\I$]\label{theo:wdoi}
	$\I$ is well-defined in terms of set theory.
\end{theorem}
\begin{proof}
	One can see that the operations $\oplus$ and
	$\otimes$ satisfy closedness with regard to $\I$.
	Symmetry is also satisfied given the explicit
	transposed forms
	(\ref{eq:def:soori:2}), (\ref{eq:def:soori:6}),
	(\ref{eq:def:soori:9}) and (\ref{eq:def:soori:10})
	for $\oplus$ and
	(\ref{eq:def:soori:13}), (\ref{eq:def:soori:17}),
	(\ref{eq:def:soori:20}), (\ref{eq:def:soori:21})
	and (\ref{eq:def:soori:24}) for $\otimes$.
	\par
	Well-definedness in terms of set theory is based
	on the condition that for given $A,B\in\I$ the two
	operations $\oplus$ and $\otimes$ must satisfy
	\begin{equation*}
		A \oplus B = \left\{ a+b
		\, | \, a \in A \wedge
		b \in B  \right\}
	\end{equation*}
	and
	\begin{equation*}
		A \otimes B = \{ a \cdot b \, | \, a \in A \wedge
		b \in B  \}
	\end{equation*}
	respectively, except for cases where undefined behaviour
	occurs. It follows from the conditions that if either
	$A=\emptyset$ or $B=\emptyset$ the resulting set is also
	empty (see (\ref{eq:def:soori:1}) and
	(\ref{eq:def:soori:12})).
	\par
	Let $a,b\in\I$ and $\ul{a},\ol{a},\ul{b},\ol{b}\in\R$.
	\begin{enumerate}
		\item[(\ref{eq:def:soori:1})]{
			This case either corresponds to
			\begin{equation*}
				\emptyset\oplus b,
			\end{equation*}
			yielding the empy set, or
			\begin{equation*}
				\R \oplus b,
			\end{equation*}
			yielding $\R$, unless $b$
			is degenerate, given it contains $\iffy$
			and $\R^*\notin\I$ is undefined, or empty,
			yielding the empty set.
		}
		\item[(\ref{eq:def:soori:3})]{
			This case corresponds to
			\begin{equation*}
				(\iffy, \ol{a}) \oplus (\iffy,\ol{b})
			\end{equation*}
			and yields, using Definition~\ref{def:ori},
			\begin{equation*}
				\{ x \in \R \, | \, x < \ol{a} \} \oplus
				\{ x \in \R \, | \, x < \ol{b} \} =
				\{ x \in \R \, | \, x < \ol{a} + \ol{b} \} =
				(\iffy, \ol{a} + \ol{b}).
			\end{equation*}
		}
		\item[(\ref{eq:def:soori:4})]{
			This case corresponds to
			\begin{equation*}
				(\ul{a},\iffy) \oplus (\ol{a},\iffy)
			\end{equation*}
			and yields, using Definition~\ref{def:ori},
			\begin{equation*}
				\{ x \in \R \, | \, x > \ul{a} \} \oplus
				\{ x \in \R \, | \, x > \ul{b} \} =
				\{ x \in \R \, | \, x > \ul{a} + \ul{b} \} =
				(\ul{a} + \ul{b}, \iffy).
			\end{equation*}
		}
		\item[(\ref{eq:def:soori:5})]{
			This case corresponds to
			\begin{equation*}
				(\iffy,\ol{a}) \oplus (\ul{b},\iffy)
			\end{equation*}
			and yields, using Definition~\ref{def:ori},
			\begin{equation*}
				\{ x \in \R \, | \, x < \ol{a} \} \oplus
				\{ x \in \R \, | \, x > \ul{b} \} =	\R.
			\end{equation*}
		}
		\item[(\ref{eq:def:soori:7})]{
			This case corresponds to
			\begin{equation*}
				(\iffy,\ol{a}) \oplus (\ul{b},\ol{b})
			\end{equation*}
			and yields, using Definition~\ref{def:ori},
			if $(\ul{b},\ol{b})$ is degenerate
			\begin{equation*}
				\{ x \in \R \, | \, x < \ol{a} \} \oplus
				((\ol{b},\iffy) \sqcup \{ \iffy \} \sqcup
				(\iffy,\ul{b})) = \R^*
			\end{equation*}
			and, thus, the empty set as $\R^*\notin\I$ is
			undefined, or else
			\begin{equation*}
				\{ x \in \R \, | \, x < \ol{a} \} \oplus
				\{ x \in \R \, | \, \ul{b} < x < \ol{b} \} =
				\{ x \in \R \, | \, x < \ol{a} + \ol{b} \} =
				(\iffy, \ol{a} + \ol{b}).
			\end{equation*}
		}
		\item[(\ref{eq:def:soori:8})]{
			This case corresponds to
			\begin{equation*}
				(\ul{a},\iffy) \oplus (\ul{b},\ol{b})
			\end{equation*}
			and yields, using Definition~\ref{def:ori},
			if $(\ul{b},\ol{b})$ is degenerate
			\begin{equation*}
				\{ x \in \R \, | \, x > \ul{a} \} \oplus
				((\ol{b},\iffy) \sqcup \{ \iffy \} \sqcup
				(\iffy,\ul{b})) = \R^*
			\end{equation*}
			and, thus, the empty set as $\R^*\notin\I$ is
			undefined, or else
			\begin{equation*}
				\{ x \in \R \, | \, x > \ul{a} \} \oplus
				\{ x \in \R \, | \, \ul{b} < x < \ol{b} \} =
				\{ x \in \R \, | \, x > \ul{a} + \ul{b} \} =
				(\ul{a} + \ul{b}, \iffy).
			\end{equation*}
		}
		\item[(\ref{eq:def:soori:11})]{
			This case corresponds to
			\begin{equation*}
				(\ul{a},\ol{a}) \oplus (\ul{b},\ol{b})
			\end{equation*}
			and yields, using Definition~\ref{def:ori}, if both
			$(\ul{a},\ol{a})$ and $(\ul{b},\ol{b})$ are
			degenerate
			\begin{equation*}
				((\ol{a},\iffy) \sqcup \{ \iffy \} \sqcup
				(\iffy,\ul{a})) \oplus
				((\ol{b},\iffy) \sqcup \{ \iffy \} \sqcup
				(\iffy,\ul{b}))
			\end{equation*}
			the empty set, as $\iffy+\iffy$ is undefined.
			If, without loss of generality, only
			$(\ul{a},\ol{a})$ is degenerate, it yields
			\begin{equation*}
				((\ul{a},\iffy) \sqcup \{ \iffy \} \sqcup
				(\iffy,\ol{a})) \oplus
				(\ul{b},\ol{b})=
				(\ul{a}+\ul{b},\iffy) \sqcup \{ \iffy \} \sqcup
				(\iffy,\ol{a}+\ol{b})=(\ol{a}+\ol{b},\ol{a}+\ol{b}).
			\end{equation*}
			If neither $(\ul{a},\ol{a})$ nor $(\ul{b},\ol{b})$ are
			degenerate, it yields
			\begin{equation*}
				\{ x \in \R \, | \, \ul{a} < x < \ol{a} \} \oplus
				\{ x \in \R \, | \, \ul{b} < x < \ol{b} \} =
				\{ x \in \R \, | \, \ul{a}+\ul{b} < x < \ol{a}+\ol{b} \} =
				(\ul{a}+\ul{b},\ol{a}+\ol{b}).
			\end{equation*}
		}
	\end{enumerate}
	The cases (\ref{eq:def:soori:12}), (\ref{eq:def:soori:14}),
	(\ref{eq:def:soori:15}), (\ref{eq:def:soori:16}),
	(\ref{eq:def:soori:18}), (\ref{eq:def:soori:19}),
	(\ref{eq:def:soori:22}), (\ref{eq:def:soori:23}),
	(\ref{eq:def:soori:25}) and
	(\ref{eq:def:soori:26}) for $\otimes$ are shown analogously.
\end{proof}
Given the complexity of open interval arithmetic alone, it
becomes clear why open intervals have been studied independently up
to this point.
We will now expand $\I$ with singletons and introduce the concept
of $\R^*$-Flakes.
\section{Flakes}\label{sec:flakes}
To model subsets of $\R^*$, one easily finds that open intervals
alone are not sufficient to model even simple sets. Using
singletons to expand $\I$ can present new possibilities.
Before we introduce the central concept of this chapter, we
first need to formalise the definition of singletons in $\R^*$.
\begin{definition}[set of singletons]\label{def:sos}
	Let $S$ be a set. The set of $S$-singletons is defined as
	\begin{equation*}
		\S(S) := \left\{ \{ x \} : x \in S \right\}.
	\end{equation*}
\end{definition}
Now we proceed to define the expansion of $\I$ with $\R^*$-singletons
as the 
\begin{definition}[set of $\R^*$-Flakes]\label{def:sorf}
	Let $a,b\in\F$. The set of $\R^*$-Flakes is defined as
	\begin{equation*}
		\F := \I \sqcup \S(\R^*).
	\end{equation*}
	To simplify notation, set the correspondences for
	$\ul{a},\ol{a},\tilde{a},\ul{b},\ol{b},\tilde{b}\in\R^*$
	\begin{align*}
		a \in \I &\quad\leftrightarrow\quad a = (\ul{a},\ol{a}) \\
		a \in \S(\R^*) &\quad\leftrightarrow\quad a = \{ \tilde{a} \} \\
		b \in \I &\quad\leftrightarrow\quad b = (\ul{b},\ol{b}) \\
		b \in \S(\R^*) &\quad\leftrightarrow\quad b = \{ \tilde{b} \}
	\end{align*}
	and use them to define the operations $\boxplus\colon \F \times \F \to \F$ defined as
	\begin{subnumcases}{(a,b) \mapsto}
		a \oplus b & $a,b \in \I$ \label{eq:def:sorf:1}\\
		\begin{cases}
			\emptyset & \tilde{a} = \tilde{b} = \iffy \\
			\{ \tilde{a} + \tilde{b} \} & \text{else}
		\end{cases} & $a,b \in \S(\R^*)$ \label{eq:def:sorf:2}\\
		\begin{cases}
			\begin{cases}
				\emptyset & \ul{b} \ge \ol{b} \\
				\{\iffy\} & \text{else}
			\end{cases} & \tilde{a} = \iffy \\
			\emptyset & \ul{b}=\ol{b} \\
			(\tilde{a} + \ul{b}, \tilde{a} + \ol{b})
			& \text{else}
		\end{cases} & $a \in \S(\R^*) \wedge b \in \I$\label{eq:def:sorf:3}\\
		b \boxplus a & $a \in \I \wedge b \in \S(\R^*)$ \label{eq:def:sorf:4}
	\end{subnumcases}
	and, using $A=\{\tilde{a} \cdot \ul{b},\tilde{a}\cdot\ol{b}\}$
	for $\tilde{a},\ul{b},\ol{b}\in\R$,
	$\boxtimes\colon \F \times \F \to \F$
	defined as
	\begin{subnumcases}{(a,b) \mapsto}
		a \otimes b & $a,b\in\I$ \label{eq:def:sorf:5}\\
		\begin{cases}
			\emptyset & \tilde{a} = \iffy \wedge
				\tilde{b} \in \{0,\iffy\} \\
			\emptyset & \tilde{a} \in \{ 0, \iffy \}
				\wedge \tilde{b} = \iffy \\
			\{ \tilde{a} \cdot \tilde{b}\} & \text{else}
		\end{cases} & $a,b \in \S(\R^*)$ \label{eq:def:sorf:6}\\
		\begin{cases}
			\begin{cases}
				\begin{cases}
					\{\iffy\} & \ol{b} < 0 \\
					\emptyset & \text{else}
				\end{cases} & \ul{b} = \iffy \\
				\begin{cases}
					\{\iffy\} & \ul{b} > 0 \\
					\emptyset & \text{else}
				\end{cases} & \ol{b} = \iffy \\
				\emptyset & \ul{b} > \ol{b} \\
				\{\iffy\} & \sgn(\ul{b})=\sgn(\ol{b}) \\
				\emptyset & \text{else}
			\end{cases} & \tilde{a} = \iffy \\
			\begin{cases}
				(\iffy, \tilde{a} \cdot \ol{b}) & \tilde{a} > 0 \\
				(\tilde{a} \cdot \ol{b}, \iffy) & \tilde{a} < 0 \\
				\emptyset & \text{else}
			\end{cases} & \ul{b} = \iffy \\
			\begin{cases}
				(\tilde{a} \cdot \ul{b}, \iffy) & \tilde{a} > 0 \\
				(\iffy, \tilde{a} \cdot \ul{b})  & \tilde{a} < 0 \\
				\emptyset & \text{else}
			\end{cases} & \ol{b} = \iffy \\
			(\max(A), \min(A)) & \ul{b} > \ol{b} \\
			\emptyset & \ul{b} = \ol{b} \\
			(\min(A), \max(A)) & \text{else}
		\end{cases} & $a \in \S(\R^*) \wedge b\in\I$ \label{eq:def:sorf:7}\\
		b \boxtimes a & $a\in\I\wedge b \in \S(\R^*)\qquad\quad$ \label{eq:def:sorf:8}
	\end{subnumcases}
	The inverse element of $a \in \F$ for $\boxplus$ is defined as
	\begin{equation*}
		-a := \begin{cases}
			\{-\tilde{a}\} & a \in \S(\R^*) \\
			\begin{cases}
				\emptyset & a = \emptyset \\
				(-\ol{a},-\ul{a}) & \text{else}
			\end{cases} & a \in \I
		\end{cases}
	\end{equation*}
	and the inverse element of $a \in \F$ for $\boxtimes$ is defined
	as
	\begin{equation*}
		/a := \begin{cases}
			\{\tilde{a}^{-1}\} & a \in \S(\R^*) \\
			\begin{cases}
				\emptyset & a = \emptyset \\
				(\ol{a}^{-1},\ul{a}^{-1}) & \text{else}
			\end{cases} & a \in \I.
		\end{cases}
	\end{equation*}
\end{definition}
While this definition is definitely complex, we can see that going
step by step and first defining operations on open $\R^*$-intervals
alone makes it easier to prove well-definedness of those operations
as a whole. It shall be noted here that $\R^*$-Flakes allow us to model
closed and open sets on $\R^*$ easily.
\begin{theorem}[well-definedness of $\F$]\label{theo:wdoff}
	$\F$ is well-defined in terms of set theory.
\end{theorem}
\begin{proof}
	One can see that the operations $\boxplus$ and
	$\boxtimes$ satisfy closedness with regard to $\F$.
	Symmetry is also satisfied given the explicit
	transposed forms
	(\ref{eq:def:sorf:4})
	for $\boxplus$ and
	(\ref{eq:def:sorf:8})
	for $\boxtimes$ and the fact that we have shown in
	Theorem~\ref{theo:wdoi} that $\oplus$ and $\otimes$
	are symmetric.
	\par
	Well-definedness in terms of set theory is based
	on the condition that for given $A,B\in\F$ the two
	operations $\boxplus$ and $\boxtimes$ must satisfy
	\begin{equation*}
		A \boxplus B = \{ a + b \, | \, a \in A \wedge
		b \in B  \}
	\end{equation*}
	and
	\begin{equation*}
		A \boxtimes B = \{ a \cdot b \, | \, a \in A \wedge
		b \in B  \}
	\end{equation*}
	respectively, except for cases where undefined behaviour
	occurs.
	\par
	Let $a,b\in\F$ as in Definition~\ref{def:sorf}.
	\begin{enumerate}
		\item[(\ref{eq:def:sorf:1})]{
			We have shown in Proposition~\ref{theo:wdoi} that
			$\boxplus$ is well-defined in terms of set theory.
		}
		\item[(\ref{eq:def:sorf:2})]{
			This case corresponds to
			\begin{equation*}
				\{ \tilde{a} \} \boxplus \{ \tilde{b} \}
			\end{equation*}
			and yields
			\begin{equation*}
				\{ \tilde{a} + \tilde{b} \}
			\end{equation*}
			unless $\tilde{a} = \tilde{b} = \iffy$, which
			is undefined, where the empty set is returned.
		}
		\item[(\ref{eq:def:sorf:3})]{
			This case corresponds to
			\begin{equation*}
				\{ \tilde{a} \} + (\ul{b},\ol{b})
			\end{equation*}
			and yields $\{ \iffy \}$ if $\tilde{a}=\iffy$
			and $b$ is not degenerate or empty, which
			yields the empty set.
			If $\tilde{a}\in\R$, it yields
			\begin{equation*}
				(\tilde{a}+\ul{b},\tilde{a}+\ol{b}),
			\end{equation*}
			for degenerate and non-degenerate $b$, unless $b$
			is empty, which yields the empty set.
		}
	\end{enumerate}
	The cases (\ref{eq:def:sorf:5}), (\ref{eq:def:sorf:6}) and
	(\ref{eq:def:sorf:7}) for $\boxtimes$ are shown analogously.
	\par
	What remains to be shown is that the inverse elements are well-defined.
	One can see that the inverse elements are all closed under $\F$ and
	map $\emptyset$ to $\emptyset$.
	We now have to show that the operation of an element in $\F$ with its
	respective inverse element results in a set containing the respective
	neutral elements of $\R^*$ except where undefined behaviour occurs.
	\par
	For $\boxplus$ with \enquote{$-$} and $\boxtimes$ and \enquote{$/$} we
	observe for singletons
	\begin{align*}
		\{ \tilde{a} \} \boxplus -\{ \tilde{a} \} &=
			\{ \tilde{a} \} \boxplus \{ -\tilde{a} \} =
			\{ \tilde{a} - \tilde{a} \} = \{ 0 \}
			\ni 0 \\
		\{ \tilde{a} \} \boxtimes /\{ \tilde{a} \} &=
			\{ \tilde{a} \} \boxtimes \{ \tilde{a}^{-1} \} =
			\begin{cases}
				\emptyset & \tilde{a} = \iffy \\
				\{ 1 \} \ni 1 & \text{else}.
			\end{cases}
	\end{align*}
	Analogously, we observe for open $\R^*$-intervals with
	$(\ul{a},\ol{a})\neq\emptyset$
	\begin{align*}
		(\ul{a},\ol{a}) \boxplus -(\ul{a},\ol{a}) &=
			(\ul{a},\ol{a}) \boxplus (-\ol{a},-\ul{a}) =
			\begin{cases}
				\R \ni 0 & \ul{a} = \iffy \lor \ol{a} = \iffy \\
				\emptyset & \ul{a} > \ol{a} \\
				(\ul{a}-\ol{a},\ol{a}-\ul{a}) \ni 0 & \text{else}
			\end{cases} \\
		(\ul{a},\ol{a}) \boxtimes /(\ul{a},\ol{a}) &=
			(\ul{a},\ol{a}) \boxtimes ({\ol{a}}^{-1},{\ul{a}}^{-1}) =
			\begin{cases}
				\R \ni 1 & \ul{a} = \iffy \lor \ol{a} = \iffy \\
				\emptyset & \ul{a} > \ol{a} \\
				(\frac{\ul{a}}{\ol{a}},\frac{\ol{a}}{\ul{a}})\ni 1
					& \text{else}.
			\end{cases}
	\end{align*}
	It follows the well-definedness of the inverse elements.
\end{proof}
Now that we have shown well-definedness of $\F$, we can proceed
with showing some useful properties that allow easier generalisations
on Flakes. One of them is the
\begin{definition}[$\R^*$-Flake evaluation of strictly increasing functions]
	\label{def:rfeosif}
	Let $f\colon\R\to\R$ be strictly increasing. The $\R^*$-Flake
	evaluation of $f$
	\begin{equation*}
		f_{\F} \colon \F \to \F
	\end{equation*}
	is defined with the notation $f(\iffy):=\iffy$ as
	\begin{equation*}
		a \mapsto
		\begin{cases}
			\{ f(\tilde{a}) \} & a = \{ \tilde{a} \} \in \S(\R^*) \\
			(f(\ul{a}),f(\ol{a})) & a = (\ul{a},\ol{a}) \in \I.
		\end{cases}
	\end{equation*}
\end{definition}
\begin{proposition}[well-definedness of $\bullet_{\F}$]
	The $\R^*$-Flake evaluation of strictly increasing functions is well-defined
	in terms of set theory.
\end{proposition}
\begin{proof}
	Let $f\colon\R\to\R$ be strictly increasing. We can see that
	$f_{\F}$ is closed in $\F$ and maps $\emptyset$ to $\emptyset$.
	For singletons well-definedness follows immediately, as it just
	corresponds to the singleton of the single function evaluation of
	$f$. In this context, $f(\iffy)=\iffy$, treating $\iffy$ as an
	invariant object, is also consistent with the axioms of
	Definition~\ref{def:pern}, as
	\begin{equation*}
		\lim_{x\uparrow\iffy}(f(x)) = \lim_{x\downarrow\iffy}(f(x))
		= \iffy.
	\end{equation*}
	For non-degenerate non-empty open $\R^*$-intervals the bounds
	grow accordingly, as
	\begin{equation*}
		\forall \ul{a},\ol{a}\in\R: \quad \ul{a} < \ol{a}
		\quad \Leftrightarrow \quad f(\ul{a}) < f(\ol{a}).
	\end{equation*}
	This also implies the well-definedness of the degenerate case, as
	for $\ul{a},\ol{a}\in\R$ and $\ul{a} > \ol{a}$ it holds that
	\begin{align*}
		f_{\F}((\ul{a},\ol{a})) &=
			f_{\F}((\ol{a},\iffy)) \sqcup f_{\F}(\{ \iffy \}) \sqcup
			f_{\F}((\iffy,\ul{a})) \\
		&= (f_{\F}(\ol{a}), \iffy) \sqcup \{ \iffy \} \sqcup
			(\iffy, f_{\F}(\ul{a})) \\
		&= (f_{\F}(\ul{a}),f_{\F}(\ol{a})). \qedhere
	\end{align*}
\end{proof}
\begin{definition}[$\R^*$ Flake evaluation of strictly decreasing functions]
	\label{def:rfeosdf}
	Let $f\colon\R\to\R$ be strictly decreasing. The $\R^*$ Flake
	evaluation of $f$
	\begin{equation*}
		f_{\F}\colon\F\to\F
	\end{equation*}
	is defined as
	\begin{equation*}
		a \mapsto -((-f)_{\F}(a)).
	\end{equation*}
\end{definition}
With these results we have shown in general that we can evaluate
strictly monotonic functions on $\R^*$-Flakes, for instance
$\exp$ or $\ln$ confined to $\R^{+}_{\neq 0}$, which will be used later.
We require strictly monotonic functions, as a constant function
$f(x)=c \in \R$, that is monotonic but not strictly monotonic, would yield
\begin{equation*}
	f_\F((1,2)) = (f(1),f(2)) = (c,c) = \emptyset,
\end{equation*}
which is not well-defined in terms of set theory.
\par
Using the results obtained in this Chapter, we can now examine a discrete
set of Unums as a subset of $\F$. This especially allows us to use those now
well-defined operations and identify them on the set of Unums,
provided we choose it properly.
\chapter{Unum Arithmetic}\label{ch:ua}
This Chapter will construct the Unum arithmetic based on the results in
Chapter~\ref{ch:ia} and the publications \cite{gu16-a} and \cite{gu16-b}
by \textsc{Gustafson}.
We start off by examining the
\begin{definition}[set of Unums]\label{def:sou}
	Let
	\begin{equation*}
		P=\{ p_1, \ldots, p_n \, | \, \forall i < j:
		p_i < p_j\} \subset (1,\iffy),
	\end{equation*}
	$p_0 := 1$ and $p_{n+1} := \iffy$.
	The set of Unums on the lattice $P$ is defined as
	\begin{align*}
		\F \supset \U(P) := &
		\bigsqcup_{i=1}^{n}{\big[
			\{ p_i \} \sqcup /\{  p_i \}
			\sqcup
			-\{ p_i \} \sqcup -/\{ p_i \}
		\big]}
		\sqcup \\ &
		\bigsqcup_{i=0}^{n}{\big[
			\{ (p_i, p_{i+1}) \} \sqcup \{ /(p_i ,p_{i+1}) \}
			\sqcup
			\{ -(p_i, p_{i+1})\} \sqcup \{ -/(p_i, p_{i+1}) \}
		\big]} \sqcup \\ &
		\{ 1 \} \sqcup \{ -1 \} \sqcup
		\{ 0 \} \sqcup \{ \iffy \}
	\end{align*}
\end{definition}
\begin{remark}\label{rem:unumclosed}
	By Definition~\ref{def:sou}, $\U$ is closed
	under inversion with regard to $\boxplus$ and $\boxtimes$.
\end{remark}
In regard to $\F$, Remark~\ref{rem:unumclosed} underlines
the fact that this choice for $\U$, generated
by a set of lattice points between $(1,\iffy)$, is in fact
a good one.
We will now proceed to derive some elemental properties of
$\U$ and prepare it to define operations on it.
\begin{proposition}[cardinality of $\U$]\label{prop:cou}
	Let $P$ as in Definition~\ref{def:sou}. The number of Unums is
	\begin{equation*}
		|\U | = 8 \cdot (|P| + 1).
	\end{equation*}
\end{proposition}
\begin{proof}
	Each quadrant of $\R^*$ is filled with $|P|$ lattice points and
	$|P| + 1$ intervals. Added to this are the $4$ fixed points
	$1$, $-1$, $0$, $\iffy$. It follows from Definition~\ref{def:sou}
	of $|\U |$ as a disjoint union of finite sets that
	\begin{equation*}
		|\U | = 4 \cdot |P| + 4 \cdot (|P| + 1) + 4 = 4 \cdot
		(2 \cdot |P| + 2) = 8 \cdot (|P| + 1).\qedhere
	\end{equation*}
\end{proof}
Before we proceed with constructing operations on the set of Unums,
we first have to define the
\begin{definition}[power set]\label{def:ps}
	Let $S$ be a set. The power set of $S$ is defined as
	\begin{equation*}
		\power(S) := \{ s \subseteq S \}.
	\end{equation*}
\end{definition}
To use the results we have derived for $\F$, we need to find a way to
\enquote{blur} $\R^*$-Flakes into sets of Unums. For this purpose, we define the
\begin{definition}[blur operator]\label{def:blur}
	Let $P$ as in Definition~\ref{def:sou}. The blur operator
	\begin{equation*}
		\blur\colon \F \to \power(\U(P))
	\end{equation*}
	is defined as
	\begin{equation*}
		f \mapsto
		\{ u \in \U : f \subseteq u \}.
	\end{equation*}
\end{definition}
We are now able to embed $\R^*$-Flakes into subsets of $\U$, which allows
us to define operations on $\U$ by identifying them with operations
on $\F$ using the $\blur$-operator.
\begin{remark}[dependent sets and dependency problem]
	It is not within the scope of this thesis to elaborate on the
	theory of dependent sets, and there are multiple ways to approach it.
	To give a simple example, evaluating for $A=(-1,1)\in\I$
	\begin{equation*}
		A - A
	\end{equation*}
	is expected to yield $\{ 0 \}$, but using interval arithmetic,
	the expression just decays to
	\begin{equation*}
		(-1,1) - (-1,1) = (-1,1) + (-1,1) = (-2, 2),
	\end{equation*}
	effectively doubling the width of the interval. This is known as
	the \emph{dependency problem}.
	\par
	It is in our interest to find an approach to limit this problem.
	As follows, we will denote two dependent sets $S_1$ and $S_2$ with
	$S_1\sim S_2$,
	and with regard to the example given above, it holds that $A\sim A$.
\end{remark}
To approach the dependency problem, we only evaluate pairwise
operations for dependent sets. The underlying idea is that if a given
value is present in the first set within a Unum, the dependency
guarantees it will also only be within this Unum in the second set.
We identify operations on $\F$ with operations on $\U$ by defining
the
\begin{definition}[dual Unum operation]\label{def:duo}
	Let $\star\colon\F\times\F\to\F$ be an operation on $\F$
	and $P$ as in Definition~\ref{def:sou}.
	The dual Unum operation
	\begin{equation*}
		\langle \star \rangle \colon \power(\U(P)) \times
		\power(\U(P)) \to \power(\U(P))
	\end{equation*}		
	is defined as
	\begin{equation*}
		(U,V) \mapsto
		\bigcup_{u\in U}\bigcup_{v\in V}
		\begin{cases}
			\emptyset        & U \sim V \wedge u \neq v \\
			\R^*             & u \star v = \emptyset \\
			\blur(u \star v) & \text{else}.
		\end{cases}
	\end{equation*}
\end{definition}
\begin{remark}[$\nan$ for Unum operations]
	As one can see in Definition~\ref{def:duo}, when an $\R^*$-Flake
	operation $\star$ yields the empty set, indicating an empty set or
	that undefined behaviour was witnessed, the Unum arithmetic
	proposed by \textsc{Gustafson} in \cite[Table~2]{gu16-b}
	mandates that the respective dual Unum operation yields $\R^*$.
	\par
	This is not the ideal behaviour, as we carefully defined
	$\boxplus$ and $\boxtimes$ to give the empty set if one operand
	is the empty set, $-\emptyset=\emptyset$ and
	$/\emptyset=\emptyset$.
	This behaviour is useful, as just like $\nan$ for floating-point
	numbers, which, once it occurs, is carried through the entire
	stream of floating-point calculations, the empty set plays this
	special role in the Unum context.
	\par
	In the interest of staying compatible with the Unum format proposed by
	\textsc{Gustafson}, this weak spot in the proposal was implemented
	in the Unum toolbox anyway.
\end{remark}
\begin{definition}(Unum evaluation of strictly increasing functions)
	\label{def:ueosif}
	Let $f\colon\R\to\R$ be strictly increasing. The Unum evaluation
	of $f$
	\begin{equation*}
		\langle f_{\F} \rangle \colon \power(\U(P)) \to \power(\U(P))
	\end{equation*}
	is defined as
	\begin{equation*}
		U \mapsto \bigcup_{u\in U} \blur(f_{\F}(u)).
	\end{equation*}
\end{definition}
\begin{definition}[Unum evaluation of strictly decreasing functions]
	\label{def:ueosdf}
	Let $f\colon\R\to\R$ be strictly decreasing. The Unum evaluation of $f$
	\begin{equation*}
		\langle f_{\F} \rangle \colon \power(\U(P)) \to \power(\U(P))
	\end{equation*}
	is defined as
	\begin{equation*}
		U \mapsto \bigcup_{u\in U} \blur(-((-f)_{\F}(u))).
	\end{equation*}
\end{definition}
\section{Lattice Selection}\label{sec:lattice}
Until now, we have worked with arbitrary $P$. This set of lattice points
is the only parametrisation for $\U$, so we want to investigate what the
ideal construction of $P$ is.
\subsection{Linear Lattice}
The simplest approach is a linear distribution of $p$ lattice points
up to a maximum value $m\in(1,\iffy)$.
\begin{definition}[linear Unum lattice]\label{def:lul}
	Let $p\in\N$ and $m\in(1,\iffy)$. The linear Unum lattice with
	$p$ lattice points and maximum $m$ is defined as
	\begin{equation*}
		P_L(p, m) := \left\{ p_i := 1 + i \cdot \frac{m-1}{p}
		\, \bigg| \, i \in \{1,\ldots,p\} \right\}.
	\end{equation*}
\end{definition}
\begin{proposition}[well-definedness of the linear Unum lattice]
	Let $p\in\N$ and $m\in(1,\iffy)$. $P_L(p,m)$ is well-defined in terms
	of Definition~\ref{def:sou}.
\end{proposition}
\begin{proof}
	The desired properties $|P_L(p,m)|=p$ and $\max(P_L(p,m))=m$ follow
	from Definition~\ref{def:lul}. We show that
	\begin{equation*}
		\forall i > j: p_i > p_j.
	\end{equation*}
	This is given because $m - 1 > 0$ and
	\begin{equation*}
		p_i - p_j = (i - j) \cdot \frac{m-1}{p} > 0.
	\end{equation*}
	The proof is finished by showing that $p_i \in (1,\iffy)$. It suffices
	to prove that $p_1,p_p\in(1,\iffy)$, as $\forall i > j: p_i > p_j$
	and the boundary points dictate the behaviour of the interior
	points.
	\begin{align*}
		p_1 & = 1 + \frac{m-1}{p} \in (1,\iffy) \\
		p_p & = 1 + m - 1 = m \in (1,\iffy) \qedhere
	\end{align*}
\end{proof}
The problem with a linear Unum lattice is the lack of dynamic range.
Just like with floating-point numbers, we want a dense distribution
of lattice points around $1$ and a lighter distribution the further we
move away from $1$. As we can deduce from this observation, a desired
quality of the Unum lattice could be, for instance, an exponential
distribution.
\subsection{Exponential Lattice}
\begin{definition}[exponential Unum lattice]\label{def:eul}
	Let $p\in\N$ and $m\in(1,\iffy)$. The exponential Unum lattice with
	$p$ lattice points and maximum $m$ is defined as
	\begin{equation*}
		P_E(p, m) := \left\{ p_i :=
		\exp\left(i\cdot\frac{\ln(m)}{p} \right)
		\, \bigg| \, i \in \{1,\ldots,p\} \right\}.
	\end{equation*}
\end{definition}
\begin{proposition}[well-definedness of the exponential Unum lattice]
	Let $p\in\N$ and $m\in(1,\iffy)$. $P_E(p,m)$ is well-defined in terms
	of Definition~\ref{def:sou}.
\end{proposition}
\begin{proof}
	The desired properties $|P_E(p,m)|=p$ and $\max(P_E(p,m))=m$ follow
	from Definition~\ref{def:eul}. We show that
	\begin{equation*}
		\forall i > j: p_i > p_j.
	\end{equation*}
	This is given because $\exp$ is strictly monotonically increasing and
	\begin{equation*}
		p_i-p_j = \exp\left(i\cdot\frac{\ln(m)}{p} \right)-
			\exp\left(j\cdot\frac{\ln(m)}{p} \right) > 0.
	\end{equation*}
	The proof is finished by showing that $p_i \in (1,\iffy)$.
	\begin{equation*}
		p_i = \exp\left(i\cdot\frac{\ln(m)}{p} \right) > \exp(0) = 1
		\qedhere
	\end{equation*}
\end{proof}
The problem of an exponential Unum lattice is that the lattice
points may have an ideal distribution, but fall onto rather
inaccessible points. For such a number system to work, it has to
contain a decent amount of integers, which is not the case here.
\subsection{Decade Lattice}
A different approach is to specify the number of desired significant
decimal digits of each lattice point and fill the set by scaling
with multiples of $10$.
For example, specifying 1 significant digit yields
\begin{equation*}
	P = \{ 2, 3,\ldots, 9, 10, 20, 30, \ldots, 90, 100, 200, 300,
	\ldots \}.
\end{equation*}
We define this formally, using the remainder of
the \textsc{Euclid}ean division of a by b, denoted by $a \bmod b$
for $a\in\N_0$ and $b\in\N$, as the
\begin{definition}[decade Unum lattice]\label{def:dul}
	Let $p\in\N_0$ and $s\in\N$.
	The decade Unum lattice with $p$ lattice points and
	$s$ significant digits is defined as
	\begin{align*}
		P_D(p,s) := \bigg\{
			p_i :=
			\left[ 1 + 10^{-(s-1)} \cdot \left(i \bmod
			(10^s-10^{s-1})\right) \right] \cdot
			10^{\left\lfloor \frac{i}{10^s-10^{s-1}} \right\rfloor}
			\, \bigg|
		\, i \in \{ 1,\ldots,p \}
		\bigg\}.
	\end{align*}
\end{definition}
\begin{proposition}[well-definedness of the decade Unum lattice]
	\label{prop:wdofdul}
	Let $p\in\N_0$ and $s\in\N$. $P_D(p,s)$ is well-defined in terms
	of Definition~\ref{def:sou}.
\end{proposition}
\begin{proof}
	The desired property $|P_E(p,m)|=p$ follows
	from Definition~\ref{def:dul}. We show that
	\begin{equation*}
		\forall i,j\in\{1,\ldots,p\}: i > j: p_i > p_j.
	\end{equation*}
	This is trivial for $p=1$. For $p>1$ and $i\in\{1,\ldots,p-1\}$ we note that
	for $m\in\N$ it holds that
	\begin{align*}
		(i+1) \bmod m = 0 & \Rightarrow
		\left\{\begin{array}{l}
			i \bmod m = m-1 \\
			\exists n \in \N_0: (i + 1) = n \cdot m \\
		\end{array}\right\} \\
		& \Rightarrow
		\left\{\begin{array}{l}
			\left\lfloor \frac{i+1}{m} \right\rfloor =
				\lfloor n \rfloor = n \\[2mm]
			\left\lfloor \frac{i}{m}  \right\rfloor = n - 1
		\end{array}\right\} \\
		& \Rightarrow
		\left\lfloor \frac{i+1}{m} \right\rfloor =
			\left\lfloor \frac{i}{m} \right\rfloor + 1
	\end{align*}
	and obtain
	\begin{align*}
		p_{i+1} - p_i =~& \left[ 1 + 10^{-(s-1)} \cdot \left((i+1) \bmod
			\left(10^s-10^{s-1}\right)\right) \right] \cdot
			10^{\left\lfloor \frac{i+1}{10^s-10^{s-1}} \right\rfloor} - \\
		& \left[ 1 + 10^{-(s-1)} \cdot \left(i \bmod
			\left(10^s-10^{s-1}\right)\right) \right] \cdot
			10^{\left\lfloor \frac{i}{10^s-10^{s-1}} \right\rfloor} \\
		\ge~& \left[ 1 + 10^{-(s-1)} \cdot 0 \right] \cdot
			10^{{\left\lfloor \frac{i}{10^s-10^{s-1}} \right\rfloor}+1} - \\
		& \left[ 1 + 10^{-(s-1)} \cdot \left( 10^s-10^{s-1}-1\right)\right]
			\cdot 10^{\left\lfloor \frac{i}{10^s-10^{s-1}} \right\rfloor} \\
		=~& \left[ 10 - 1 - 10^{-(s-1)+s} + 10^{-(s-1)+s-1} + 10^{-(s-1)} \right]
			\cdot 10^{\left\lfloor \frac{i}{10^s-10^{s-1}} \right\rfloor} \\
		=~& 10^{-(s-1)}\cdot 10^{\left\lfloor \frac{i}{10^s-10^{s-1}} \right\rfloor} \\
		>~& 0.
	\end{align*}
	The proof is finished by showing that $p_i \in (1,\iffy)$.
	\begin{align*}
		p_i & = \left[ 1 + 10^{-(s-1)} \cdot \left(i \bmod
			(10^s-10^{s-1})\right) \right] \cdot
			10^{\left\lfloor \frac{i}{10^s-10^{s-1}}
			\right\rfloor} \\
		& \ge 1 + 10^{-(s-1)} \cdot \left(i \bmod
			(10^s-10^{s-1})\right) \\
		& > 1 \qedhere
	\end{align*}
\end{proof}
\begin{proposition}[maximum of the decade Unum lattice]
	Let $p\in\N_0$ and $s\in\N$. The maximum of the decade Unum lattice is
	\begin{equation*}
		\max\left\{ P_D(p,s) \right\} =
			\left( 1 +
			10^{-(s-1)}\cdot \left[
			p \bmod \left(10^s-10^{s-1}\right)
			\right]\right) \cdot
			10^{\left\lfloor \frac{p}{10^s-10^{s-1}} \right\rfloor}.
	\end{equation*}
\end{proposition}
\begin{proof}
	As shown in the proof of Proposition~\ref{prop:wdofdul},
	$\forall i > j: p_i > p_j$ and thus
	\begin{equation*}
		\max\left\{ P_D(p,s) \right\} = p_p =
		\left( 1 +
		10^{-(s-1)}\cdot \left[
		p \bmod \left(10^s-10^{s-1}\right)
		\right]\right) \cdot
		10^{\left\lfloor \frac{p}{10^s-10^{s-1}} \right\rfloor}.\qedhere
	\end{equation*}
\end{proof}
Comparing the resulting distribution to an exponential curve
fitted to the boundary-points, as shown in Figure~\ref{fig:decade},
one can see that a nearly exponential distribution has been achieved.
As we can see, the decade Unum lattice is a good compromise between a
linear and an exponential Unum lattice.
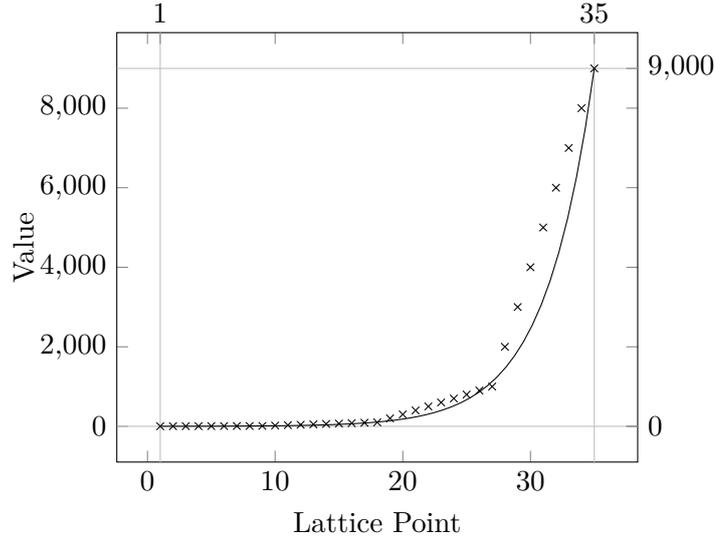
\begin{figure}[htpb]
	\centering
	\begin{tikzpicture}
		\begin{axis}[xlabel=Lattice Point, ylabel=Value,
			         extra x ticks={1,35},
			         extra x tick style={xticklabel pos=right,
			         	xtick pos=right},
			         extra y ticks={0,9000},
			         extra y tick style={yticklabel pos=right,
			         	ytick pos=right},
			         extra tick style={grid=major}]
			\addplot[mark=x, only marks] coordinates {
				(   1,   2) (   2,   3) (   3,   4) (   4,   5)
				(   5,   6) (   6,   7) (   7,   8) (   8,   9)
				(   9,  10) (  10,  20) (  11,  30) (  12,  40)
				(  13,  50) (  14,  60) (  15,  70) (  16,  80)
				(  17,  90) (  18, 100) (  19, 200) (  20, 300)
				(  21, 400) (  22, 500) (  23, 600) (  24, 700)
				(  25, 800) (  26, 900) (  27,1000) (  28,2000)
				(  29,3000) (  30,4000) (  31,5000) (  32,6000)
				(  33,7000) (  34,8000) (  35,9000)
			};
			\addplot[domain=1:35, samples=50] {exp(0.26014*x)};
		\end{axis}
	\end{tikzpicture}
	\caption{$P_D(35,1)$ (demarked by crosses) in comparison with an
		exponential curve fitted to the endpoints $(0,0)$ and
		$(35,\max(P_D(35,1))$.}
	\label{fig:decade}
\end{figure}
\section{Machine Implementation}
The goal of a machine implementation for Unums is to find a model
for $\power(\U(P))$ on a specially chosen lattice $P$. This
means the ability to model subsets of $\R^*$ using multiple Unums,
including degenerate intervals.
\subsection{Unum Enumeration}
We start off with the definition of the
\begin{definition}[ascension operator]\label{def:asc}
	Let $(S,<)$ be a finite strictly ordered set.
	The ascension operator
	\begin{equation*}
		\asc\colon S\times\{1,\ldots,|S|\}\to S
	\end{equation*}
	is defined for
	\begin{equation*}
		s_i\in\{s_i \, | \, i\in\{1,\ldots,|S|\} \wedge
			s_1 < \ldots < s_{|S|}\} = S
	\end{equation*}
	as
	\begin{equation*}
		(S,n)\mapsto s_n
	\end{equation*}
\end{definition}
Using the ascension operator, we enumerate the elements
in $\U(P)$ with $P$ as in Definition~\ref{def:sou},
taking note that $\U(P)\cap \power((0,1))$,
$\U(P)\cap \power((1,\iffy))$, $\U(P)\cap \power((\iffy,-1))$
and $\U(P)\cap \power((-1,0))$ are finite strictly ordered
sets.
In other words, we define a mapping from $\{0,\ldots,|\U(P)|-1\}$,
which is $\{0,\ldots,8\cdot(|P|+1)-1\}$ according to
Proposition~\ref{prop:cou}, into $\U(P)$, called the
\begin{definition}[Unum enumeration]\label{def:ue}
	Let $P$ as in Definition~\ref{def:sou}.
	The Unum enumeration
	\begin{equation*}
		u\colon\{0,\ldots,|\U(P)|-1\}\to\U(P)
	\end{equation*}
	is defined as
	\begin{equation*}
		n \mapsto
		\begin{cases}
			\{ 0 \} & n = 0\cdot(|P|+1) \\
			\asc(\U(P)\cap \power((0,1)), n-0\cdot(|P|+1)) &
				0\cdot(|P|+1) < n < 2\cdot(|P| + 1) \\
			\{ 1 \} & n = 2\cdot(|P|+1) \\
			\asc(\U(P)\cap \power((1,\iffy)), n-2\cdot(|P|+1)) &
				2\cdot(|P|+1) < n < 4 \cdot (|P|+1) \\
			\{ \iffy \} & n = 4\cdot (|P|+1) \\
			\asc(\U(P)\cap \power((\iffy,-1)), n-4\cdot(|P|+1)) &
				4\cdot(|P|+1) < n < 6 \cdot (|P|+1) \\
			\{ -1 \} & n = 6\cdot (|P|+1) \\
			\asc(\U(P)\cap (-1,0), n-6\cdot(|P|+1)) &
				6\cdot(|P|+1) < n < 8 \cdot (|P|+1).
		\end{cases}
	\end{equation*}
\end{definition}
\begin{remark}[enumeration of infinity]
	For arbitrary $\U(P)$ with
	$P$ as in Definition~\ref{def:sou} it follows that
	\begin{equation*}
		u\left(\frac{|\U(P)|}{2}\right) = \{ \iffy \}.
	\end{equation*}
\end{remark}
To describe the enumeration intuïtively, we cut the $\R^*$-circle at
$0$ and trace all Unums from $0$ to $0$ in a counter-clockwise direction.
In the machine the Unum enumeration mapping can be realised using unsigned integers.
One can deduce that for a given number of \emph{Unum bits} $n_b\in\N$ an unsigned
$n_b$-bit integer can represent $2^{n_b}$ values, namely
$0$ through $2^{n_b}-1$.
\par
Even though in theory the size of $|P|$ can be arbitrary, as it is the case
for the provided toolbox, one must respect the fundamental data-types in a
machine, resulting in the limitation $n_b\in\{8,16,32,64,\ldots\}$ in the interest
of not wasting any bit patterns in the process.
It follows that we are interested in finding out the required lattice size
for a given $n_b$.
\begin{proposition}[lattice size depending on Unum bits]
	Let $n_b\in\N$, $n_b>2$ and $P$ as in Definition~\ref{def:sou}. Given $n_b$
	Unum bits it follows that
	\begin{equation*}
		|P|=2^{n_b-3}-1.
	\end{equation*}
\end{proposition}
\begin{proof}
	With $n_b$ Unum bits it follows that $|\U(P)|=2^{n_b}$. According to
	Proposition~\ref{prop:cou} we know that $|\U(P)|=8\cdot(|P|+1)$ and thus
	\begin{equation*}
		2^{n_b}=8\cdot(|P|+1)=2^3\cdot(|P|+1) \quad\Leftrightarrow\quad
		|P|=2^{n_b-3}-1\qedhere
	\end{equation*}
\end{proof}
According to the results obtained in Section~\ref{sec:lattice},
we will only take decade lattices into account. We are led to the
\begin{definition}[set of machine Unums]\label{def:somu}
	Let $n_b\in\N$, $n_b>2$ and $n_s\in\N$. The set of machine Unums
	with $n_b$ bits and $n_s$ significant digits is defined as
	\begin{equation*}
		\U_M(n_b,n_s) := \U(P_D(2^{n_b-3}-1,n_s).
	\end{equation*}
\end{definition}
Having found an expression for machine Unums, it is now possible to
represent arbitrary elements of $\power(\U_M(n_b,n_s))$ in the
machine to model sets of real numbers.
\subsection{Operations on Sets of Real Numbers}\label{subsec:sornops}
Unums alone are not very useful for arithmetic purposes, given the
nature of dual Unum operations (see Definition~\ref{def:duo}), which
we want to illustrate with the following example.
\begin{example}
	Let $P=\{ 2, 3.5, 5, 6\}$, which satisfies Definition~\ref{def:sou}.
	We see that $(1,2),\{3.5\} \in \U(P)$, but
	\begin{equation*}
		\blur\left( (1,2) \boxplus \{3.5\} \right) =
		\blur\left( (4.5,5.5) \right) =
		\{(3.5,5), \{ 5 \}, (5,6)\} \notin \U(P).
	\end{equation*}
\end{example}
The basic datatype, thus, has to be an element of
$\power(\U_M(n_b,n_s))$. Given this set is finite with
$2^{(|\U_M(n_b,n_s)|)}=2^{2^{n_b}}$ elements, a bit string of length
$2^{n_b}$ can represent all elements of $\power(\U_M(n_b,n_s))$.
We call this bit string a \enquote{SORN} for \enquote{set of real numbers}.
\par
Operations on SORNs are carried out in the machine by having
lookup tables (LUTs) for $\blur(u(i)\star u(j))$, where
$\star \in \{\boxplus,\boxtimes\}$ is an $\R^*$-Flake-operation
evaluated for arbitrary Unum-indices $i,j\in\{0,\ldots,2^{n_b}-1\}$.
Given $\boxplus$ and $\boxtimes$ are
associative, limiting the lookup table to $i\le j$ is sufficient,
resulting in a triangular array for each operation.
\par
The results $\blur(u(i)\star u(j))$, being connected subsets of
$\U_M(n_b,n_s)$, can be expressed as an oriented range
$[u(m),u(n)]$ with $m,n\in\{0,\ldots,2^{n_b}-1\}$ and $m\le n$,
containing all Unums between $u(m)$ and $u(n)$.
This can be stored in the machine as indices $\{m,n\}$ each taking
up $n_b$ bit of storage. Thus, each table entry
takes up $2\cdot n_b$ bit of storage.
\begin{proposition}[size of LUTs]
	The Unum LUTs for $\boxplus$ and $\boxtimes$ take up
	$n_b \cdot 2^{n_b+1}\cdot(2^{n_b}+1)$ bit.
\end{proposition}
\begin{proof}
	With $2^{n_b}$ rows, we know that each LUT has
	$\sum_{i=1}^{2^{n_b}}{i}$
	entries. Using the Gauß summation formula and
	the facts that each entry takes up $2\cdot n_b$
	bit and we have two operations and, thus, two
	LUTs, the total storage size is
	\begin{equation*}
		2\cdot (2\cdot n_b) \cdot \left(\sum_{i=1}^{2^{n_b}}{i}\right)
		\text{bit} =
		4 \cdot n_b \cdot \left(\frac{2^{n_b}\cdot(2^{n_b}+1)}{2}\right)
		\text{bit} =
		n_b \cdot 2^{n_b+1}\cdot(2^{n_b}+1) \, \text{bit}.\qedhere
	\end{equation*}
\end{proof}
With the lookup tables constructed, operations on SORNs
are analogous to dual Unum operations (see Definition~\ref{def:duo}),
with the only difference that the set union for the bit strings is
realised with a bitwise OR.
\subsection{Unum Toolbox}\label{subsec:toolbox}
To examine the numerical properties of Unums, there needs to be a
toolbox to see how this concept works out inside the machine. The
reason why a new toolbox was developed in the course of this thesis
is that all other toolboxes available at the time of writing are
not using LUTs to do calculations. Instead, they emulate Unum-arithmetic
with floating-point numbers that are mapped to a given lattice.
\par
To give an answer to the question if Unums could in theory replace
floating-point numbers for some applications, it is necessary to avoid
floating-point arithmetic at run-time as much as possible.
A possible future machine implementing Unums in hardware would also
be constrained to LUTs and would not be able to
use floating-point numbers in the process and at the same time
leverage the energy and complexity savings projected by
\textsc{Gustafson} in \cite{gu16-b}.
\par
The Unum toolbox programmed in the course of this thesis and
used to examine the numerical behaviour of Unums in
Section~\ref{sec:revisitieee} is split up in two parts.
The first part is the environment generator \texttt{gen}
(see Listing~\ref{lst:genc}), generating the LUTs in \texttt{table.c},
based on type definitions in \texttt{table.h} (see Listing~\ref{lst:tableh})
and the environment parametres in \texttt{config.mk} (see Listing~\ref{lst:configmk}),
and the lattice-specific toolbox-header \texttt{unum.h}. The choice of
lattice points can be arbitrary and it is relatively simple to extend the
generator, but because of the results obtained in Section~\ref{sec:lattice}
only the generating function for a decade Unum lattice is implemented
(see \texttt{gendeclattice()} in Listing~\ref{lst:genc}).
\par
The second part is the toolbox itself (see Listing~\ref{lst:unumc}),
working with the previously generated \texttt{table.c} and \texttt{unum.h},
but being lattice-agnostic in general. The fundamental data type for
operations is \texttt{SORN} defined in \texttt{unum.h},
corresponding to the SORN-concept constructed earlier. Just as proposed
by \textsc{Gustafson} in \cite[Section~3.2]{gu16-b}, the SORN is a bit
array on which operations are carried out as proposed and close to
how it would happen in a native machine implementation.
\par
The provided toolbox functions (see \texttt{unum.h} and Listing~\ref{lst:unumc})
are of both arithmetic and set theoretical nature. The arithmetic functions
corresponding to addition and subtraction are \texttt{uadd()} and
\texttt{usub()}. Addition in this context means the dual Unum operation
$\langle\boxplus\rangle$ using the addition LUT \texttt{addtable} in
\texttt{table.c}. Subtraction is achieved
by negating the second argument on a per-Unum basis and performing an
addition, preserving set-dependencies if present.
Analogously, there are \texttt{umul()} and \texttt{udiv()} for
multiplication and division using $\langle\boxtimes\rangle$ and
the multiplication LUT \texttt{multable} in \texttt{table.c}.
The arithmetic functions \texttt{uneg()} and \texttt{uinv()}
negate and invert a SORN respectively on a per-Unum basis corresponding
to the $\R^*$-Flake negation \enquote{$-$} and inversion \enquote{$/$}.
The function \texttt{uabs()} corresponds to a Unum modulus function and the
\texttt{ulog()} function is an implementation of the $\ln$ function
on Unums using the LUT \texttt{logtable}.
\par
SORN operations and modifications are generalised in the functions
\texttt{\textunderscore sornop()} and \texttt{\textunderscore sornmod()} in
Listing~\ref{lst:unumc} respectively. They are the foundation for almost
all arithmetic functions of this toolbox. Dependent sets are detected
by comparing the two pointers to the operands passed to the arithmetic
functions. If they are equal, the sets are dependent.
\par
The set theoretical functions are \texttt{uemp()} and \texttt{uset()}
for emptying and setting SORNs, \texttt{ucut()} and \texttt{uuni()}
for cutting and taking the union of two SORNs and \texttt{uequ()} and
\texttt{usup()} to check if two SORNs are equal and if one SORN is
the superset of another.
\par
The input and output functions play a special role in this toolbox.
\texttt{uint()} is the only function using floating point numbers to add a
closed interval to a SORN and \texttt{uout()} prints a SORN in a
human-readable format to standard output.
\par
When using the Unum toolbox, only the components \texttt{unum.h}
and the static library \texttt{libunum.a} are relevant and need to
be present when compiling programs using the Unum toolbox
(see Section~\ref{sec:unumproblems}). All functions are reëntrant
and, thus, thread-safe.
\section{Revisiting Floating-Point-Problems}\label{sec:revisitieee}
Using the toolbox presented in Subsection~\ref{subsec:toolbox},
we implement the IEEE 754 floating-point problems studied in
Section~\ref{sec:ieeeproblems} and examine their behaviour within
the Unum arithmetic. For all examples in this section the
environment was set to $(n_b,n_s)=(12,2)$.
\subsection{The Silent Spike}\label{subsec:unum-silentspike}
We can express the spike function (\ref{eq:spike}) within the
Unum arithmetic, using a LUT-based natural logarithm
\begin{equation*}
	\LN\colon\power(\U(P))\to\power(\U(P))
\end{equation*}
defined as
\begin{equation*}
	U  \mapsto
	\begin{cases}
		{\langle \ln_{\F} \rangle}(U) &
			U \cap \power((\iffy,0]) = \emptyset \\
		\emptyset & \text{else}
	\end{cases}
\end{equation*}
(see \texttt{ulog()} in Listing~\ref{lst:unumc}) and an elementary
Unum modulus function $|\cdot|$
(see \texttt{uabs()} in Listing~\ref{lst:unumc}), as
\begin{equation}\label{eq:spike-unum}
	F(X) := \LN(|\blur(\{3\})\boxtimes(\blur(\{1\})\boxplus -X)\boxplus
	        \blur(\{1\})|).
\end{equation}
As we have previously evaluated $f$ in an environment of all floating-point
numbers of the singularity at $\frac{4}{3}$ (see Figure~\ref{fig:spike}), we
evaluate $F$ in an environment
of all Unums of the singularity $\blur\left(\frac{4}{3}\right)$ using
the Unum toolbox (see Listing~\ref{lst:spikec-unum}).
The behaviour is exhibited in Figure~\ref{fig:spike-unum} and it
can be observed that the spike is not hidden any more as was the
case with the floating-point implementation.
\begin{figure}[htpb]
	\centering
	\begin{tikzpicture}
		\begin{axis}[xlabel=$x$, ylabel=$F(x)$,
		         xmin=0.8,xmax=2.0,
		         ymin=-4.5,ymax=0.8,
		         extra x ticks={0.8333333333, 1.3333333333, 1.9},
		         extra x tick style={xticklabel pos=right,
		         	xtick pos=right},
		         extra y ticks={},
		         extra y tick style={yticklabel pos=right,
		         	ytick pos=right},
		         extra x tick labels={$\frac{1}{1.2}$, $\frac{4}{3}$, $1.9$},
		         extra tick style={grid=major}]
			\draw [black, rounded corners]
				(axis cs:0.8333, 0.1818) rectangle (axis cs:0.9091, 0.4167);
			\addplot[mark=o,draw=black] coordinates {
				(0.9091, 0.1818) (0.9091, 0.3448)
			};
			\draw [black, rounded corners]
				(axis cs:0.9091, 0) rectangle (axis cs:1, 0.3448);
			\addplot[mark=*,draw=black] coordinates {
				(1, 0)
			};
			\draw [black, rounded corners]
				(axis cs:1, -0.4167) rectangle (axis cs:1.1, 0);
			\addplot[mark=o,draw=black] coordinates {
				(1.1, -0.4167) (1.1, -0.3333)
			};
			\draw [black, rounded corners]
				(axis cs:1.1, -1) rectangle (axis cs:1.2, -0.3333);
			\addplot[mark=o,draw=black] coordinates {
				(1.2, -1) (1.2, -0.8333)
			};
			\draw [black, rounded corners]
				(axis cs:1.2, -2.4) rectangle (axis cs:1.3, -0.8333);
			\addplot[mark=o,draw=black] coordinates {
				(1.3, -2.4) (1.3, -1.7)
			};
			\draw [black, rounded corners]
				(axis cs:1.3, -5) rectangle (axis cs:1.4, -1.1);
			\node at (axis cs:1.35,-3.4) {$\iffy$};
			\draw[->] (axis cs:1.35,-3.6) -- (axis cs: 1.35 ,-4.4)
				node {};
			\addplot[mark=o,draw=black] coordinates {
				(1.4, -2.4) (1.4, -1.1)
			};
			\draw [black, rounded corners]
				(axis cs:1.4, -2.4) rectangle (axis cs:1.5, -0.6667);
			\addplot[mark=o,draw=black] coordinates {
				(1.5, -0.7143) (1.5, -0.6667)
			};
			\draw [black, rounded corners]
				(axis cs:1.5, -0.7143) rectangle (axis cs:1.6, -0.0909);
			\addplot[mark=o,draw=black] coordinates {
				(1.6, -0.4167) (1.6, -0.0909)
			};
			\draw [black, rounded corners]
				(axis cs:1.6, -0.4167) rectangle (axis cs:1.7, 0.2632);
			\addplot[mark=o,draw=black] coordinates {
				(1.7, 0.2632)
			};
			\addplot[mark=none,draw=black] coordinates {
				(1.7, 0) (1.7, 0.2632)
			};
			\addplot[mark=*,draw=black] coordinates {
				(1.7, 0)
			};
			\draw [black, rounded corners]
				(axis cs:1.7, 0) rectangle (axis cs:1.8, 0.4167);
			\addplot[mark=o,draw=black] coordinates {
				(1.8, 0.2564) (1.8, 0.4167)
			};
			\draw [black, rounded corners]
				(axis cs:1.8, 0.2564) rectangle (axis cs:1.9, 0.5882);
		\end{axis}
	\end{tikzpicture}
	\caption{Evaluation of the Unum spike function $F$
		(see (\ref{eq:spike-unum})) on all Unums in $[\frac{1}{1.2},1.9]$ with
		$(n_b,n_s)=(12,2)$ ($\circ\slash\bullet$ demarks open/closed
		interval endpoints); see Listing~\ref{lst:spikec-unum}.
	}
	\label{fig:spike-unum}
\end{figure}
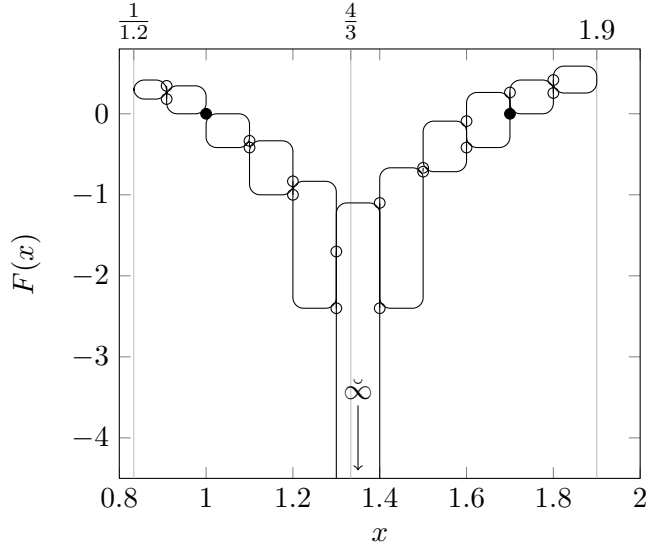
\par
This shows that Unums can effectively be used to quickly evaluate
guaranteed bounds for a given function and observe singular behaviour
without taking the risk of missing it. The bounds are guaranteed
as the foundation for the Unum arithmetic are the well-defined
operations on $\R^*$-Flakes (see Definition~\ref{def:sorf} and
Theorem~\ref{theo:wdoff}).
\subsection{Devil's Sequence}\label{subsec:unum-devil}
The devil's sequence is translated into Unum arithmetic by transforming
(\ref{eq:devil}) into the equivalent SORN-sequence
\begin{equation*}
	U_n :=
	\begin{cases}
		\blur(\{2\}) & n = 0\\
		\blur(\{-4\}) & n = 1 \\
		\blur(\{111\}) \boxplus -\blur(\{1130\}) \boxtimes /U_{n-1} \boxplus
		\blur(\{3000\}) \boxtimes /(U_{n-1}\boxtimes U_{n-2}) &
			n \ge 2.
	\end{cases}
\end{equation*}
Running the Unum toolbox implementation (see
Listing~\ref{lst:devilc-unum}) of this problem, we obtain
\begin{equation*}
	U_{25} = \R^*.
\end{equation*}
This indicates the instability of the problem posed. Even though
the information loss is great, this result can at least be a
warning to investigate the numerical behaviour of the given sequence.
\subsection{The Chaotic Bank Society}\label{subsec:unum-bank}
Taking a look at the chaotic bank society problem, we determine the
equivalent SORN-sequence to (\ref{eq:bank}) as
\begin{equation*}
	A_n :=
	\begin{cases}
		A_0 & n = 0\\
		A_{n-1} \boxtimes \blur(\{n\}) \boxplus -\blur(\{1\}) & n \ge 1.
	\end{cases}
\end{equation*}
Again, running the Unum toolbox implementation (see
Listing~\ref{lst:bankc-unum}), we obtain for
$A_0=\blur(\{ e -1 \})=(1.7,1.8)$
\begin{equation*}
	A_{25} = \R^*.
\end{equation*}
This is consistent with the theoretical results we obtained, given
we can find an $\varepsilon >0$ such that
$A_0$ contains $e-1 + \delta$ with
$\delta\in(-\varepsilon,\varepsilon)$, as
$e-1\not\in P_D(2^{n_b-3}-1,n_d)$.\\
We observe that, even though the results do not lie about the solution,
the information loss is great.
\par
Concluding, introducing Unums as a number format
allowing you to neglect stability analysis has turned out to be a
false promise. We can also not sustain the notion that
naïvely implementing algorithms in Unums abolishes the need for a
break condition. Besides complete information loss, sticking- and
creeping-effects elaborated in Subsection~\ref{subsec:stickcreep}
additionally make it difficult to think of proper ways to do that.
\section{Discussion}
With the theoretical formulation of Unums and practical results, it
is now time to discuss the format taking into account the results
obtained in the previous chapters.
\subsection{Comparison to IEEE 754 Floating-Point Numbers}
It is of central interest to see how the Unums hold up to the previously
introduced IEEE~754 floating-point numbers.
To illustrate the behaviour of the machine Unums, different parametres
of the systems are laid out in Table~\ref{table:unum}.
\begin{table}[htpb]
	\centering
	\begin{tabular}{ | l || l | l | l | l | }
		\hline
		$n_b$ (bit) & 8 & 16 & 32 & 64 \\
		\hline
		$n_s$ & 1 & 3 & 7 & 15 \\
		\hline\hline
		$|P_D|$ & $=3.10\cdot 10^{+1}$ &
			$\approx 8.19 \cdot 10^{+3}$ &
			$\approx 5.37 \cdot 10^{+8}$ &
			$\approx 2.31 \cdot 10^{+18}$ \\
		\hline
		$|\U_M|$ & $=2.56\cdot 10^{+2}$ &
			$\approx 6.55 \cdot 10^{+4}$ &
			$\approx 4.29 \cdot 10^{+9}$ &
			$\approx 1.84 \cdot 10^{+19}$ \\
		\hline
		$\max(P_D)$ & $=5.00\cdot 10^{+3}$ &
			$= 1.91 \cdot 10^{+9}$ &
			$\approx 6.87 \cdot 10^{+59}$ &
			$\approx 1.43 \cdot 10^{+2562}$ \\
		\hline
		${\max(P_D)}^{-1}$ &
			$= 2.00 \cdot 10^{-4}$ &
			$\approx 5.24 \cdot 10^{-10}$ &
			$\approx 1.45 \cdot 10^{-60}$ &
			$\approx 6.99 \cdot 10^{-2563}$ \\
		\hline\hline
		Size of LUTs &
			$\approx 132$ kB &
			$\approx 17$ GB &
			$\approx 1.48\cdot 10^{20}$ B &
			$\approx 5.44 \cdot 10^{39}$ B \\
		\hline
	\end{tabular}
	\caption{Machine Unums properties for $n_b\in\{8,16,32,64\}$ and
		$n_s$ selected to match IEEE~754 significant decimal digits
		($=\lfloor\log_{10}(2^{n_m+1})\rfloor$) for each storage
		size.}
	\label{table:unum}
\end{table}
\par
Comparing Table~\ref{table:unum} to Table~\ref{table:floatsizes}, we
note that for the same number of storage bits, the dynamic range,
the ratio of the largest and smallest representable numbers, of Unums
is orders of magnitude larger than that of IEEE~754 floating-point numbers.
For example, with a storage size of 16 bit, the dynamic range of IEEE~754
floating-point numbers is
\begin{equation*}
	\frac{\max(\M_1)}{\min(\M_0\cap \R^+_{\neq 0})}
	\approx \frac{6.55 \cdot 10^{+4}}{5.96 \cdot 10^{-8}}
	\approx 1.10 \cdot 10^{12}.
\end{equation*}
For Unums, we obtain
\begin{equation*}
	\frac{\max(P_D)}{{\max(P_D)}^{-1}}
	\approx \frac{1.91 \cdot 10^{+9}}{5.24 \cdot 10^{-10}}
	\approx 3.65 \cdot 10^{18}
\end{equation*}
respectively, which is an increase of roughly 6 orders of magnitude.
The reason for this significant difference is the fact that no bit patterns are
wasted for $\nan$-representations in the Unum number format.
\par
One the other hand, one can see that any values for $n_b$ beyond roughly
$12$ bit (corresponding to a LUT size of $\approx 50$ MB) is not
feasible given the huge size of the LUTs. It shows that we can only really
reason about machine Unum environments with $n_b \in \{ 3,\ldots,12 \}$.
\subsection{Sticking and Creeping}\label{subsec:stickcreep}
Working with the Unum toolbox, two effects seem to influence iterative
calculations substantially. A fitting description would be to call
them \emph{sticking-} and \emph{creeping-effects} respectively.
They can be observed, for instance, when evaluating infinite series
within the Unum arithmetic, and this example will be examined here.
\begin{example}[\textsc{Euler}'s number]\label{ex:euler}
	Determining \textsc{Euler}'s number in the Unum arithmetic can be done by
	defining a SORN-series $E_n$ satisfying
	\begin{equation*}
		\blur(\{e\}) \in \lim_{n\to\infty} (E_n),
	\end{equation*}
	where
	\begin{equation}\label{eq:euler}
		E_n := \mathop{\Large\langle\boxplus\rangle}\limits_{k=0}^{n}
			\left[{/}
			{\mathop{\Large\langle\boxtimes\rangle}\limits_{\ell=0}^{k}
			{\blur(\{\ell\})}}\right],
	\end{equation}
	which corresponds to the partial sums of the infinite series representation
	of $e$ as
	\begin{equation*}
		e = \sum_{k=0}^{\infty}{\frac{1}{k!}},
	\end{equation*}
	Using the Unum toolbox (see Listing~\ref{lst:euler}), the partial sums
	of this problem are visualised in Figure~\ref{fig:euler}.
	The first $21$ iterates are depicted and illustrate a
	pathological behaviour.
	\begin{figure}[htpb]
		\centering
		\begin{tikzpicture}
			\begin{axis}[xlabel=$n$, ylabel=$E_n$,
			         extra x ticks={0,20},
			         extra x tick style={xticklabel pos=right,
			         	xtick pos=right},
			         extra y ticks={2.7182818284},
			         extra y tick style={yticklabel pos=right,
			         	ytick pos=right},
			         extra y tick labels={e},
			         extra tick style={grid=major}]
				\addplot[mark=*,draw=black] coordinates {(0 ,1.0)};
				\addplot[mark=*,draw=black] coordinates {(1 ,2.0)};
				\addplot[mark=*,draw=black] coordinates {(2 ,2.5)};
				\addplot[mark=o,draw=black] coordinates {(3 ,2.6) (3 ,2.7)};
				\addplot[mark=o,draw=black] coordinates {(4 ,2.6) (4 ,2.8)};
				\addplot[mark=o,draw=black] coordinates {(5 ,2.6) (5 ,2.9)};
				\addplot[mark=o,draw=black] coordinates {(6 ,2.6) (6 ,3.0)};
				\addplot[mark=o,draw=black] coordinates {(7 ,2.6) (7 ,3.1)};
				\addplot[mark=o,draw=black] coordinates {(8 ,2.6) (8 ,3.2)};
				\addplot[mark=o,draw=black] coordinates {(9 ,2.6) (9 ,3.3)};
				\addplot[mark=o,draw=black] coordinates {(10,2.6) (10,3.4)};
				\addplot[mark=o,draw=black] coordinates {(11,2.6) (11,3.5)};
				\addplot[mark=o,draw=black] coordinates {(12,2.6) (12,3.6)};
				\addplot[mark=o,draw=black] coordinates {(13,2.6) (13,3.7)};
				\addplot[mark=o,draw=black] coordinates {(14,2.6) (14,3.8)};
				\addplot[mark=o,draw=black] coordinates {(15,2.6) (15,3.9)};
				\addplot[mark=o,draw=black] coordinates {(16,2.6) (16,4.0)};
				\addplot[mark=o,draw=black] coordinates {(17,2.6) (17,4.1)};
				\addplot[mark=o,draw=black] coordinates {(18,2.6) (18,4.2)};
				\addplot[mark=o,draw=black] coordinates {(19,2.6) (19,4.3)};
				\addplot[mark=o,draw=black] coordinates {(20,2.6) (20,4.4)};
			\end{axis}
		\end{tikzpicture}
		\caption{Evaluation of the Unum \textsc{Euler} partial sums
		(\ref{eq:euler}) for iterations $n\in\{0,\ldots,20\}$ with
		$(n_b,n_s)=(12,2)$ ($\circ\slash\bullet$ demarks open/closed
		interval endpoints); see Listing~\ref{lst:euler}.
		}
		\label{fig:euler}
	\end{figure}
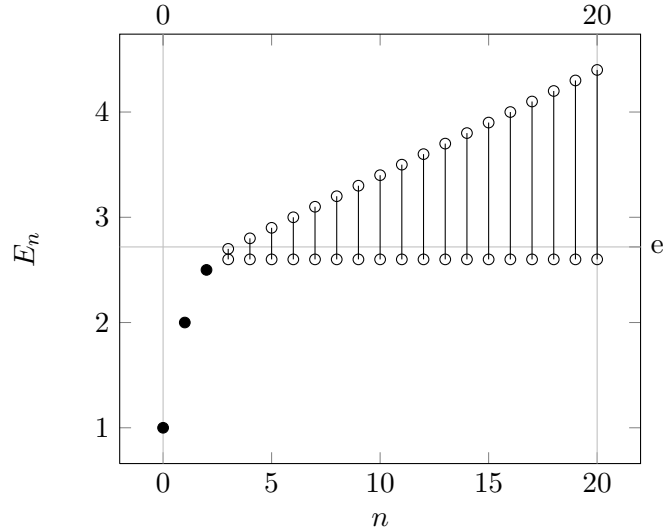
	\par
	Starting from $n=3$, the lower bound of the solution set is stuck
	at the value 2.6.
	One can also observe that the upper bound is growing linearly on each
	iteration. It creeps away from $e$ and reduces the quality of
	the solution with each step.
	\par
	The cause of these sticking- and creeping-effects is the fact that we add
	infinitesimally small values to the SORN on each iteration.
	The lower bound gets stuck because the value added is smaller than the
	length of the lowest interval, hitting a blind spot of the blur function.
	The upper bound creeps away because even though we add an infinitesimally small
	value, it expands to at least the next following Unum value.
\end{example}
This problem makes it impossible to work with Unums to examine infinite
series or sequences and iterative problems in general. Even though
Unums do not lie about the solution, the quality of it is decreased on each
iteration, as we could already see in Subsections \ref{subsec:unum-devil} and
\ref{subsec:unum-bank}. There is also no chance of formulating a break condition for
the given algorithm because of this behaviour. We observe comparable
problems for finding break conditions for infinite series that do not
converge quickly using floating-point numbers, so we can generally think
of it as an unsolved problem following from the finite nature of the machine.
\subsection{Lattice Switching}
A strong theoretical advantage is that one could evaluate an expression
on a set of Unums with a coarse lattice first and then refine the
lattice as soon as the set of possible solutions shrinks.
It is questionable how this could be possible within the machine.
One may find ways to reduce the size of the lookup tables, but assuming
multiple different lattice-precisions including all LUTs would take up
massive amounts of space on a microchip.
Additionally, there needs to be a theory on how existing SORNs are
translated between the different lattices, which might require its own
set of LUTs for each transition, greatly increasing complexity.
\par
If the solution space is only observed within small bounds, considerable
amounts of memory are wasted for set representations beyond
the bounds using a naïve SORN representation. The SORNs may be needed
for intermediate values of a calculation, which could easily expand
beyond the bounds of the solution space, but not for the final results.
\par
This problem can be approached using a run-length encoding for SORNs
comparable to how LUTs were implemented (see Subsection~\ref{subsec:sornops}),
but this would make SORN operations in general less efficient unless
the operations take place directly on top of Unum enumeration indices.
\subsection{Complexity}
Despite the efforts to simplify arithmetic operations and overhead by
creating lookup tables and working on bit strings in a simple manner,
the cost of this simplification weighs heavily.
The contradiction lies within the fact that to at least reduce the
detrimental effects of sticking and creeping it is necessary to
increase the number of Unum bits $n_b$. However, this is only
possible up to a certain point until the LUTs become too large.
In this context, dealing with strictly monotonic functions like $\ln$
in the Unum context requires LUTs for each of them as well
(see Subsection~\ref{subsec:unum-silentspike}).
\par
It is questionable how useful the Unum arithmetic is within the tight
bounds set by these limiting factors. However, it should be taken into
account that there are possible uses for Unums on very coarse grids,
for instance inverse kinematics.
\textsc{Gustafson} also identifies the problem (see
\cite[Section~6]{gu16-b}) and notes that this problem could indicate
that Unums are
\textcquote[Section~6.2]{gu16-b}{
	\textelp{} primarily practical for low-accuracy but high-validity applications,
	thereby complementing float arithmetic instead of replacing it.
}
\chapter{Summary and Outlook}
In the course of this thesis we started off with the construction of a
mathematical description of IEEE 754 floating-point numbers, compared the
properties of different binary storage formats and studied examples which
uncover inherent weaknesses of this arithmetic.
\par
Following from these observations, we constructed the projectively extended
real numbers based on a small set of axioms. After introducing a
definition of finite and infinite limits on the projectively extended real
numbers, we showed their well-definedness in terms of these limits.
Based on this foundation, we developed the Flake arithmetic and proved
well-definedness in terms of set theory.
\par
This effort led us to the mathematical foundation of Unums, which as
proposed approaches the interval arithmetic dependency problem in a
new way and is meant to be easy to implement in the machine. We presented
different types of Unum lattices, evaluated the requirements for hardware
implementations and studied the numerical behaviour in a Unum toolbox,
which was developed in the course of this thesis. Using these results,
we were able to draw the conclusion that Unums may not be a number
format allowing naïve computations, but exhibited promising results in
low-precision but high-validity applications.
\par
The author expected to find drawbacks of this nature for the Unum
number format, as any numerical system exhibits
its strength only within certain conditions, making it easy to find
examples where it fails. In this context, it was observed that IEEE 754
floating-point numbers and Unums complement each other.
Given the nature of Unum arithmetic, it may be on the one hand difficult to
do stability analysis due to the complexity of the arithmetic rules,
but on the other hand the guaranteed bounds of the result do not cover up
when an algorithm is not fit for this environment and indicate the need
to approach the problem using a different numerical approach.
\par
At the point of writing, the revised Unum format was approached
with neither a mathematical foundation nor formalisation. The available
toolboxes were only emulating Unums using floating-point arithmetic, hiding
numerous drawbacks with regard to the complexity of lookup tables.
The results obtained in this thesis make it possible to reason about
Unums in the bounds that will also be present when it comes to
implementing Unums in hardware and not only in software.
\par
In general it is questionable if the approach of using lookup tables
is really the best way to go, despite the possible advantage of
simplifying calculations.
It is questionable if it is really worth it to throw the entire
IEEE 754 floating-point infrastructure overboard and have two
exclusive numerical systems.
\par
The $\blur$ operator presented in this
thesis corresponds to the rounding operation for floating-point numbers
to a certain extent. A topic for further research could be to
introduce closed $\R^*$-intervals for $a,b\in\R^*$ with
\begin{equation*}
	[a,b] := \{ a \} \sqcup (a,b) \sqcup \{ b \} \subset \F.
\end{equation*}
Operations on $\R^*$-Flakes were shown to be well-defined and, thus,
it is possible to extend Flake operations $\star\in\{\boxplus,\boxtimes\}$
to automatically well-defined operations $\diamond$ on closed intervals with
\begin{align*}
	[a,b] \diamond [c,d] :=~& \{ a \} \star \{ c \} \cup
	\{ a \} \star (c,d) \cup \{ a \} \star \{ d \} ~\cup \\
	& (a,b) \star \{ c \} \cup (a,b) \star (c,d) \cup
		(a,b) \star \{ d \} ~\cup \\
	& \{ b \} \star \{ c \} \cup \{ b \} \star (c,d) \cup
		\{ b \} \star \{ d \}
\end{align*}
and simplify it accordingly.
Discretisation is achieved by using floating-point numbers
for the interval bounds and directed rounding for guaranteed
bounds, as elaborated in (\ref{eq:round-interval}).
\par
Unums present the need to have lookup tables for every
operation and nearly every elementary function to be feasible,
which is a huge complexity problem. This is the reason why
using floats instead of lookup tables to achieve this makes
sense, because we studied the behaviour of strictly monotonic
functions on Flakes in this thesis and can directly use
strictly monotonic floating-point functions for Flake
arithmetic instead of lookup tables.
In the end, this could combine the accuracy of floating
point numbers and the certainty of interval arithmetic.
The difference between this and ordinary interval arithmetic
using floating-point number bounds (see \cite{ieee15})
is the use of the projectively extended real numbers instead
of the affinely extended real numbers, making it possible
to model degenerate intervals and divide by zero, and the
knowledge of the results obtained in this thesis to approach
the dependency problem. It comes at the cost of a total
order relation and only offers a partial order, which
can be assumed to be a smaller problem than it seems.
\par
In the end, it all boils down to the question if using
two $32$ bit IEEE 754 floating-point numbers to model such
a closed interval is better than using a single $64$ bit
IEEE 754 floating-point number for a diverse set of algorithms.
The strategy of finding a solution by going from a coarse
to a fine grid for Unums could be easily realised with
floating-point bounded closed intervals and also prove to
be useful for certain applications.
\par
Reaching a point where the dynamic range of high-bit
floating-point numbers is exceeding the range of numbers
of usable magnitude only to compensate rounding errors
to a certain extent, we might find interval arithmetic
on the projectively extended real numbers
to be a good future direction for improving the results
of the very calculations we are doing every day.
\begin{appendix}
	\chapter{Notation Directory}
	\section{Section~\ref{ch:ieee}: \nameref{ch:ieee}}
	\begin{longtable}{p{3.3cm} p{10cm}}
		$n_m$ & number of mantissa-digits ($\equiv$ mantissa-bits in
			base-2) \\
		$n_e$ & number of exponent-bits \\
		$\M_0(n_m,\ul{e})$ & set of subnormal floating-point numbers; see
			Definition~\ref{def:sosfpn} \\
		$\M_1(n_m,\ul{e},\ol{e})$ & set of normal floating-point
			numbers; see Definition~\ref{def:sonfpn} \\
		$\M(n_m,\ul{e}-1,\ol{e}+1)$ & set of floating-point numbers; see
			Definition~\ref{def:sofpn} \\
		$|\nan|(n_m)$ & number of $\nan$ representations; see
			Proposition~\ref{prop:nonr} \\
		$\ul{e}(n_e)$, $\ol{e}(n_e)$ & exponent bias; see
			Definition~\ref{def:eb} \\
		$\rde$ & \nameref{def:natter}; see
			Definition~\ref{def:natter} \\
		$\rd_{\uparrow}$ & upward rounding; see
			Definition~\ref{def:ur} \\
		$\rd_{\downarrow}$ & downward rounding; see
			Definition~\ref{def:dr} 
	\end{longtable}
	\section{Section~\ref{ch:ia}: \nameref{ch:ia}}
	\begin{longtable}{p{3.3cm} p{10cm}}
		$\R^*$ & projectively extended real numbers; see
			Definition~\ref{def:pern} \\
		$\iffy$ & infinity symbol of $\R^*$;
			see Definition~\ref{def:pern} \\
		$\sqcup$ & disjoint union; see Definition~\ref{def:du} \\
		$(\ul{a},\ol{a})$ & open $\R^*$-interval between $\ul{a}$ and
			$\ol{a}$; see Definition~\ref{def:ori} \\
		$\I$ & set of open $\R^*$-intervals; see
			Definition~\ref{def:soori} \\
		$\oplus$  & addition operator on $\I$; see
			Definition~\ref{def:soori} \\
		$\otimes$ & multiplication operator on $\I$; see
			Definition~\ref{def:soori} \\
		$\S(S)$ & set of S-singletons; see
			Definition~\ref{def:sos} \\
		$\F$ & set of $\R^*$-Flakes; see
			Definition~\ref{def:sorf} \\
		$\boxplus$ & addition operator on $\F$; see
			Definition~\ref{def:sorf} \\
		$\boxtimes$ & multiplication operator on $\F$; see
			Definition~\ref{def:sorf} \\
		$f_{\F}$ & $\R^*$-Flake evaluation of the strictly
			monotonic function $f$; see Definitions
			\ref{def:rfeosif} and \ref{def:rfeosdf}
	\end{longtable}
	\section{Section~\ref{ch:ua}: \nameref{ch:ua}}
	\begin{longtable}{p{3.3cm} p{10cm}}
		$\U(P)$ & set of Unums on the lattice $P$; see
			Definition~\ref{def:sou} \\
		$\power(S)$ & powerset of S; see
			Definition~\ref{def:ps} \\
		$\blur$ & blur operator; see Definition~\ref{def:blur} \\
		$\langle\star\rangle$ & dual Unum operation; see
			Definition~\ref{def:duo} \\
		$\langle f_{\F} \rangle$ & Unum evaluation of the strictly
			monotonic function $f$; see Definitions
			\ref{def:ueosif} and \ref{def:ueosdf} \\
		$P_L(p,n)$ & linear Unum lattice; see
			Definition~\ref{def:lul} \\
		$P_E(p,m)$ & exponential Unum lattice; see
			Definition~\ref{def:eul} \\
		$P_D(p,s)$ & decade Unum lattice; see
			Definition~\ref{def:dul} \\
		$\asc$ & ascension operator; see
			Definition~\ref{def:asc} \\
		$u(n)$ & $n$th Unum in the Unum enumeration; see
			Definition~\ref{def:ue} \\
		$n_b$ & number of Unum bits \\
		$n_d$ & number of significant digits \\
		$\U_M(n_b,n_s)$ & set of machine Unums; see
			Definition~\ref{def:somu} \\
	\end{longtable}
	\newpage
	\chapter{Code Listings}\label{ch:code}
	\lstset{basicstyle=\ttfamily\normalsize,inputencoding=utf8/latin1,
		morekeywords={},float,numberstyle=\footnotesize\noncopynumber,
		breakatwhitespace=false,breaklines=false,escapeinside={\%*}{*)},
		columns=fullflexible,keepspaces=true,numbers=left,
		frame=single,showstringspaces=true,keywordstyle=\ttfamily}
	\section{IEEE 754 Floating-Point Problems}
	\subsection{spike.c}
\begin{lstlisting}[language=C,label=lst:spike]
#include <float.h>
#include <math.h>
#include <stdio.h>

#define LEN(x) (sizeof (x) / sizeof *(x))
#define NUMPOINTS (10)
#define POLE (4.0/3.0)

int
main(void)
{
	double x[NUMPOINTS * 2 + 1][2];
	int i;

	/* fill POINT environment */
	x[NUMPOINTS][0] = POLE;
	for (i = NUMPOINTS - 1; i >= 0; i--) {
		x[i][0] = nextafter(x[i + 1][0], -INFINITY);
	}
	for (i = NUMPOINTS + 1; i < NUMPOINTS * 2 + 1; i++) {
		x[i][0] = nextafter(x[i - 1][0], +INFINITY);
	}

	/* calculate values of |3*(1-x)+1| in the POINT environment */
	for (i = 0; i < NUMPOINTS * 2 + 1; i++) {
		x[i][1] = fabs(3 * (1 - x[i][0]) + 1);
		printf("x-%.2f=%.20e |3*(1-x)+1|=%.20f\n", POLE,
		       x[i][0]-POLE, x[i][1]);
	}

	putchar('\n');

	/* calculate values of f(x) in the POINT environment */
	for (i = 0; i < NUMPOINTS * 2 + 1; i++) {
		x[i][1] = log(fabs(3 * (1 - x[i][0]) + 1));
		printf("x-%.2f=%.20e f(x)=%.20f\n", POLE,
		       x[i][0]-POLE, x[i][1]);
	}

	return 0;
}
\end{lstlisting}
	\subsection{devil.c}
\begin{lstlisting}[language=C,label=lst:devil]
#include <stdio.h>

int
main(void)
{
	double a, b, tmp;
	int i;

	a = 2;
	b = -4;

	for (i = 2; i < 26; i++) {
		tmp = 111 - 1130 / b + 3000 / (b * a);
		a = b;
		b = tmp;
		printf("u_%.2d = %f\n", i, b);
	}

	return 0;
}
\end{lstlisting}
	\subsection{bank.c}
\begin{lstlisting}[language=C,label=lst:bank]
#include <float.h>
#include <math.h>
#include <stdio.h>

const int years = 25;

int
main(void)
{
	double a;
	int n;

	a = 1.718281828459045235;

	for (n = 1; n <= years; n++) {
		a = a * n - 1;
	}

	printf("u_%d = %f\n", years, a);

	return 0;
}
\end{lstlisting}
	\subsection{Makefile}
\begin{lstlisting}[language=make]
PROBLEMS = devil bank
LMPROBLEMS = spike

all: $(PROBLEMS) $(LMPROBLEMS)

$(LMPROBLEMS): LDFLAGS = -lm

%: %.c
	cc $^ -o $@ $(LDFLAGS)
clean:
	rm -f $(PROBLEMS) $(LMPROBLEMS)
\end{lstlisting}
	\section{Unum Toolbox}
	\subsection{gen.c}
\begin{lstlisting}[language=C,label=lst:genc]
#include <fenv.h>
#include <float.h>
#include <math.h>
#include <stdint.h>
#include <stdio.h>
#include <stdlib.h>
#include <string.h>

#undef LEN
#define LEN(x) sizeof(x) / sizeof(*x)
#undef UCLAMP
#define UCLAMP(i, off) (((((off < 0) && (i) < -off) ? \
                       numunums - ((-off - (i)) % numunums) : \
                       ((off > 0) && (i) + off > numunums - 1) ? \
                       ((i) + off % numunums) % numunums : (i) + off)) \
		       % numunums)
#undef MIN
#define MIN(x,y)  ((x) < (y) ? (x) : (y))
#undef MAX
#define MAX(x,y)  ((x) > (y) ? (x) : (y))

struct _unum {
	double val;
	char *name;
};

struct _unumrange {
	size_t low;
	size_t upp;
};

struct _latticep {
	char *name;
	double val;
};

static void
printunums(struct _unum *unum, size_t numunums)
{
	size_t i;

	fputs("\nstruct _unum unums[] = {\n", stdout);

	for (i = 0; i < numunums; i++) {
		if (isnan(unum[i].val) && !unum[i].name) {
			printf("\t{ NAN, NULL },\n");
		} else if (isinf(unum[i].val)) {
			printf("\t{ INFINITY, \"%s\" },\n",
			       unum[i].name);
		} else {
			printf("\t{ %f, \"%s\" },\n", unum[i].val,
			       unum[i].name);
		}
	}

	fputs("};\n", stdout);
}

size_t
blur(double val, struct _unum *unum, size_t numunums)
{
	size_t i;

	/* infinity is infinity */
	if (isinf(val)) {
		return numunums / 2;
	}

	for (i = 0; i < numunums; i++) {
		/* equality within relative epsilon */
		if (isfinite(unum[i].val) && fabs(unum[i].val - val) <=
		    DBL_EPSILON * MAX(fabs(unum[i].val), fabs(val))) {
			return i;
		}

		/* in range */
		if (isnan(unum[i].val) &&
		    val < unum[UCLAMP(i, +1)].val &&
		    val > unum[UCLAMP(i, -1)].val) {
			return i;
		}
	}

	/* negative or positive range outshot */
	return (val < 0) ? (numunums / 2 + 1) : (val > 0) ?
	       (numunums / 2 - 1) : 0;
}

void
add(size_t a, size_t b, struct _unumrange *res,
    struct _unum *unum, size_t numunums)
{
	double av, bv, aupp, alow, bupp, blow;

	av = unum[a].val;
	bv = unum[b].val;

	if (isnan(av) && isnan(bv)) {
		/*
		 * a interval, b interval
		 */
		aupp = unum[UCLAMP(a, +1)].val;
		alow = unum[UCLAMP(a, -1)].val;
		bupp = unum[UCLAMP(b, +1)].val;
		blow = unum[UCLAMP(b, -1)].val;

		if ((isinf(alow) && isinf(aupp)) ||
		    (isinf(blow) && isinf(bupp))) {
			/* all real numbers */
			res->low = numunums / 2 + 1;
			res->upp = numunums / 2 - 1;
			return;
		} else if (isinf(alow) && isinf(blow)) {
			/* (iffy, aupp + bupp) */
			res->low = numunums / 2 + 1;
			fesetround(FE_UPWARD);
			res->upp = blur(aupp + bupp, unum,
			                numunums);
		} else if (isinf(aupp) && isinf(bupp)) {
			/* (alow + blow, iffy) */
			fesetround(FE_DOWNWARD);
			res->low = blur(alow + blow, unum, numunums);
			res->upp = numunums / 2 + 1;
		} else {
			/* (alow + blow, aupp + bupp) */
			fesetround(FE_DOWNWARD);
			res->low = blur(alow + blow, unum, numunums);
			fesetround(FE_UPWARD);
			res->upp = blur(aupp + bupp, unum, numunums);
		}
	} else if (!isnan(av) && !isnan(bv)) {
		/*
		 * a point, b point
		 */
		if (isinf(av) && isinf(bv)) {
			/* all extended real numbers */
			res->low = 0;
			res->upp = numunums - 1;
			return;
		} else {
			fesetround(FE_DOWNWARD);
			res->low = blur(av + bv, unum, numunums);
			fesetround(FE_UPWARD);
			res->upp = blur(av + bv, unum, numunums);
		}
	} else if (!isnan(av) && isnan(bv)) {
		/*
		 * a point, b interval
		 */
		bupp = unum[UCLAMP(b, +1)].val;
		blow = unum[UCLAMP(b, -1)].val;

		if (isinf(av)) {
			/* all extended real numbers */
			res->low = 0;
			res->upp = numunums - 1;
			return;
		} else {
			fesetround(FE_DOWNWARD);
			res->low = blur(av + blow, unum, numunums);
			fesetround(FE_UPWARD);
			res->upp = blur(av + bupp, unum, numunums);
		}
	} else if (!isnan(bv) && isnan(av)) {
		/*
		 * a interval, b point
		 */
		add(b, a, res, unum, numunums);
		return;
	}

	if (isnan(av) || isnan(bv)) {
		/* we had an open interval in our calculation
		 * and need to check if res->upp or res->low
		 * are a point. If this is the case, we have
		 * to round it down to respect the openness
		 * of the real interval */
		if (!isnan(unum[res->low].val)) {
			res->low = UCLAMP(res->low, +1);
		}
		if (!isnan(unum[res->upp].val)) {
			res->upp = UCLAMP(res->upp, -1);
		}
	}
}

void
mul(size_t a, size_t b, struct _unumrange *res,
    struct _unum *unum, size_t numunums)
{
	double av, bv, aupp, alow, bupp, blow;

	av = unum[a].val;
	bv = unum[b].val;

	if (isnan(av) && isnan(bv)) {
		/*
		 * a interval, b interval
		 */
		aupp = unum[UCLAMP(a, +1)].val;
		alow = unum[UCLAMP(a, -1)].val;
		bupp = unum[UCLAMP(b, +1)].val;
		blow = unum[UCLAMP(b, -1)].val;

		if ((isinf(alow) && isinf(aupp)) ||
		    (isinf(blow) && isinf(bupp))) {
			/* all real numbers */
			res->low = numunums / 2 + 1;
			res->upp = numunums / 2 - 1;
			return;
		} else if (isinf(alow) && isinf(blow)) {
			if (aupp <= 0 && bupp <= 0) {
				/* (aupp * bupp, iffy) */
				fesetround(FE_DOWNWARD);
				res->low = blur(aupp * bupp,
				                unum, numunums);
				res->upp = numunums / 2 - 1;
			} else {
				/* all real numbers */
				res->low = numunums / 2 + 1;
				res->upp = numunums / 2 - 1;
				return;
			}
		} else if (isinf(aupp) && isinf(bupp)) {
			if (alow >= 0 && blow >= 0) {
				/* (alow * blow, iffy) */
				fesetround(FE_DOWNWARD);
				res->low = blur(alow * blow,
				                unum, numunums);
				res->upp = numunums / 2 - 1;
			} else {
				/* all real numbers */
				res->low = numunums / 2 + 1;
				res->upp = numunums / 2 - 1;
				return;
			}
		} else if (isinf(alow) && isinf(bupp)) {
			if (aupp <= 0 && blow >= 0) {
				/* (iffy, aupp * blow) */
				res->low = numunums / 2 + 1;
				fesetround(FE_UPWARD);
				res->upp = blur(aupp * blow,
				                unum, numunums);
			} else {
				/* all real numbers */
				res->low = numunums / 2 + 1;
				res->upp = numunums / 2 - 1;
				return;
			}
		} else if (isinf(aupp) && isinf(blow)) {
			mul(b, a, res, unum, numunums);
			return;
		} else if (isinf(alow)) {
			if (blow >= 0) {
				/* (iffy, MAX(aupp * blow,
				 *            aupp * bupp) */
				res->low = numunums / 2 + 1;
				fesetround(FE_UPWARD);
				res->upp = blur(MAX(aupp * blow,
				                    aupp * bupp),
				                unum, numunums);
			} else if (bupp <= 0) {
				/* (MIN(aupp * blow, aupp * bupp),
				 *  iffy) */
				fesetround(FE_DOWNWARD);
				res->low = blur(MIN(aupp * blow,
				                    aupp * bupp),
				                unum, numunums);
				res->upp = numunums / 2 - 1;
			} else {
				/* all real numbers */
				res->low = numunums / 2 + 1;
				res->upp = numunums / 2 - 1;
				return;
			}
		} else if (isinf(aupp)) {
			if (blow >= 0) {
				/* (MIN(alow * blow, aupp * bupp),
				 *  iffy) */
				fesetround(FE_DOWNWARD);
				res->low = blur(MIN(alow * blow,
				                    alow * bupp),
				                unum, numunums);
				res->upp = numunums / 2 - 1;
			} else if (bupp <= 0) {
				/* (iffy, MAX(alow * blow,
				 *            alow * bupp) */
				res->low = numunums / 2 + 1;
				fesetround(FE_UPWARD);
				res->upp = blur(MAX(alow * blow,
				                    alow * bupp),
				                unum, numunums);
			} else {
				/* all real numbers */
				res->low = numunums / 2 + 1;
				res->upp = numunums / 2 - 1;
				return;
			}
		} else if (isinf(blow) || isinf(bupp)) {
			mul(b, a, res, unum, numunums);
		} else {
			/* (MIN(C), MAX(C)) */
			fesetround(FE_DOWNWARD);
			res->low = blur(MIN(MIN(alow * blow,
			                        alow * bupp),
			                    MIN(aupp * blow,
			                        aupp * bupp)),
					unum, numunums);
			fesetround(FE_UPWARD);
			res->upp = blur(MAX(MAX(alow * blow,
			                        alow * bupp),
			                    MAX(aupp * blow,
			                        aupp * bupp)),
					unum, numunums);
		}
	} else if (!isnan(av) && !isnan(bv)) {
		/*
		 * a point, b point
		 */
		if ((isinf(av) && (fabs(bv) <= DBL_EPSILON *
		                   fabs(bv) || isinf(bv))) ||
		    (isinf(bv) && (fabs(av) <= DBL_EPSILON *
		                   fabs(av) || isinf(av)))) {
			/* all extended real numbers */
			res->low = 0;
			res->upp = numunums - 1;
			return;
		} else {
			fesetround(FE_DOWNWARD);
			res->low = blur(av * bv, unum, numunums);
			fesetround(FE_UPWARD);
			res->upp = blur(av * bv, unum, numunums);
		}
	} else if (!isnan(av) && isnan(bv)) {
		/*
		 * a point, b interval
		 */
		bupp = unum[UCLAMP(b, +1)].val;
		blow = unum[UCLAMP(b, -1)].val;

		if (isinf(av)) {
			if (isinf(blow)) {
				if (bupp < 0) {
					/* infinity */
					res->low = numunums / 2;
					res->upp = numunums / 2;
					return;
				} else {
					/* all extended real
					 * numbers */
					res->low = 0;
					res->upp = numunums - 1;
					return;
				}
			} else if (isinf(bupp)) {
				if (blow > 0) {
					/* infinity */
					res->low = numunums / 2;
					res->upp = numunums / 2;
					return;
				} else {
					/* all extended real
					 * numbers */
					res->low = 0;
					res->upp = numunums - 1;
					return;
				}
			} else if ((blow < 0) == (bupp < 0)) {
				/* infinity */
				res->low = numunums / 2;
				res->upp = numunums / 2;
				return;
			} else {
				/* all extended real numbers */
				res->low = 0;
				res->upp = numunums - 1;
				return;
			}
		} else {
			/* (MIN(av * blow, av * bupp),
			 *  MAX(av * blow, av * bupp)) */
			fesetround(FE_DOWNWARD);
			res->low = blur(MIN(av * blow,
			                    av * bupp),
			                unum, numunums);
			fesetround(FE_UPWARD);
			res->upp = blur(MAX(av * blow,
			                    av * bupp),
			                unum, numunums);
		}
	} else if (isnan(av) && !isnan(bv)) {
		/*
		 * a interval, b point
		 */
		mul(b, a, res, unum, numunums);
		return;
	}

	if (isnan(av) || isnan(bv)) {
		/* we had an open interval in our calculation
		 * and need to check if res->upp or res->low
		 * are a point. If this is the case, we have
		 * to round it down to respect the openness
		 * of the real interval */
		if (!isnan(unum[res->low].val)) {
			res->low = UCLAMP(res->low, +1);
		}
		if (!isnan(unum[res->upp].val)) {
			res->upp = UCLAMP(res->upp, -1);
		}
	}
}

static void
gentable(char *name, void (*f)(size_t, size_t, struct _unumrange *,
    struct _unum *, size_t), struct _unum *unum, size_t numunums)
{
	struct _unumrange res;
	size_t s, z;

	printf("\nstruct _unumrange %stable[] = {\n", name);

	for (z = 0; z < numunums; z++) {
		putc('\t', stdout);

		for (s = 0; s <= z; s++) {
			f(s, z, &res, unum, numunums);
			printf("%s{ %zd, %zd },", s ? " " : "",
			       res.low, res.upp);
		}

		fputs("\n", stdout);
	}

	fputs("};\n", stdout);
}

void
ulog(size_t u, struct _unumrange *res, struct _unum *unum,
    size_t numunums)
{
	double uv, ulow, uupp;

	uv = unum[u].val;

	if (isnan(uv)) {
		ulow = unum[UCLAMP(u, -1)].val;
		uupp = unum[UCLAMP(u, +1)].val;

		res->low = blur(log(ulow), unum, numunums);
		res->upp = blur(log(uupp), unum, numunums);
	} else {
		res->low = res->upp = blur(log(uv), unum, numunums);
	}
}

static void
genfunctable(char *name, void (*f)(size_t, struct _unumrange *,
    struct _unum *, size_t), struct _unum *unum, size_t numunums)
{
	struct _unumrange res;
	size_t u;

	printf("\nstruct _unumrange %stable[] = {\n", name);

	for (u = 0; u <= numunums / 2; u++) {
		f(u, &res, unum, numunums);
		printf("\t{ %zd, %zd },\n", res.low, res.upp);
	}

	fputs("};\n", stdout);
}

static void
genunums(struct _latticep *lattice, size_t latticesize,
         struct _unum *unum, size_t numunums)
{
	size_t off;
	ssize_t i;

	off = 0;

	/* 0 */
	unum[off].val = 0.0;
	unum[off].name = "0";
	off++;
	unum[off].val = NAN;
	unum[off].name = NULL;
	off++;

	/* (0,1) */
	for (i = latticesize - 1; i >= 0; i--, off++) {
		unum[off].val = 1 / lattice[i].val;
		if (lattice[i].name[0] == '/') {
			unum[off].name = lattice[i].name + 1;
		} else {
			/* add '/' prefix */
			if (!(unum[off].name =
			      malloc(strlen(lattice[i].name) + 2))) {
				fprintf(stderr, "out of memory\n");
				exit(1);
			}
			strcpy(unum[off].name + 1, lattice[i].name);
			unum[off].name[0] = '/';
		}
		off++;
		unum[off].val = NAN;
		unum[off].name = NULL;
	}

	/* 1 */
	unum[off].val = 1.0;
	unum[off].name = "1";
	off++;
	unum[off].val = NAN;
	unum[off].name = NULL;
	off++;

	/* (1,INF) */
	for (i = 0; i < latticesize; i++, off++) {
		unum[off].val = lattice[i].val;
		unum[off].name = lattice[i].name;
		off++;
		unum[off].val = NAN;
		unum[off].name = NULL;
	}

	/* INF */
	unum[off].val = INFINITY;
	unum[off].name = "\u221E";
	off++;
	unum[off].val = NAN;
	unum[off].name = NULL;
	off++;

	/* (INF,-1) */
	for (i = latticesize - 1; i >= 0; i--, off++) {
		unum[off].val = -lattice[i].val;
		if (!(unum[off].name =
		      malloc(strlen(lattice[i].name) + 2))) {
			fprintf(stderr, "out of memory\n");
			exit(1);
		}
		strcpy(unum[off].name + 1, lattice[i].name);
		unum[off].name[0] = '-';
		off++;
		unum[off].val = NAN;
		unum[off].name = NULL;
	}

	/* -1 */
	unum[off].val = -1.0;
	unum[off].name = "-1";
	off++;
	unum[off].val = NAN;
	unum[off].name = NULL;
	off++;

	/* (-1, 0) */
	for (i = 0; i < latticesize; i++, off++) {
		unum[off].val = - 1 / lattice[i].val;
		if (lattice[i].name[0] == '/') {
			if (!(unum[off].name =
			      strdup(lattice[i].name))) {
				fprintf(stderr, "out of memory\n");
				exit(1);
			}
			unum[off].name[0] = '-';
		} else {
			/* add '-/' prefix */
			if (!(unum[off].name =
			      malloc(strlen(lattice[i].name) + 3))) {
				fprintf(stderr, "out of memory\n");
				exit(1);
			}
			strcpy(unum[off].name + 2, lattice[i].name);
			unum[off].name[0] = '-';
			unum[off].name[1] = '/';
		}
		off++;
		unum[off].val = NAN;
		unum[off].name = NULL;
	}
}

void
gendeclattice(struct _latticep **lattice, size_t *latticesize,
              double maximum, int sigdigs)
{
	size_t i, maxlen;
	double c1, c2, curmax;
	char *fmt = "%.*f";

	/*
	 * Check prerequisites
	 */
	if (sigdigs == 0) {
		fprintf(stderr, "invalid number of "
		        "significant digits\n");
	}
	if ((*latticesize == 0) == isinf(maximum)) {
		fprintf(stderr, "gendeclattice: accepting "
		"only one parameter besides number of "
		"significant digits\n");
		exit(1);
	}

	c1 = pow(10, sigdigs) - pow(10, sigdigs - 1);
	c2 = pow(10, -(sigdigs - 1));

	if (*latticesize == 0) {
		/* calculate lattice size until maximum is
		 * contained */
		for (curmax = 0; curmax < maximum; (*latticesize)++) {
			curmax = (1 + c2 *
			          (*latticesize % (size_t)c1)) *
			         pow(10, floor(*latticesize / c1));
		}
	} else { /* isinf(maximum) */
		/* calculate maximum */
		maximum = (1 + c2 *
		           (*latticesize % (size_t)c1)) *
		          pow(10, floor(*latticesize / c1));
	}

	/*
	 * Generate lattice
	 */
	if (!(*lattice = malloc(sizeof(struct _latticep) *
	                        *latticesize))) {
		fprintf(stderr, "out of memory\n");
		exit(1);
	}
	maxlen = snprintf(NULL, 0, fmt, sigdigs - 1, maximum) + 1;
	for (i = 0; i < *latticesize; i++) {
		(*lattice)[i].val = (1 + c2 *
		                      ((i + 1) % (size_t)c1)) *
			             pow(10, floor((i + 1) / c1));
		if (!((*lattice)[i].name = malloc(maxlen))) {
			fprintf(stderr, "out of memory\n");
			exit(1);
		}
		snprintf((*lattice)[i].name, maxlen, fmt,
		         sigdigs - 1, (*lattice)[i].val);
	}
}

int
main(void)
{
	struct _unum *unum;
	struct _latticep *lattice;
	size_t latticebits, latticesize, numunums;
	ssize_t i;
	int bits;

	/* Generate lattice */
	latticesize = (1 << (UBITS - 3)) - 1;
	gendeclattice(&lattice, &latticesize, INFINITY, DIGITS);

	/*
	 * Print unum.h includes
	 */
	fprintf(stderr, "#include <math.h>\n#include <stddef.h>\n"
	        "#include <stdint.h>\n\n");

	/*
	 * Determine number of effective bits used
	 */
	struct {
		int bits;
		char *type;
	} types[] = {
		{ 8,  "uint8_t"  },
		{ 16, "uint16_t" },
		{ 32, "uint32_t" },
		{ 64, "uint64_t" },
	};

	numunums = 8 * (latticesize + 1);
	for (bits = 1; bits <= types[LEN(types) - 1].bits; bits++) {
		if (numunums == (1 << bits))
			break;
	}
	if (bits > types[LEN(types) - 1].bits) {
		fprintf(stderr, "invalid number of lattice"
		        "points\n");
		return 1;
	}

	/*
	 * Determine type needed to store the unum
	 */
	for (i = 0; i < LEN(types); i++) {
		if (types[i].bits >= bits)
			break;
	}
	if (i == LEN(types)) {
		fprintf(stderr, "cannot fit bits into system"
		        "types\n");
		return 1;
	}

	/*
	 * Print list of preliminary unum.h definitions
	 */
	fprintf(stderr, "typedef %s unum;\n#define ULEN %d\n"
	        "#define NUMUNUMS %zd\n", types[i].type, bits,
	        numunums);

	fprintf(stderr, "#define UCLAMP(i, off) (((((off < 0) && (i) <"
	        "-off) ? \\\n\tNUMUNUMS - ((-off - (i)) %% NUMUNUMS) :"
		"\\\n\t((off > 0) && (i) + off > NUMUNUMS - 1) ? \\\n\t"
		"((i) + off %% NUMUNUMS) %% NUMUNUMS : (i) + off)) %% "
		"NUMUNUMS)\n\n");

	fprintf(stderr, "typedef struct {\n\tuint8_t data[%d];\n} SORN;\n",
	        (1 << bits) / 8);

	fprintf(stderr, "\nvoid uadd(SORN *, SORN *);\n"
	        "void usub(SORN *, SORN *);\n"
		"void umul(SORN *, SORN *);\n"
		"void udiv(SORN *, SORN *);\n"
		"void uneg(SORN *);\n"
		"void uinv(SORN *);\n"
		"void uabs(SORN *);\n\n"
		"void ulog(SORN *);\n\n"
		"void uemp(SORN *);\n"
		"void uset(SORN *, SORN *);\n"
		"void ucut(SORN *, SORN *);\n"
		"void uuni(SORN *, SORN *);\n"
		"int uequ(SORN *, SORN *);\n"
		"int usup(SORN *, SORN *);\n\n"
		"void uint(SORN *, double, double);\n"
		"void uout(SORN *);\n");

	/*
	 * Generate unums
	 */
	if (!(unum = malloc(sizeof(struct _unum) * numunums))) {
		fprintf(stderr, "out of memory\n");
		return 1;
	}
	genunums(lattice, latticesize, unum, numunums);

	/*
	 * Print table.c includes
	 */
	printf("#include \"table.h\"\n");

	/*
	 * Print list of unums
	 */
	printunums(unum, numunums);

	/*
	 * Generate and print tables
	 */
	gentable("add", add, unum, numunums);
	gentable("mul", mul, unum, numunums);

	/*
	 * Generate function tables
	 */
	genfunctable("log", ulog, unum, numunums);

	return 0;
}
\end{lstlisting}
	\subsection{table.h}
\begin{lstlisting}[language=C,label=lst:tableh]
#include "unum.h"

struct _unumrange {
	unum a;
	unum b;
};

struct _unum {
	double val;
	char *name;
};

extern struct _unum unums[];
extern struct _unumrange addtable[];
extern struct _unumrange multable[];
extern struct _unumrange logtable[];
\end{lstlisting}
	\subsection{unum.c}
\begin{lstlisting}[language=C,label=lst:unumc]
#include <float.h>
#include <math.h>
#include <stdio.h>

#include "table.h"

#undef MAX
#define MAX(x,y)  ((x) > (y) ? (x) : (y))

static size_t
blur(double val)
{
        size_t i;

        /* infinity is infinity */
        if (isinf(val)) {
                return NUMUNUMS / 2;
        }

	for (i = 0; i < NUMUNUMS; i++) {
                /* equality within relative epsilon */
                if (isfinite(unums[i].val) && fabs(unums[i].val - val) <=
		    DBL_EPSILON * MAX(fabs(unums[i].val), fabs(val))) {
                        return i;
                }

                /* in range */
                if (isnan(unums[i].val) &&
		    val < unums[UCLAMP(i, +1)].val &&
                    val > unums[UCLAMP(i, -1)].val) {
                        return i;
                }
        }

	/* negative or positive range outshot */
	return (val < 0) ? (NUMUNUMS / 2 + 1) : (val > 0) ?
	       (NUMUNUMS / 2 - 1) : 0;
}

static void
_sornaddrange(SORN *s, unum lower, unum upper)
{
	unum u;
	size_t i, j;
	int first;

	for (first = 1, u = lower; u != UCLAMP(upper, +1) ||
	     (first && lower == UCLAMP(upper, +1));
	    u = UCLAMP(u, +1)) {
	     	first = 0;
		i = u / (sizeof(*s->data) * 8);
		j = u % (sizeof(*s->data) * 8);

		s->data[i] |= (1 << (sizeof(*s->data) * 8 - 1 - j));
	}
}

static unum
_unumnegate(unum u)
{
	return UCLAMP(NUMUNUMS, -u);
}

static unum
_unuminvert(unum u)
{
	return _unumnegate(UCLAMP(u, +(NUMUNUMS / 2)));
}

static unum
_unumabs(unum u)
{
	return (u > NUMUNUMS / 2) ? _unumnegate(u) : u;
}

static void
_sornop(SORN *a, SORN *b, struct _unumrange table[],
        unum (*mod)(unum))
{
	unum u, v, low, upp;
	size_t i, j, m, n;
	static SORN res;

	/* empty result SORN */
	for (i = 0; i < sizeof(res.data); i++) {
		res.data[i] = 0;
	}

	for (i = 0; i < sizeof(a->data); i++) {
		for (j = 0; j < sizeof(*a->data) * 8; j++) {
			if (!(a->data[i] & (1 << (sizeof(*a->data) * 8 -
			                    1 - j)))) {
				continue;
			}
			/* unum u = (so * i + j) is in the first set */
			u = sizeof(*a->data) * 8 * i + j;
			for (m = 0; m < sizeof(b->data); m++) {
				for (n = 0; n < sizeof(*b->data) * 8;
				     n++) {
					if (!(b->data[m] & (1 <<
					    (sizeof(*b->data) * 8 - 1 -
					     n)))) {
						continue;
					}
					/* unum v = (so * m + n) is in
					 * the second set */
					v = sizeof(*b->data) * 8 * m + n;

					/*
					 * compare struct pointers to
					 * identify dependent arguments
					 * and in this case only do
					 * pairwise operations
					 */
					if (a == b && u != v)
						continue;

					/* apply an optional modifier
					 * after the dependency check */
					if (mod) {
						v = mod(v);
					}

					/* get bounds from table;
					 * according to gauß 1 + 2 + 3
					 * ... + n = (n * (n + 1)) / 2,
					 * used to traverse triangle
					 * array */
					if (u <= v) {
						low = table[((size_t)v *
						             (v + 1)) /
						            2 + u].a;
						upp = table[((size_t)v *
						             (v + 1)) /
						            2 + u].b;
					} else {
						low = table[((size_t)u *
						             (u + 1)) /
						            2 + v].a;
						upp = table[((size_t)u *
						             (u + 1)) /
						            2 + v].b;

					}
					_sornaddrange(&res, low, upp);
				}
			}
		}
	}

	/* write result to first operand */
	for (i = 0; i < sizeof(a->data); i++) {
		a->data[i] = res.data[i];
	}
}

static void
_sornmod(SORN *s, unum (mod)(unum))
{
	SORN res;
	unum u;
	size_t i, j, k, l;

	for (i = 0; i < sizeof(res.data); i++) {
		res.data[i] = 0;
	}

	for (i = 0; i < sizeof(s->data); i++) {
		for (j = 0; j < sizeof(*s->data) * 8; j++) {
			if (!(s->data[i] & (1 <<
			      (sizeof(*s->data) * 8 - 1 - j)))) {
				continue;
			}
			/* unum u = (so * i + j) is in the set */
			u = sizeof(*s->data) * 8 * i + j;
			u = mod(u);

			k = u / (sizeof(*s->data) * 8);
			l = u % (sizeof(*s->data) * 8);
			res.data[k] |= (1 << (sizeof(*res.data) *
			                      8 - 1 - l));
		}
	}

	/* write result to operand */
	for (i = 0; i < sizeof(s->data); i++) {
		s->data[i] = res.data[i];
	}
}

void
uadd(SORN *a, SORN *b)
{
	_sornop(a, b, addtable, NULL);
}

void
usub(SORN *a, SORN *b)
{
	_sornop(a, b, addtable, _unumnegate);
}

void
umul(SORN *a, SORN *b)
{
	_sornop(a, b, multable, NULL);
}

void
udiv(SORN *a, SORN *b)
{
	_sornop(a, b, multable, _unuminvert);
}

void
uneg(SORN *s)
{
	_sornmod(s, _unumnegate);
}

void
uinv(SORN *s)
{
	_sornmod(s, _unuminvert);
}

void
uabs(SORN *s)
{
	_sornmod(s, _unumabs);
}

void
ulog(SORN *s)
{
	unum u;
	size_t i, j;
	static SORN res;

	/* empty result SORN */
	for (i = 0; i < sizeof(res.data); i++) {
		res.data[i] = 0;
	}

	for (i = 0; i < sizeof(s->data); i++) {
		for (j = 0; j < sizeof(*s->data) * 8; j++) {
			if (!(s->data[i] & (1 << (sizeof(*s->data) * 8 -
			                    1 - j)))) {
				continue;
			}
			/* unum u = (so * i + j) is in the set */
			u = sizeof(*s->data) * 8 * i + j;

			/* is SORN negative? not defined */
			if (u > NUMUNUMS / 2) {
				_sornaddrange(&res, 0, NUMUNUMS - 1);
				goto done;
			}

			/* read the table and apply ranges */
			_sornaddrange(&res, logtable[u].a,
			              logtable[u].b);
		}
	}
done:
	/* write result to operand */
	for (i = 0; i < sizeof(s->data); i++) {
		s->data[i] = res.data[i];
	}
}

void
uemp(SORN *s)
{
	size_t i;

	for (i = 0; i < sizeof(s->data); i++) {
		s->data[i] = 0;
	}
}

void
uset(SORN *a, SORN *b)
{
	size_t i;

	for (i = 0; i < sizeof(a->data); i++) {
		a->data[i] = b->data[i];
	}
}

void
ucut(SORN *a, SORN *b)
{
	size_t i;

	for (i = 0; i < sizeof(a->data); i++) {
		a->data[i] = a->data[i] & b->data[i];
	}
}

void
uuni(SORN *a, SORN *b)
{
	size_t i;

	for (i = 0; i < sizeof(a->data); i++) {
		a->data[i] = a->data[i] | b->data[i];
	}
}

int
uequ(SORN *a, SORN *b)
{
	size_t i;

	for (i = 0; i < sizeof(a->data); i++) {
		if (a->data[i] != b->data[i]) {
			return 0;
		}
	}

	return 1;
}

int
usup(SORN *a, SORN *b)
{
	ucut(a, b);

	return uequ(a, b);
}

void
uint(SORN *s, double lower, double upper)
{
	_sornaddrange(s, blur(lower), blur(upper));
}

void
uout(SORN *s)
{
	unum loopstart, u;
	size_t i, j;
	int active, insorn, loop2run;

	loop2run = 0;
	for (active = 0, i = sizeof(s->data) / 2; i < sizeof(s->data);
	     i++) {
loop1start:
		for (j = 0; j < sizeof(*s->data) * 8; j++) {
			u = sizeof(*s->data) * 8 * i + j;
			insorn = s->data[i] & (1 << (sizeof(*s->data) *
			                       8 - 1 - j));
			if (!active && insorn) {
				/* print the opening of a closed
				 * subset */
				active = 1;
				if (unums[u].name) {
					printf("[%s,", unums[u].name);
				} else {
					printf("(%s,", unums[UCLAMP(u,
					               -1)].name);
				}
			} else if (active && !insorn) {
				/* print the closing of a closed
				 * subset */
				active = 0;
				if (unums[UCLAMP(u, -1)].name) {
					printf("%s] ", unums[UCLAMP(u,
					               -1)].name);
				} else {
					printf("%s) ", unums[u].name);
				}
			}
		}
		if (loop2run) {
			goto loop2end;
		}
	}

	loop2run = 1;
	for (i = 0; i < sizeof(s->data) / 2; i++) {
		goto loop1start;
loop2end:
		;
	}

	if (active) {
		printf("\u221E)");
	}
}
\end{lstlisting}
	\subsection{config.mk}
\begin{lstlisting}[language=make,label=lst:configmk]
UBITS = 12
DIGITS = 2
\end{lstlisting}
	\subsection{Makefile}
\begin{lstlisting}[language=make]
include config.mk

all: libunum.a

libunum.a: table.o unum.o
	ar rcs libunum.a table.o unum.o

unum.o: unum.c
	cc -c unum.c -lm

table.o: table.c
	cc -c table.c

table.c: gen
	./gen 2> unum.h 1> table.c

gen: gen.c config.mk
	cc -o gen -DUBITS=${UBITS} -DDIGITS=${DIGITS} gen.c -lm

%: %.c libunum.a
	cc $^ -o $@

clean:
	rm -f gen table.c unum.h table.o unum.o libunum.a
\end{lstlisting}
	\section{Unum Problems}\label{sec:unumproblems}
	These programs expect libunum.a and unum.h in the current directory
	at compile time. It is recommended to create symbolic links to the
	toolbox directory given both are generated dynamically there and
	thus subject to change.
	\par
	The environment parametres for the decade lattice are set in
	\texttt{config.mk} (see Listing~\ref{lst:configmk}).
	\subsection{euler.c}
\begin{lstlisting}[language=C,label=lst:euler]
#include <stdio.h>

#include "unum.h"

void
factorial(SORN *s, int f)
{
	SORN tmp;
	int i;

	uemp(s);
	uint(s, 1, 1);

	for (i = f; i > 1; i--) {
		uemp(&tmp);
		uint(&tmp, i, i);
		umul(s, &tmp);
	}
}

int
main(void)
{
	SORN e, tmp;
	int i;

	uemp(&e);
	uint(&e, 1, 1);
	uout(&e);
	putchar('\n');

	for (i = 1; i <= 20; i++) {
		factorial(&tmp, i);
		uinv(&tmp);
		uadd(&e, &tmp);
		uout(&e);
		putchar('\n');
	}

	return 0;
}
\end{lstlisting}
	\subsection{devil.c}
\begin{lstlisting}[language=C,label=lst:devilc-unum]
#include <stdio.h>

#include "unum.h"

int
main(void)
{
	SORN a, b, c, tmp1, tmp2, tmp3;
	int n;

	uemp(&a);
	uemp(&b);
	uemp(&c);

	uint(&a, 2, 2);
	uint(&b, -4, -4);

	for (n = 2; n <= 25; n++) {
		if (n > 2) {
			uset(&a, &b);
			uset(&b, &c);
		}

		uemp(&tmp1);
		uint(&tmp1, 111, 111);

		uemp(&tmp2);
		uint(&tmp2, 1130, 1130);
		udiv(&tmp2, &b);
		usub(&tmp1, &tmp2);

		uemp(&tmp2);
		uint(&tmp2, 3000, 3000);
		uset(&tmp3, &b);
		umul(&tmp3, &a);

		udiv(&tmp2, &tmp3);
		uadd(&tmp1, &tmp2);

		uset(&c, &tmp1);

		printf("U_%d = ", n);
		uout(&c);
		putchar('\n');
	}

	return 0;
}
\end{lstlisting}
	\subsection{bank.c}
\begin{lstlisting}[language=C,label=lst:bankc-unum]
#include <math.h>
#include <stdio.h>

#include "unum.h"

int
main(void)
{
	SORN a, tmp;
	int y;

	uemp(&a);
	uint(&a, M_E - 1, M_E - 1);

	for (y = 1; y <= 25; y++) {
		uemp(&tmp);
		uint(&tmp, y, y);
		umul(&a, &tmp);

		uemp(&tmp);
		uint(&tmp, 1, 1);
		usub(&a, &tmp);

		printf("year %2d: ", y);
		uout(&a);
		putchar('\n');
	}

	return 0;
}
\end{lstlisting}
	\subsection{spike.c}
\begin{lstlisting}[language=C,label=lst:spikec-unum]
#include <stdio.h>

#include "unum.h"

#define NUMPOINTS (10)
#define POLE (4.0/3.0)

int
main(void)
{
	SORN res, tmp;
	size_t i, j;
	unum pole, u;

	/* get the unum containing the POLE */
	uemp(&res);
	uint(&res, POLE, POLE);
	for (i = 0; i < sizeof(res.data); i++) {
		if (res.data[i]) {
			for (j = 0; j < sizeof(res.data[i]) * 8; j++) {
				if ((1 << (8 - j - 1)) & res.data[i]) {
					break;
				}
			}
			pole = 8 * i + j;
			break;
		}
	}

	for (u = UCLAMP(pole, -NUMPOINTS); u <= UCLAMP(pole,
	     +NUMPOINTS); u++) {
	     	/* fill res just with the single u */
		uemp(&res);
		res.data[u / 8] = 1 << (8 - (u % 8) - 1);
		uout(&res);
		printf(" |-> ");

		/* calculate F(res) */
		uneg(&res);

		uemp(&tmp);
		uint(&tmp, 1, 1);
		uadd(&res, &tmp);

		uemp(&tmp);
		uint(&tmp, 3, 3);
		umul(&res, &tmp);

		uemp(&tmp);
		uint(&tmp, 1, 1);
		uadd(&res, &tmp);

		uabs(&res);
		ulog(&res);
		uout(&res);
		putchar('\n');
	}

	return 0;
}
\end{lstlisting}
	\subsection{Makefile}
\begin{lstlisting}[language=make]
PROBLEMS = euler devil bank spike

all: $(PROBLEMS)

%: %.c
	cc -o $@ $^ libunum.a

clean:
	rm -rf $(PROBLEMS)
\end{lstlisting}
	\section{License}
	This ISC license applies to all code listings in
	Chapter~\ref{ch:code}.
\begin{lstlisting}[label=lst:license]
Copyright © 2016, Laslo Hunhold

Permission to use, copy, modify, and/or distribute this software for any
purpose with or without fee is hereby granted, provided that the above
copyright notice and this permission notice appear in all copies.

THE SOFTWARE IS PROVIDED "AS IS" AND THE AUTHOR DISCLAIMS ALL WARRANTIES
WITH REGARD TO THIS SOFTWARE INCLUDING ALL IMPLIED WARRANTIES OF
MERCHANTABILITY AND FITNESS. IN NO EVENT SHALL THE AUTHOR BE LIABLE FOR
ANY SPECIAL, DIRECT, INDIRECT, OR CONSEQUENTIAL DAMAGES OR ANY DAMAGES
WHATSOEVER RESULTING FROM LOSS OF USE, DATA OR PROFITS, WHETHER IN AN
ACTION OF CONTRACT, NEGLIGENCE OR OTHER TORTIOUS ACTION, ARISING OUT OF
OR IN CONNECTION WITH THE USE OR PERFORMANCE OF THIS SOFTWARE.
\end{lstlisting}
\end{appendix}
\backmatter
\nocite{*}
\bibliography{ba-laslo_hunhold}
\selectlanguage{ngerman}
\chapter{Eigenständigkeitserklärung}
Hiermit bestätige ich, daß ich die vorliegende Arbeit selbstständig
verfaßt und keine anderen als die angegebenen Hilfsmittel verwendet habe.
\newline
Die Stellen der Arbeit, die dem Wortlaut oder dem Sinn nach anderen Werken
entnommen sind, wurden unter Angabe der Quelle kenntlich gemacht.
\\[3cm]
Laslo Hunhold
\end{document}